\documentclass{article}
\usepackage{amsmath}
\usepackage{amssymb}
\usepackage{amsthm}
\usepackage{authblk}
\usepackage{color}
\usepackage{bm}
\usepackage{stmaryrd}
\usepackage{url}
\usepackage{cleveref}
\usepackage{enumerate}
\numberwithin{equation}{section}
\usepackage{graphicx,color,overpic,rotating}
\usepackage[a4paper,left=27mm,bottom=40mm,right=27mm,top=27mm]{geometry}
\usepackage{yfonts}

\usepackage{vem_macros}
\renewcommand{\d}{\operatorname{d}\!}

\newcommand{\bilBP}[2]{\mathcal{B}_{\E}(#1,#2)}

\newcommand{\matPiBkP}{{\bm\Pi}^{\mathcal{B}}_{\k}}%,k}_{\E}}
\newcommand{\matB}{\mathsf{B}}

\newcommand{\bigbasis}{\psi} % basis for the big space, used in implementation section.

\theoremstyle{definition}
\newtheorem{assumption}{Assumption}
\newenvironment{assump}[2][]
  {\renewcommand\theassumption{#2}\begin{assumption}[#1]}
  {\end{assumption}}

\title{Conforming and nonconforming virtual element methods for elliptic problems}
\author[1]{Andrea Cangiani}
\author[2]{Gianmarco Manzini}
\author[1]{Oliver J. Sutton}
\affil[1]{Department of Mathematics, University of Leicester, University Road, Leicester LE1 7RH, United Kingdom.}
\affil[2]{IMATI-CNR, via Ferrata 1, 27100 Pavia; T5 Group, Theoretical Division, Los Alamos National Laboratory, Los Alamos 87545, New Mexico, USA.}
\date{}

\begin{document}

\maketitle

%$^1$Department of Mathematics, University of Leicester, University Road, Leicester LE1 7RH, United Kingdom.
%
%$^2$ IMATI-CNR, via Ferrata 1, 27100 Pavia; T5 Group, Theoretical Division, Los Alamos National Laboratory,
%Los Alamos 87545, New Mexico, USA.

\begin{abstract}
We present in a unified framework new conforming and nonconforming Virtual Element Methods (VEM) for general second order elliptic problems in two and three dimensions.
The differential operator is split into its symmetric and non-symmetric parts and conditions for stability and accuracy on their discrete counterparts are established.
These conditions are shown to lead to optimal $H^1$- and $\LTWO$-error estimates, confirmed by numerical experiments on a set of polygonal meshes. The accuracy of the numerical approximation provided by the two methods is shown to be comparable.
\end{abstract}

\section{Introduction}

The Virtual Element Method (VEM) was introduced in~\cite{BasicsPaper}
as a generalisation of the conforming finite element method (FEM),
offering great flexibility in utilising meshes with (almost) arbitrary
polygonal elements.
Unlike the polygonal finite element method (PFEM)~\cite{PFEM} and other conforming FEM extensions based 
on the Generalised Finite Element~\cite{GFEM} framework such as the CFE method~\cite{hackbusch_sauter_cfe_nm} and the  XFEM~\cite{XFEM}, 
the VEM handles meshes with general shaped elements in a manner that avoids the explicit evaluation of the shape functions.
Indeed, the VEM only requires the knowledge of a polynomial subspace of the local finite element space to provide stable and accurate numerical methods.
This feat is achieved by separating the contributions of the polynomial subspace from that of the remaining non-polynomial \emph{virtual} subspace through the introduction of suitable projection operators which can be computed just using the VEM degrees of freedom. The polynomial \emph{consistency} terms of the bilinear form, responsible for convergence properties of the method, are computed accurately. The remaining terms are only required to ensure the \emph{stability} of the method, and hence they can be   rougly estimated from the degrees of freedom. 
The VEM approach can also be viewed as a variational analogue of the mimetic
finite difference (MFD) method; see~\cite{MFDBook} and the recent
review paper~\cite{MFDReviewJCP}. 
As such, for its analysis we can take
advantage of the standard tools of finite element analysis.

An alternative approach
is to completely relax the inter-element conformity
requirements for the discrete space, so that simple (polynomial)
spaces can be used. For instance, discontinuous Galerkin and weak
Galerkin methods, whereby inter-element continuity is weakly imposed,
are naturally suited to general meshes;
see~\cite{Cangiani-Georgoulis-Houston:2014,Cockburn-Qiu-Solano:2014,Mu-Lin-Wang-YE:2015} and the references therein.

Here we shall stop just short of that, presenting a general VEM
framework which \emph{is} based on relaying on some form of
continuity, in the spirit of~\cite{CrouzeixRaviart}. The framework is
used to introduce  a $C^0$-conforming and a nonconforming VEM.

The original VEM in~\cite{BasicsPaper} is a $C^0$-conforming method for
solving the two-dimensional Poisson equation and the same problem is
considered in~\cite{NonconformingVEM}, where a nonconforming
formulation is presented.
The extension of these methods to general elliptic problems with variable coefficients in two and three dimensions, 
is non trivial.
Diffusion problems with non-constant diffusion tensors in two
dimensions are treated in~\cite{ArbitraryRegularityVEM}, where VEMs
which incorporate inter-element continuity of arbitrary degree are
presented.
A crucial step towards the inclusion of low-order differential terms is provided in~\cite{EquivalentProjectors} with the extension of the original $C^0$-conforming VEM to reaction-diffusion problems with constant coefficients in two and three dimensions.
This approach is extended to the solution of general elliptic problems in two dimensions in~\cite{General}.
Concurrently, the VEM framework has been extended to the solution of
plate-bending problems~\cite{PlateBendingVEM}, linear elasticity
problems in two and three spatial dimensions~\cite{LinearElasticity2D,
  LinearElasticity3D}, the Steklov eigenvalue
problem~\cite{SteklovVEM}, the simulation of discrete fracture
networks~\cite{DiscreteFractureVEM}, and the two-dimensional
streamline formulation of the Stokes problem~\cite{StokesVEM}.

We present here a conforming and a nonconforming VEM for the numerical treatment of general linear elliptic problems with variable coefficients in two and three spatial dimensions.
The accuracy and stability of the two methods is determined  in~\sref{sec:vemFramework} through a unified abstract framework, cf. Assumption~\ref{ass:vem}.
Here the partial differential operator is split into its symmetric and skew-symmetric parts and the VEM polynomial consistency and stability properties are established for each of these components separately, cf. Assumption~\ref{ass:bilinearForms}. 
This approach is quite natural in that, for instance, it is clear that only the symmetric component is needed for the method's stability. Indeed a unique stabilisation for all the terms that contribute to the symmetric part (the diffusion, reaction, and symmetric contribution of the convection term) is introduced.
The stabilisation automatically adjusts
with the relative magnitude of the  (symmetric) terms. 
It also leads  the way to the design of ad hoc stabilisation techniques  for the pre-asymptotically stable solution of convection-dominated problems, such as the classical streamline diffusion method~\cite{Hughes-Brooks:1979}, although this is not considered here.

To deal with non-constant coefficients and the lower-order terms, we take the approach of~\cite{General}, rather than that of~\cite{ArbitraryRegularityVEM}.
A crucial role in the former formulation is played by the $\LTWO$-projector which maps the functions of the virtual element space and their gradients onto polynomials.
In order to have the the $\LTWO$-projection operator computable by using only the degrees of freedom, 
in~\sref{sec:vemSpaces}  we generalise a procedure introduced in~\cite{EquivalentProjectors}, dubbed VEM enhancement, used here in the context of nonconforming VEM for the first time.
In this way  a family of virtual element spaces is defined from which a particularly simple choice can be made, cf. Section~\ref{subsec:imple}.
This approach differs completely from that presented for a non-constant diffusion tensor in~\cite{ArbitraryRegularityVEM}, which required the construction of a bespoke projection operator dependent on the diffusion tensor.
The key advantage of removing this dependence is that lower order terms can be dealt with in an identical manner. 
Furthermore, we are now easily able to analyse the impact on the method of the approximation of the problem's coefficients. Here it is important to stress that such approximation is \emph{only} needed to compute the integrals involved in the polynomial consistency terms of the bilinear form. 
This fact is discussed in~\sref{sec:bilinearForms}, with the conclusion
that the stability and optimal accuracy of the method based on using polynomials of order up to  $k$ are unaffected
by the use of a quadrature scheme to approximate the consistency terms, provided that this is of at least degree $2\k-2$.
We stress that this is \emph{exactly the same requirement of the
  finite element methods}~\cite{Ciarlet}.

The new unified formulation offers some indisputable advantages.
From a theoretical viewpoint, it permits us to analyse in a unified
manner the conforming and nonconforming
VEM following the standard analyses of finite element methods for
elliptic problems. 
The analysis, detailed in \sref{sec:h1ErrorBound},  ultimately leads to optimal order $H^1$- and $\LTWO$-error estimates for both methods under the same regularity assumptions on the mesh and the exact solution.
By contrast, the analysis of conforming VEMs for the same problem in two space dimensions given in~\cite{General} is based on an inf-sup argument relaying on the mesh size being small enough. 
From a practical viewpoint, the implementation of the conforming and
nonconforming VEM is formally the same (see Section~\ref{subsec:imple}).
In fact, the only difference is in the construction of the $\LTWO$
projection operator for the shape functions and their gradients.
Such construction depends on the degrees of freedom, which necessarily
differ for the conforming and nonconforming VEM.
Also, as mentioned above, we foresee that the present unified framework will facilitate the treatment of convection-dominated diffusion problems, as well as the design of VEM for the Stokes system.

The conforming and nonconforming VEMs are assessed in \sref{sec:numerics} solving numerically a
representative convection-reaction-diffusion problem with variable coefficients in two dimensions.
The accuracy of the numerical approximation provided by the two methods is comparable and confirms the optimal convergence rates in the $\LTWO$- and $H^1$- norm established by  the theoretical analysis presented in~\sref{sec:h1ErrorBound}.
Finally, in~\sref{sec:conclusion} we offer our final conclusions.

\section{The Continuous Problem}
\label{sec:contProb}
Consider the boundary value problem
\begin{subequations}
\label{eq:pde}
\begin{align}
  -\nabla\cdot(\diff(\vx)\nabla \u) + \conv(\vx) \cdot \nabla \u + \reac(\vx) \u &= \force(\vx)\phantom{0} \quad\text{in}~\D, \\
  \u &= 0\phantom{f(\vx)}\quad\text{on}~\dD,
\end{align}
\end{subequations}
where $\D \subset \Re^\spacedim$ is a polygonal domain for $\spacedim=2$ and a polyhedral domain for $\spacedim=3$.
We assume that the coefficients $\diff_{i,j}(\vx), \conv_{i}(\vx), \reac(\vx)$ are in $L^{\infty}(\D)$, and $\force \in L^{\infty}(\D)$ is the forcing function.
We further suppose that $\diff(\vx)$ is a full symmetric $\spacedim\times \spacedim$ diffusivity tensor and is strongly elliptic, i.e. there exist $\elipLower,\elipUpper > 0$, independent of $\vec{v}$ and $\vx$, such that
\begin{equation}
	\elipLower \abs{\vec{v}(\vx)}^2 \leq \vec{v}(\vx) \cdot \diff(\vx) \vec{v}(\vx) \leq \elipUpper \abs{\vec{v}(\vx)}^2,
	\label{eq:diffEllipticity}
\end{equation}
for almost every $\vx \in \D$ and for any $\vec{v} \in (H^1_0(\Omega))^{\spacedim}$, where $\abs{\cdot}$ denotes the standard Euclidean norm on $\Re^\spacedim$.
Finally, we suppose that there exists $\reacLower > 0$ such that
\begin{equation}
	\reacSym(\vx) := \reac(\vx) - \frac{1}{2} \nabla \cdot \conv(\vx) \geq \reacLower \ge 0,
	\label{eq:reacBound}
\end{equation}
for almost every $\vx \in \D$, and assume that $\nabla\cdot\conv\in L^{\infty}(\D)$.

The variational form of problem~\eqref{eq:pde} reads: find $\u \in H^1_0(\D)$ such that
\begin{equation}
	(\diff \nabla \u, \nabla \v) + (\conv \cdot \nabla \u, \v) + (\reac \u, \v) = (\force, \v) \quad \quad \forall \v \in H^1_0(\D),
	\label{eq:origVariationalForm}
\end{equation}
with $(\cdot,\cdot)$ denoting the $\LTWO$ inner product on $\D$. 
We  split the bilinear form on the left-hand side of~\eqref{eq:origVariationalForm} into its symmetric and skew-symmetric parts:
\begin{subequations}
  \begin{align}
    \a(\u,\v) &:= (\diff \nabla \u, \nabla \v)+\left(\reacSym \u, \v\right), \label{eq:a:def}\\
    \b(\u,\v) &:= \frac{1}{2} \left[ \left(\conv \cdot \nabla \u, \v\right) - \left(\u, \conv \cdot \nabla \v\right) \right],\label{eq:b:def}
  \end{align}
\end{subequations}
and consider discretising the problem written in the equivalent form: find $\u \in H^1_0(\D)$ such that
\begin{equation}
  A(u,v) :=\,  \a(\u,\v) + \b(\u,\v) = (\force, \v) \quad \quad \forall \v \in H^1_0(\D).
  \label{eq:newVariationalForm}
\end{equation}
Rewriting the variational form in this way would not be necessary for the classical finite element method, but it turns out to be a useful step for the Virtual Element Method in view of ensuring that the discrete framework preserves the information about the symmetric and skew-symmetric parts of the bilinear form.

It is simple to check that the bilinear form $A$ is coercive and bounded, and the variational problem therefore possesses a unique solution by the Lax-Milgram lemma.

\section{The Virtual Element Framework}
\label{sec:vemFramework}

We assume that a Virtual Element Method (VEM) consists of the following fundamental ingredients:
\begin{assump}{A1}
  \label{ass:vem}
  For any fixed $h>0$ and $\k\in\mathbb{N}$, we have:
  \begin{itemize} 
    \item A finite decomposition (mesh) $\{\Th\}$ of the domain
      $\Omega$ into non-overlapping \emph{simple  polygonal/ polyhedral elements} with maximum size $h$. The adjective simple refers to the fact that the boundary of each element in the decomposition must be non-intersecting. Further, the boundary of any element $\E\in\Th$ is made of a uniformly bounded number of interfaces (edges/faces) which are either part of the boundary of $\Omega$ or shared with another element in the decomposition.
  \item A finite dimensional function space $\Vh\subset H^1(\Th)$ where  
\begin{equation}\label{eq:brokenH1}
H^1(\Th):=\{v\in \LTWO(\D)\, :\, v_{|\E}\in H^1(\E), \forall \E\in\Th\}
\end{equation} 
(not necessarily a subspace of $H^1_0(\D)$) to be used as a trial and test space, on which the following Poincar\'e-Friedrichs inequality for piecewise $H^1$ functions holds:
 \begin{equation}\label{eq:PF}
\norm{\vh}_{0,\D}^2\le \PFconst |\vh|_{1,h}^2=: \sum_{\E \in \Th} \norm{\nabla\vh}_{0,\E}^2,
   \end{equation} 
hence $|\cdot|_{1,h}$ is a norm on $\Vh$.
  
	Further, for each element $\E \in \Th$, the space $\VhE
    := \Vh|_{\E}$ must contain the space $\PE{\k}$ of polynomials of
    degree $\k$ on $\E$;
    
  \item A bilinear form $\Ah: \Vh \times \Vh
    \rightarrow \Re$, which may be split over the elements in the
    mesh $\Th$ as
    \begin{equation*}
      \Ah(\uh, \vh) = \sum_{\E \in \Th} \AhE(\uh, \vh),
    \end{equation*}
    for any $\uh, \vh \in \Vh$, where $\AhE$ is a
    bilinear form over the space $\VhE$.
      \item An element $\forceh \in \Vh'$ approximating the forcing term. 
  \end{itemize}
\end{assump}

\begin{remark}
The definition of \emph{simple polygons} and \emph{simple polyhedra}  is general enough to include, for instance, elements with consecutive co-planar edges/faces, such as those typical of locally refined meshes with hanging nodes and non-convex elements.
Later on, in Assumption~(A3) in~\sref{subsec:approximation-properties}, we shall introduce some standard  mesh regularity assumptions which are required for the approximation properties of the virtual element spaces.
\end{remark}

In view of the following analysis, it is useful to extend the definition of the continuous bilinear form $\A$, as well as of its parts $\a$ and $\b$, to the whole of $H^1(\Th)$ as a sum of elemental contributions.
Henceforth, we shall use
\begin{equation*}
 \A(\u,\v) :=\sum_{\E\in\Th} \AE(\u,\v)=\sum_{\E\in\Th} \,  \aE(\u,\v) + \bE(\u,\v)\qquad \forall \u,\v\in H^1(\Th). 
\end{equation*}

We also place a few more restrictions on the nature of the bilinear forms $\AhE$. As with the continuous bilinear form, we write  $\AhE$ as the sum of a symmetric and a skew-symmetric part:
\begin{equation*}
	\AhE(\uh, \vh) = \ahE(\uh, \vh) + \bhE(\uh, \vh),
\end{equation*}
and require that they satisfy the following properties:
\begin{assump}{A2}
	\label{ass:bilinearForms}
	The bilinear forms $\ahE$ and $ \bhE$ are assumed to satisfy the properties of \emph{polynomial consistency} and \emph{stability}, defined as
	\begin{itemize}
		\item \emph{Polynomial consistency}:
			If either $\uh \in \PE{\k}$ or $\vh \in \PE{\k}$, the symmetric and skew-symmetric parts of the local virtual element bilinear form must satisfy
			\begin{align*}
				\ahE(\uh, \vh) &= \intE \diff \Po{\k-1} (\nabla \uh)\cdot \Po{\k-1} (\nabla \vh) \dx+\intE \reacSym \Po{\k} \uh \Po{\k} \vh \dx, \\
				\bhE(\uh, \vh) &= \frac{1}{2}\intE \conv \cdot \left[\Po{\k-1} (\nabla \uh) \Po{\k} \vh - \Po{\k} \uh \Po{\k-1} (\nabla \vh)\right] \dx,
			\end{align*}
                        where the operator $\Po{\ell} : \LTWO(\E) \rightarrow \PE{\ell}$ for $\ell \leq \k$ denotes the $\LTWO(\E)$-orthogonal projection onto the polynomial 
                        space $\PE{\ell}$, and is defined for any function $v\in \LTWO(\E)$ as the unique element $\Po{\ell}v$ of $\PE{\ell}$ such that
			\begin{equation*}
				%\label{eq:l2ProjDefn}
				(\Po{\ell} \v, p)_{\E} = (\v, p)_{\E} \quad \forall p \in \PE{\ell}.
			\end{equation*}
		\item \emph{Stability}:
			There exist positive constants $\diffStabLower, \diffStabUpper$, and $\convContinuityConst$  independent of $h$ and the mesh element $\E$ such that, for all $\vh,\wh \in \VhE$, the symmetric part satisfies
			\begin{equation*}
				\diffStabLower \aE(\vh, \vh) \leq \, %&
				\ahE(\vh, \vh) \leq \diffStabUpper \aE(\vh, \vh), %\\
				\end{equation*}
				and the skew-symmetric part satisfies
				\begin{equation*}
				\bhE(\vh, \vh) = 0 \quad \text{ and } %&
				\quad \bhE(\vh, \wh) \leq \convContinuityConst \norm{\vh}_{1,\E} \norm{\wh}_{1,\E}.
			\end{equation*}

	\end{itemize}
\end{assump}

Precise methods of choosing the spaces and bilinear forms of the method are, of course, the focus of most of the remainder of this paper.
However, the simple properties presented above are enough to prove two crucial facts about the behaviour of such a Virtual Element Method: firstly that any such method possesses a unique solution and secondly an abstract Strang-type convergence result, which will be used later on to derive optimal order error bounds in the $H^1$  norm.
These results are encapsulated in the following theorems.

\begin{theorem}[Existence and uniqueness of a virtual element solution]\label{thm:vem-wellposed}
  Under Assumptions~\ref{ass:vem} and \ref{ass:bilinearForms}, the problem: find $\uh \in \Vh$ such that
  \begin{equation}\label{eq:VEMproblem}
    %\Ah(\uh, \vh) := 
	\ah(\uh, \vh) + \bh(\uh, \vh)% + \ch(\uh, \vh) 
    = \langle \forceh, \vh\rangle \quad \forall \vh \in \Vh,
  \end{equation}
  possesses a unique solution.
Here $\langle \cdot, \cdot\rangle$ denotes the duality  pairing between $\Vh$ and its dual $\Vh'$.
\end{theorem}
\begin{proof}
The stability condition on $\ahE$ ensures that this
bilinear form inherits the coercivity of its counterpart
$\aE(\cdot,\cdot)$, hence
\begin{equation}\label{eq:local-coerc}
  \ahE(\vh,\vh)
  \geq \diffStabLower 
    \big(\elipLower \norm{\nabla\vh}_{0,\E}^2+\reacLower\norm{\vh}_{0,\E}^2\big).
\end{equation}
Summing up the contribution of all elements and using~\eqref{eq:PF} we deduce the coercivity bound
\begin{equation}\label{eq:VEMcoerc}
  \ah(\vh,\vh)
  \geq \frac{\diffStabLower}{1+C_{{\rm PF}}} 
    \big(\elipLower |\vh|_{1,h}^2+(\elipLower+\reacLower)\norm{\vh}_{\D}^2\big)\geq \frac{\diffStabLower}{1+C_{{\rm PF}}} \elipLower \norm{\vh}_{1,h}^2.%=:\coerc\norm{\vh}_{1,h}^2.
\end{equation}
From inequality~\eqref{eq:local-coerc} and the symmetry and bilinearity of
$\ahE$, it follows that $\ahE$ is an inner
product on $\VhE$.
Hence, we can apply the Cauchy-Schwartz inequality and use the
  right stability inequality of Assumption~A2 to prove the continuity
  of $\ahE$:
\begin{align*}
  \ahE(\uh,\vh) 
  &\leq (\ahE(\uh,\uh))^\frac{1}{2}(\ahE(\vh,\vh))^\frac{1}{2} \leq \diffStabUpper (\aE(\uh,\uh))^\frac{1}{2}(\aE(\vh,\vh))^\frac{1}{2} \\
  &\leq \diffStabUpper\max\{\elipUpper,\norm{\reacSym}_{\infty}\} \norm{\uh}_{1,\E} \norm{\vh}_{1,\E},
%    = \gamma\norm{\uh}_{1,\E}\norm{\vh}_{1,\E}
\end{align*}
and the continuity of $\ah$ easily follows.

The stability property for $\bhE$ means that this term
does not feature in the coercivity analysis and imposes continuity with constant $\convContinuityConst$.	
Hence we may conclude that the problem~\eqref{eq:VEMproblem} admits a unique solution by
the Lax-Milgram lemma.
\end{proof}

\begin{theorem}[Abstract a priori error bound]
	\label{thm:vemStrang}
	Under Assumptions~\ref{ass:vem} and~\ref{ass:bilinearForms}, there exists a constant $C>0$ depending only on the coercivity and the continuity constants such that
	\begin{align}
          \norm{\u - \uh}_{1,h} \leq 
          &C \left( \inf_{\vh \in \Vh} \norm{\u - \vh}_{1,h} + \sup_{\substack{\wh \in \Vh \\ \wh \neq 0}} \frac{\abs{(\forceh, \wh) - (\force, \wh)}}{\norm{\wh}_{1,h}} + \right. \nonumber\\
          & + \inf_{p \in \P{\Th}{\k}} \left[\norm{\u - p}_{1, h} + \sum_{\E \in \Th} \sup_{\substack{\wh \in \VhE \\ \wh \neq 0}} \frac{ \abs{\AE(p, \wh) - \AhE(p, \wh)}}{\norm{\wh}_{1,\E}}\right] + \nonumber\\
          &\left. + \sup_{\substack{\wh \in \Vh \\ \wh \neq 0}} \frac{\abs{
%\sum_{\E \in \Th}\AE(\u, \wh)
\A(\u, \wh)
- \langle\force, \wh\rangle}}{\norm{\wh}_{1,h}} \right),
          \label{eq:vemStrang}
	\end{align}
	where 
	\begin{equation*}
		\P{\Th}{\k} := \{ p \in \LTWO(\D) : p_{|_\E} \in \PE{\k} \quad \forall \E \in \Th \}.
	\end{equation*}
        The last term in the right-hand side of the above error estimate measures the nonconformity error, i.e. it is non-zero only when $\Vh$ is a non-conforming virtual space.
\end{theorem}
\begin{proof}
	Let $\vh$ be an arbitrary element of $\Vh$ and let $\wh = \uh - \vh \in \Vh$. Then, the coercivity of $\Ah$ implies that
	\begin{align*}
		\coerc \norm{\uh - \vh}_{1,h}^2 &\leq \Ah(\uh - \vh, \wh)
                        = \langle\forceh, \wh\rangle + \sum_{\E \in \Th} \left(\AhE(p - \vh, \wh) - \AhE(p, \wh)\right)\\
  &= \langle\forceh, \wh\rangle + \sum_{\E \in \Th} \left[\AhE(p - \vh , \wh) - \AE(p, \wh) 
            + \left(\AE(p, \wh) - \AhE(p, \wh)\right)\right] \\
            &= \langle\forceh, \wh\rangle - \A(u, \wh) + \sum_{\E \in \Th} \left[\AhE(p - \vh, \wh) + \AE(u - p, \wh) \right. \\
            &
            \quad\left. + \big( \AE(p, \wh) - \AhE(p, \wh) \big) \right],
          \end{align*}
for any $p \in \P{\Th}{\k}$.
We express the potential nonconformity of the virtual element space $\Vh$ as
	\begin{equation*}
		\A(u, \wh) = \langle\force, \wh\rangle + \big( \A(u, \wh) - (\force, \wh) \big),
	\end{equation*}
	so that
	\begin{align*}
          \coerc \norm{\uh - \vh}_{1,h}^2 \leq 
          & \big[\langle\forceh, \wh\rangle - (\force, \wh)\big] + 
           \left[(\force, \wh) - 
%\marco{\sum_{\E \in \Th}\AE(u, \wh)}\right]\\
\A(u, \wh)\right]\\
          & +\sum_{\E \in \Th} \big[\AhE(p - \vh, \wh) + \AE(u - p, \wh)
            + \big( \AE(p, \wh) - \AhE(p, \wh) \big) \big].
	\end{align*}
Hence, for all $\wh\in\Vh\setminus\{ 0\}$ and $p \in \P{\Th}{\k}$, applying the continuity of the bilinear forms, multiplying and dividing by
    $\norm{\wh}_{1,h}$, and using the triangle inequality we find that
	\begin{align*}
		\coerc \norm{\uh - \vh}_{1,h} &\leq  \frac{\abs{\langle\forceh, \wh\rangle - (\force, \wh)}}{\norm{\wh}_{1,h}} + 
                                               \frac{\abs{(\force, \wh) - 
%\marco{\sum_{\E \in \Th}\AE(u, \wh)} 
\A(u, \wh)
                                                 }}{\norm{\wh}_{1,h}}\\
					& \quad +\norm{u-\vh}_{1,h}+ 2\norm{u-p}_{1,h} + \sum_{\E \in \Th} \frac{\abs{\AE(p, \wh) - \AhE(p, \wh)}}{\norm{\wh}_{1,\E}}.
%					\\
%				&\leq \norm{u-\vh}_{1,h} + \sup_{\wh \in \Vh}\frac{\abs{(f_h, \wh) - (f, \wh)}}{\norm{\wh}_{1,h}}  + \sup_{\wh \in \Vh}\frac{\abs{(f, \wh) - \A(u, \wh)}}{\norm{\wh}_{1,h}}\\
%					& \quad + 2\norm{u-p}_{1,h} + \sum_{\E \in \Th} \sup_{\wh \in \VhE}\frac{\abs{\AE(p, \wh) - \AhE(p, \wh)}}{\norm{\wh}_{1,\E}}.
	\end{align*}
The result now follows 
by the triangle inequality.
\end{proof}

\section{The Virtual Element Spaces}
\label{sec:vemSpaces}

We introduce two types of spaces in two- and three-dimensions implementing the framework of Section~\ref{sec:vemFramework}: a conforming and a nonconforming virtual element space.
With deliberate ambiguity we refer to both spaces as $\Vh$ to emphasise the fact that the method is otherwise the same in either case.
In all cases the local virtual element space $\VhE$ 
must contain the space $\PE{\k}$ of polynomials of degree up to $\k$ on $\E$, cf. Assumption~\ref{ass:vem}. 
The complement of $\PE{\k}$ in $\VhE$ is made up of functions which are deemed expensive to evaluate, although they may be described through a set of (known) degrees of freedom.
Central to the virtual element methodology is the idea of \emph{computability}, defined as follows:
\begin{definition}
	\label{def:computable}
	A term will be called \emph{computable} if it may be evaluated using just the degrees of freedom of a function and the polynomial component of the virtual element space.
\end{definition}

We require that the polynomial consistency forms of Assumption~\ref{ass:bilinearForms} are computable. Hence the spaces must be constructed in such a way that the projections
\begin{equation}
	\label{eq:projectionsToCompute}
	\Po{\k-1} \nabla \vh \quad \text{and} \quad \Po{\k} \vh,
\end{equation}
are computable for any function $\vh \in \VhE$.
While the first of these is computable for functions in the original conforming (for $\spacedim = 2$) and nonconforming virtual element spaces presented in~\cite{BasicsPaper, NonconformingVEM}, the second is not.
However, a modification to the conforming space was introduced in~\cite{EquivalentProjectors} which renders $\Po{\k}\vh$ computable without changing the degrees of freedom used to describe the space.
For the construction of the spaces which follow, we consider a generalisation of the process presented in~\cite{EquivalentProjectors}.

We first introduce an appropriately scaled basis for $\PE{\k}$, $\k\in\mathbb{N}$.
Denote by $\MEstar{\ell}$, $\ell\in\mathbb{N}$, the set of \emph{scaled monomials}
\begin{equation*}
	\MEstar{\ell} := \left\{\left(\frac{\vx - \vx_{\E}}{h_{\E}}\right)^{\boldsymbol{s}}, \abs{\boldsymbol{s}} = \ell \right\} ,
\end{equation*}
where $\boldsymbol{s}$ is a multi-index with $\abs{\boldsymbol{s}} := s_1 + s_2$ and $\vx^{\boldsymbol{s}} := x_1^{s_1}x_2^{s_2}$.
Further, we define
$\ME{\k} := \bigcup_{l \leq \k} \MEstar{l}=:\{\m_\alpha\}_{\alpha=1}^{N_{\spacedim,\k}},$
 a basis of $\PE{\k}$, where $N_{\spacedim,\k} := \dim(\P{\Re^\spacedim}{\k})$.

Below,  $\s$ denotes a $\spacedim-1$ dimensional mesh interface (either an edge when $\spacedim = 2$ or a face when $\spacedim = 3$) of the mesh element $\E$, and the set of all mesh interfaces in $\Th$ will be denoted by $\Edges$.
This set is divided into the set of boundary edges $\BoundaryEdges := \{ \s \in \Edges : \s \subset \dD \}$ and internal edges $\InternalEdges := \Edges \setminus \BoundaryEdges$.
Bases for polynomial spaces defined on a interface $\s$ can be similarly constructed; the same notation will be used. 
We denote by $\nuE$ the number of interfaces $\s\in\dE$.

\subsection{The Local Spaces}

The construction of the local nonconforming virtual element space is formally the same for $\spacedim=2$ or 3, while the construction of the conforming space is hierarchical in the space dimension.
Because of this, while we simultaneously provide here a definition of the nonconforming space for $\spacedim = 2$ or 3, we initially only consider the conforming space for $\spacedim = 2$, postponing the construction for $\spacedim = 3$ until the end of the section.

We first introduce the two sets of degrees of freedom 
used in~\cite{BasicsPaper,NonconformingVEM} to describe the original local virtual element spaces.
The final spaces presented later in the section will be described using exactly the same degrees of freedom, so we record them here.
\begin{definition}
	\label{def:originalDofs}
	The degrees of freedom for the local conforming and nonconforming spaces are
	\begin{enumerate}[$(a)$]
		\item for the conforming space
			\begin{itemize}
				\item the value of $\vh$ at each vertex of $\E$;
				\item for $\k > 1$, the moments of $\vh$ of up to order $\k-2$ on each mesh interface $\s\subset\dE$
				\begin{equation*}
				\frac{1}{\abs{\s}} \ints \vh \ma \ds \quad \forall m_{\alpha} \in \Ms{\k-2};
				\end{equation*}
			\end{itemize}
		for the nonconforming space, the moments of $\vh$ of up to order $\k-1$ on each mesh interface $\s\subset\dE$
			\begin{equation*}
				\frac{1}{\abs{\s}} \ints \vh \ma \ds \quad \forall \ma \in \Ms{\k-1};
			\end{equation*}
		\item for $\k > 1$, the moments of $\vh$ of up to order $\k-2$ inside the element $\E$
			\begin{equation*}
				\frac{1}{\abs{\E}} \intE \vh \ma \dx \quad \forall \ma \in \ME{\k-2}.
			\end{equation*}
	\end{enumerate}
\end{definition}

A counting argument shows that the cardinality of the above sets of degrees of freedom is  $\NE = \nuE + \nuE N_{1,k-2} + N_{2,k-2}$ for the conforming case and $\NE = \nuE N_{\spacedim - 1, \k-1} + N_{\spacedim, \k}$ for the nonconforming case. The degrees of freedom for a hexagonal element are represented in Figure~\ref{fig:dofs:hexa:nonconf-VEM} and Figure~\ref{fig:dofs:hexa:conf-VEM} for the nonconforming and conforming case, respectively. The nonconforming degrees of freedom for a cubic element are shown in Figure~\ref{fig:dofs:cube:nonconf-VEM}.

We first consider a conforming and a nonconforming superspace in which both terms in~\eqref{eq:projectionsToCompute} are computable.
Again with deliberate ambiguity, we refer to the enlarged space in each case as $\biglocalspace$.
For the enlarged nonconforming space on the element $\E$, we define
\begin{equation*}
	\biglocalspace := \left\{ \vh \in H^1(\E) : \Delta \vh \in \PE{\k} \text{ and } \frac{\partial \vh}{\partial \n} \in \P{\s}{\k-1} \,\, \forall \s \subset \dE \right\}.
\end{equation*}
For the conforming space, we first introduce the boundary space
\begin{equation*}
	B_{\k}(\dE) := \left\{ \v \in C^0(\dE) : \v|_\s \in \P{\s}{\k} \text{ for each interface } \s \text{ of } \dE \right\},
\end{equation*}
and define $\biglocalspace$ as
\begin{equation*}
	\biglocalspace := \left\{ \vh \in H^1(\D) : \Delta \vh \in \PE{\k} \text{ and } \vh|_{\dE} \in B_{\k}(\dE)\right\}.
\end{equation*}
It is clear that, in either case, we have $\PE{\k}\subset \biglocalspace$. Further, 
the space $\biglocalspace$ may be described using the combination of the  degrees of freedom of Definition~\ref{def:originalDofs} and the \emph{extra} degrees of freedom.

\begin{figure}[!t]
  \centering
  \begin{tabular}{cccc}
    \includegraphics[width=0.2\textwidth]{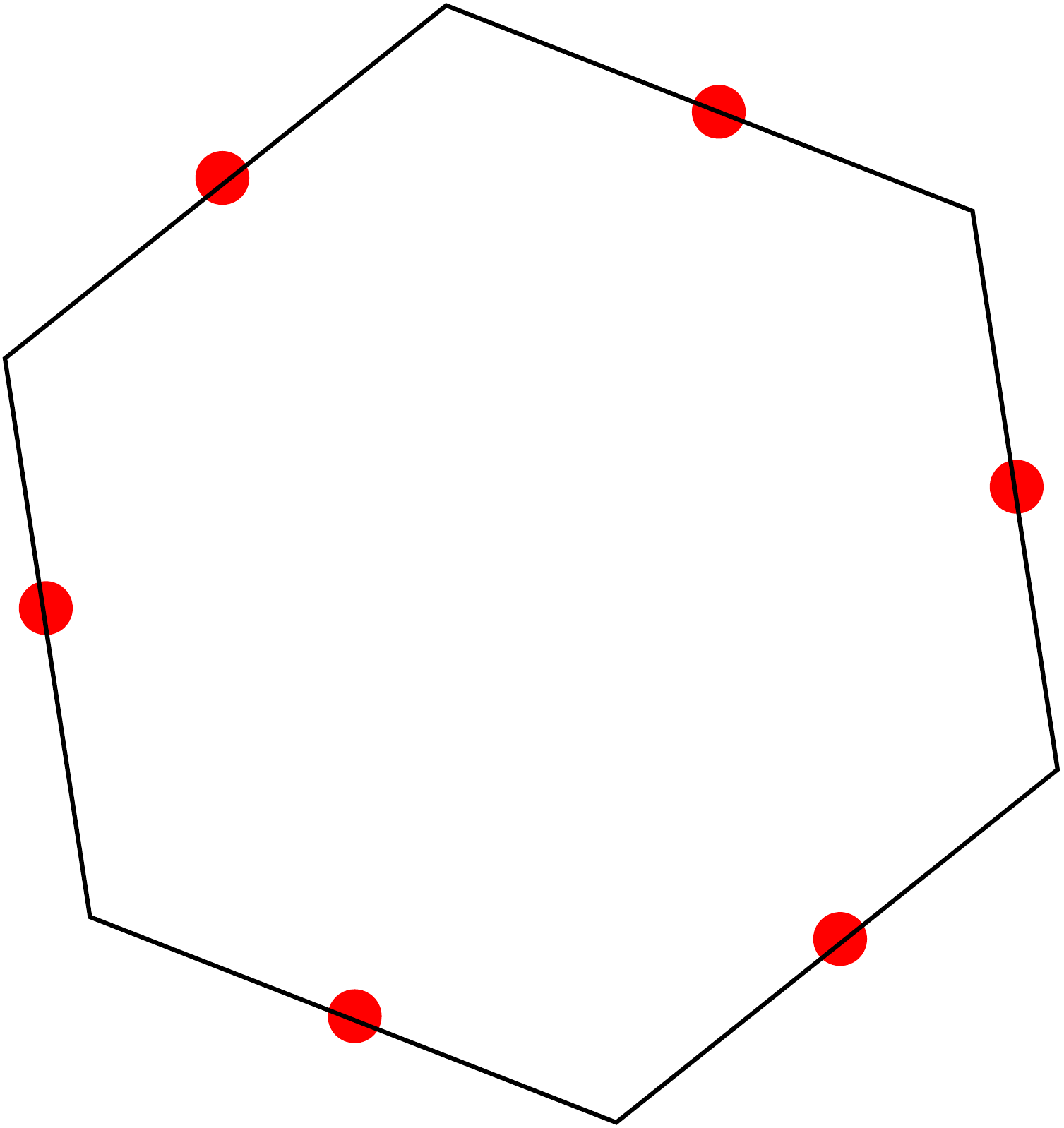} &\quad
    \includegraphics[width=0.2\textwidth]{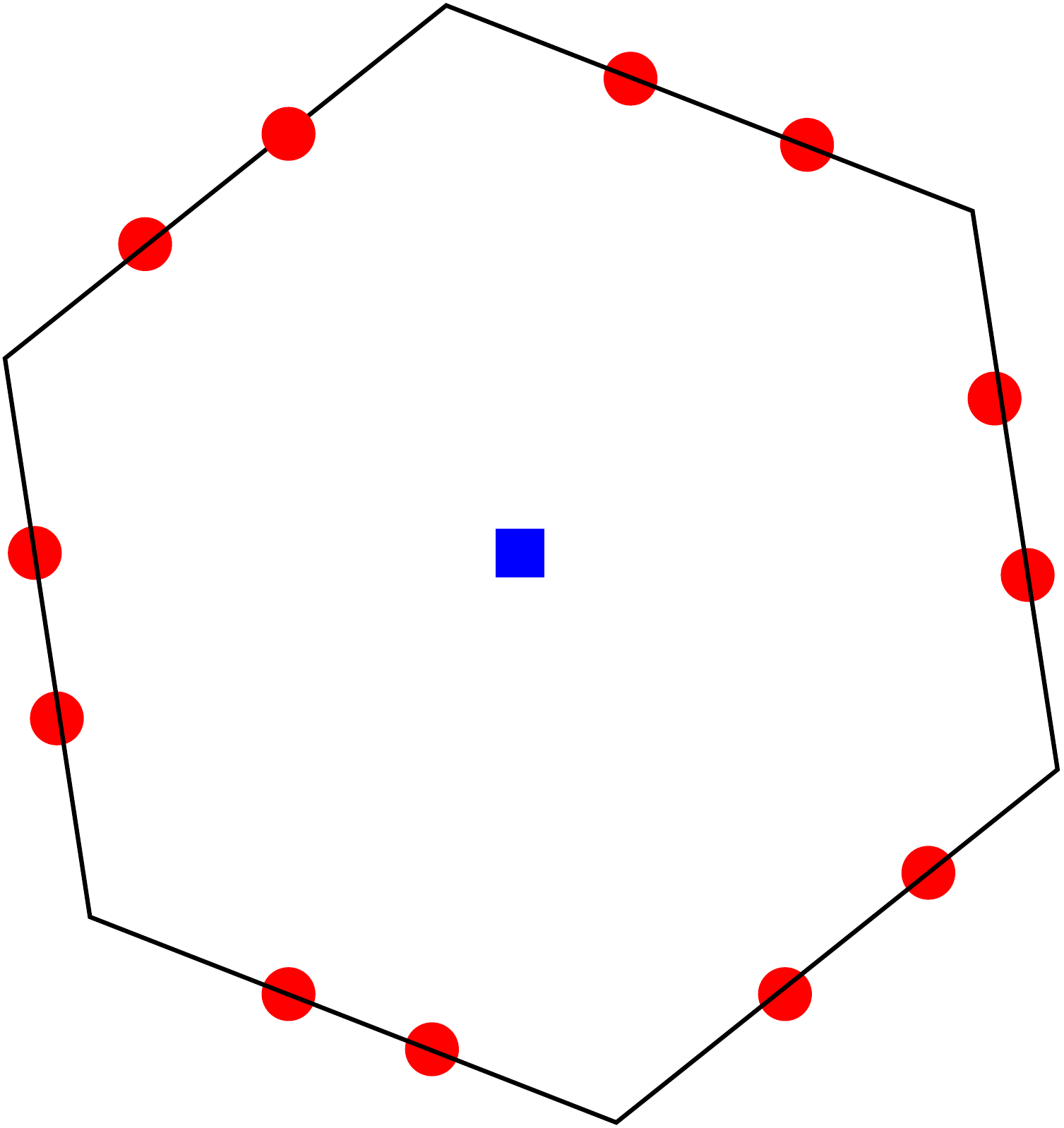} &\quad
    \includegraphics[width=0.2\textwidth]{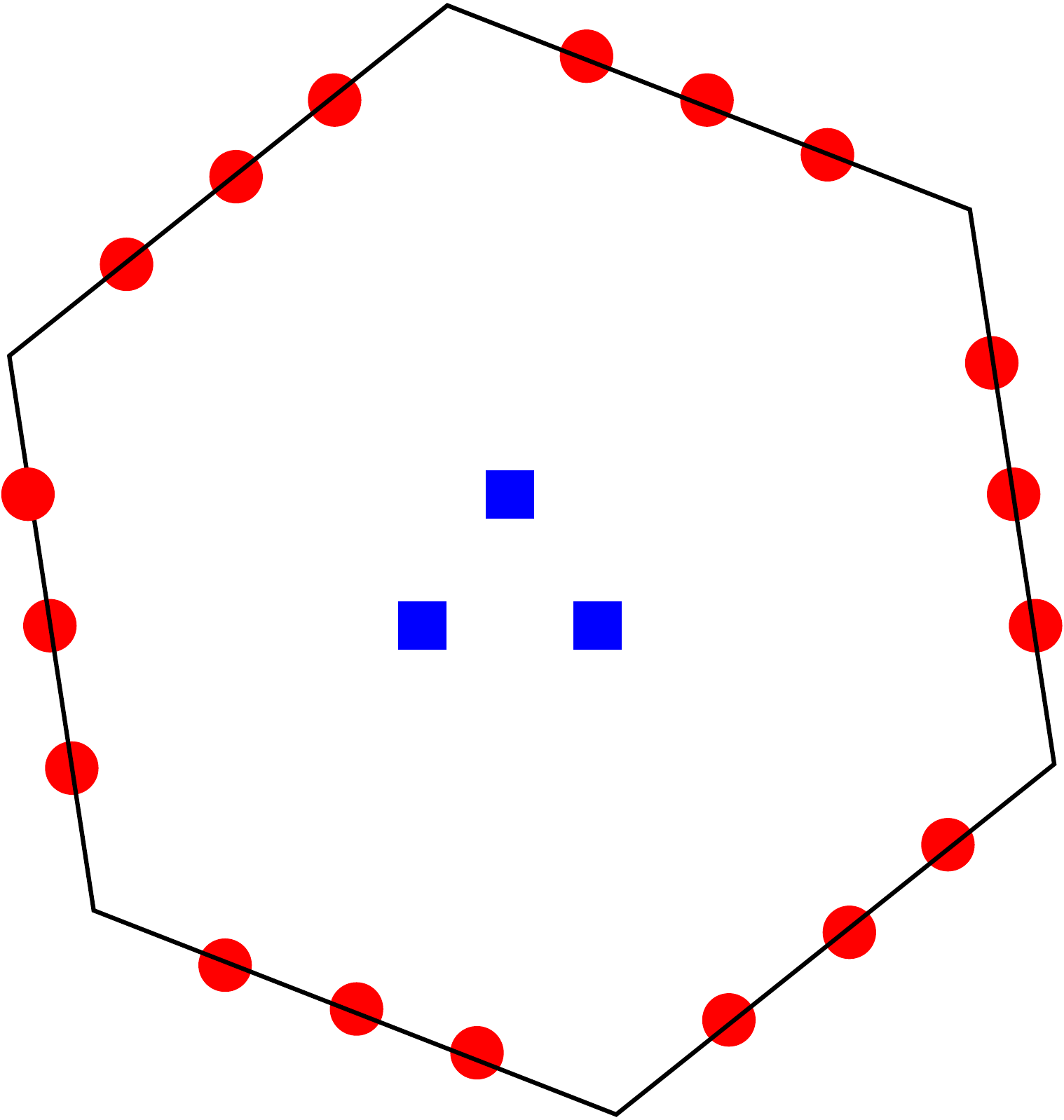} &\quad
    \includegraphics[width=0.2\textwidth]{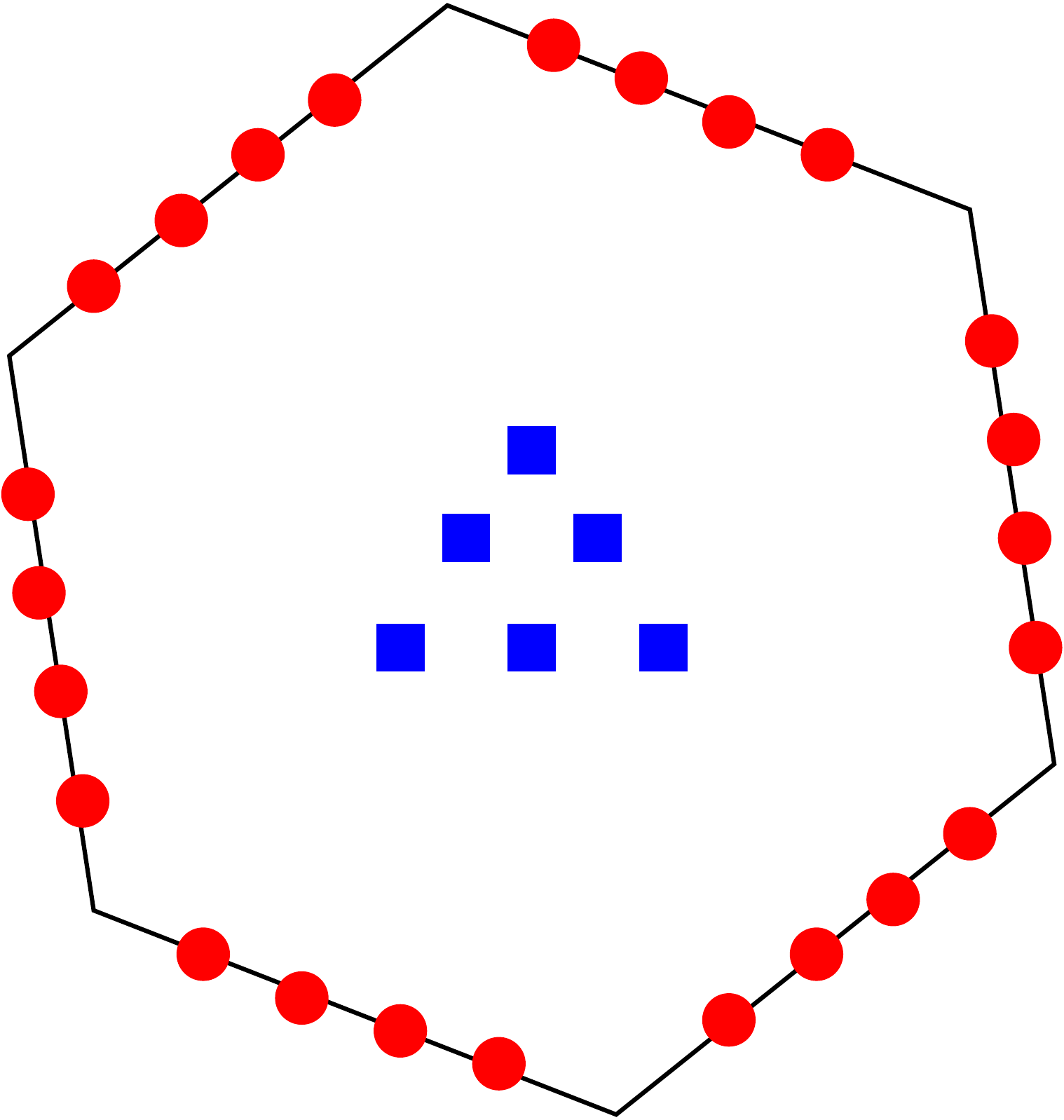} \\
    $\mathbf{k=1}$ & $\mathbf{k=2}$ & $\mathbf{k=3}$ & $\mathbf{k=4}$
  \end{tabular}
  \caption{Degrees of freedom of the non-conforming VEM for a
    hexagonal mesh element for $k=1,2,3,4$; edge moments are marked by a
    circle; internal moments are marked by a square.}
  \label{fig:dofs:hexa:nonconf-VEM}
\end{figure}
\begin{figure}[!t]
  \centering
  \begin{tabular}{cccc}
    \includegraphics[width=0.2\textwidth]{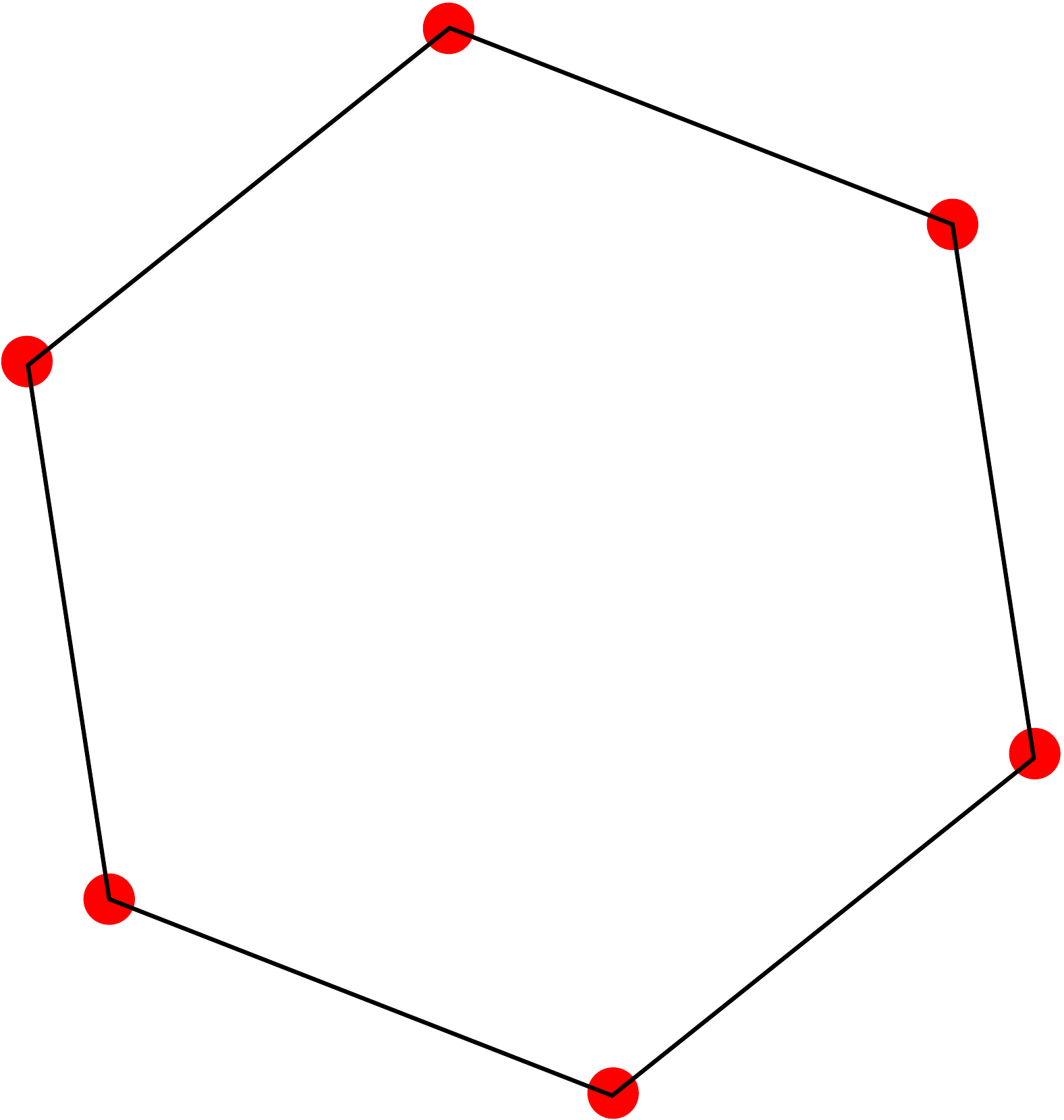} &\quad
    \includegraphics[width=0.2\textwidth]{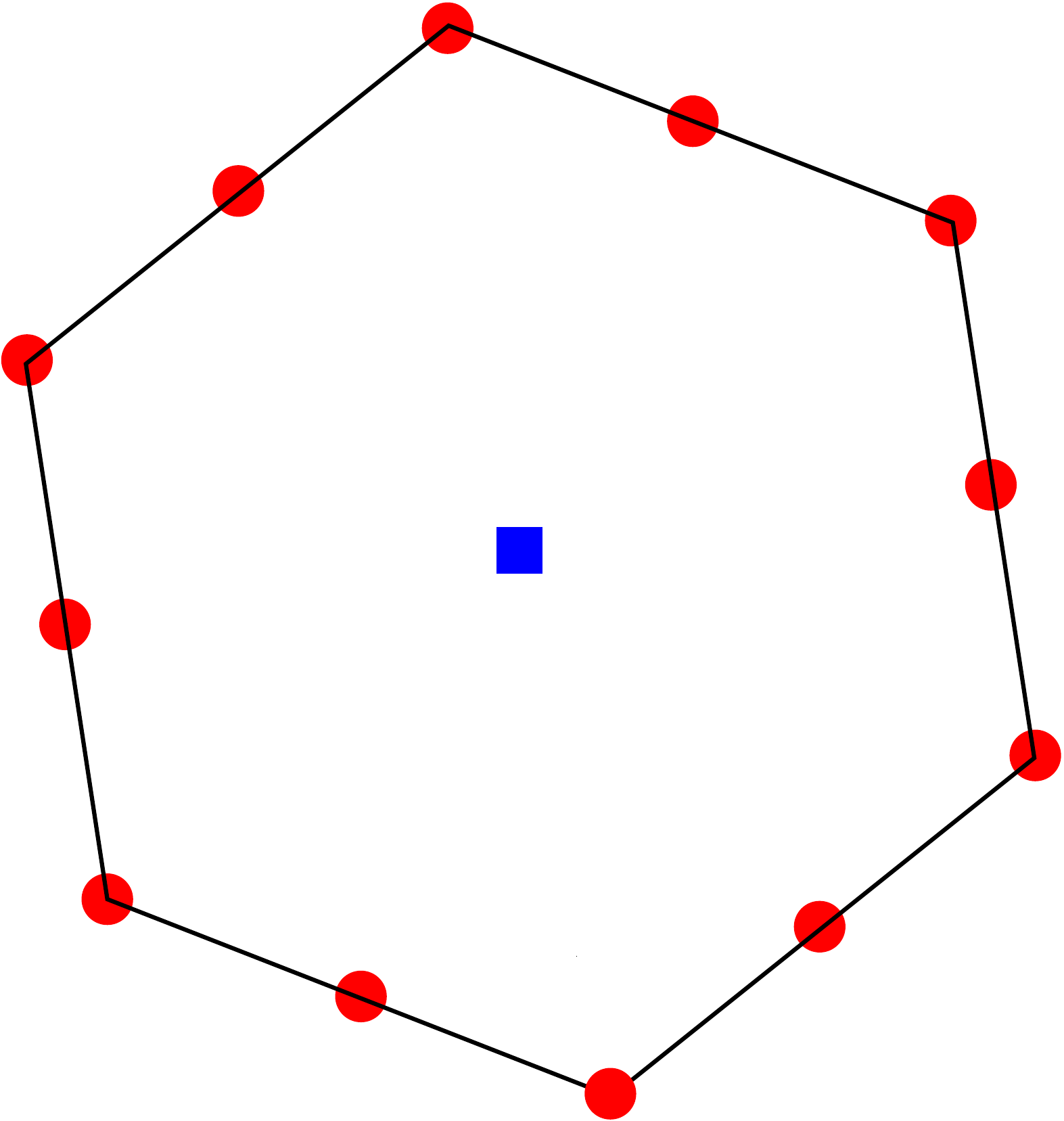} &\quad
    \includegraphics[width=0.2\textwidth]{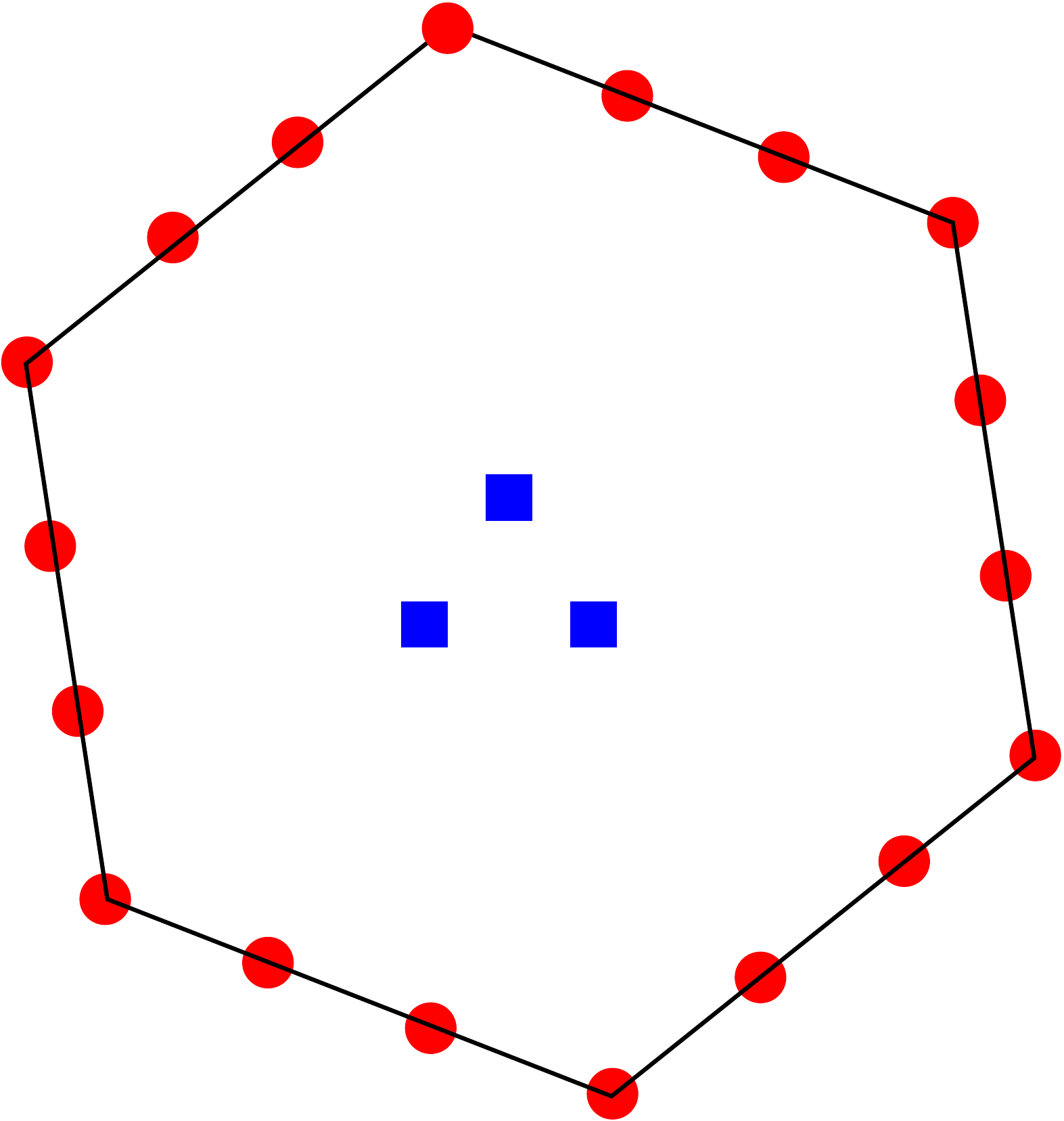} &\quad
    \includegraphics[width=0.2\textwidth]{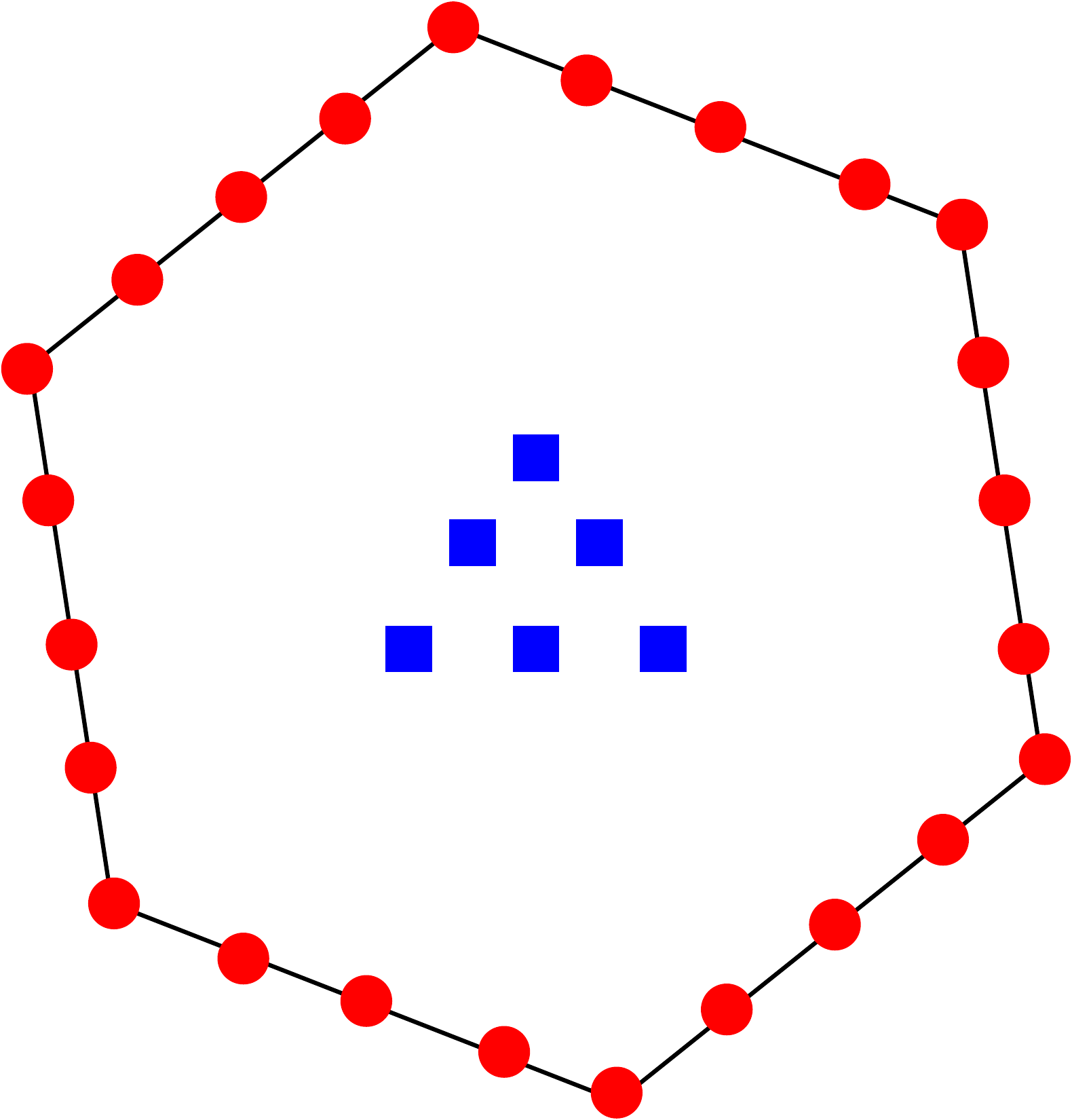} \\
    $\mathbf{k=1}$ & $\mathbf{k=2}$ & $\mathbf{k=3}$ & $\mathbf{k=4}$
  \end{tabular}
  \caption{Degrees of freedom of the conforming VEM for a hexagonal
    mesh element for $k=1,2,3,4$; vertex values and edge moments are marked by
    a circle; internal moments are marked by a square.}
  \label{fig:dofs:hexa:conf-VEM}
\end{figure}

\begin{definition}
	\label{def:extraDofs}
	The \emph{extra} degrees of freedom are taken as the moments of $\vh$ of order $\k$ and $\k-1$ inside the element $\E$
	\begin{equation*}
		\frac{1}{\abs{\E}} \intE \vh \ma \dx \quad \forall \ma \in \MEstar{\k} \cup \MEstar{\k-1}.
	\end{equation*}
\end{definition}

The proof that the combined set of degrees of freedom is unisolvent for the conforming space is given in~\cite{EquivalentProjectors}, and for the nonconforming space follows exactly as the original unisolvence proof in~\cite{NonconformingVEM}. It is based on the observation that each space defines its elements as those functions which solve a particular class of Poisson problem, with piecewise polynomial Dirichlet and Neumann boundary conditions in the conforming and nonconforming case, respectively, specified by the degrees of freedom. 
Similarly, we can easily prove the crucial fact that any $\pk \in \PE{\k}$ is uniquely determined by the original degrees of freedom of Definition~\ref{def:originalDofs}, cf.~\cite{BasicsPaper, NonconformingVEM}. Indeed,  if $\pk\in\PE{\k}$ and all the original degrees of freedom of $\pk$ are zero then, 
\begin{equation*}
(\nabla\pk,\nabla\pk)_\E=(-\Delta \pk\, \pk)_\E+%\sum_{\s\in\dE}
\intdE  \frac{\partial \pk}{\partial \n}\,\pk \ds=0.
\end{equation*}
The first term is zero as $\Delta\pk=0$ if $k=1$ and  as $\Delta\pk\in\PE{k-2}$ if $k>1$ and hence the first term is a linear combination of the (zero) internal degrees of freedom of $\pk$ in this case. The second term is also zero because it is \emph{always} a linear combination of the boundary degrees of freedom of $\pk$. 
Hence, $\pk=~~$constant in $\E$. The fact that $\pk\equiv 0$ now follows from the hypothesis that $\pk$ is zero at any vertex in the conforming case, and $\int_s \pk\ds=0$ for any $s\in\dE$ in the nonconforming case. 

Using collectively the degrees of freedom of Definitions~\ref{def:originalDofs} and~\ref{def:extraDofs}, it is possible to compute both of the projections in~\eqref{eq:projectionsToCompute} for any $\vh \in \biglocalspace$.
Calculating $\Po{\k} \vh$ requires solving the variational problem: find $\Po{\k}\vh \in \PE{\k}$ such that
\begin{equation}
	\label{eq:computingPovh}
	(\Po{\k} \vh, \ma)_{\E} = (\vh, \ma)_{\E} \qquad \forall \ma \in \ME{\k},
\end{equation}
which is computable as the quantities on the right-hand side are the internal degrees of freedom of $\vh$.
Similarly, computing $\Po{\k-1} \nabla \vh$ requires finding the polynomial $\Po{\k-1}\nabla \vh \in \left(\PE{\k}\right)^2$ such that
\begin{align}
  (\Po{\k-1} \nabla \vh, \bm{\ma})_{\E} = (\nabla \vh, \bm{\ma})_{\E} \qquad \forall \bm{\ma} \in (\ME{\k-1})^2.
  \label{eq:computingProjGradvh}
\end{align}
This is possible because the right-hand side is a sum of boundary and internal degrees of freedom of $\vh$:
\begin{align*}
  (\nabla \vh, \bm{\ma})_{\E}
  = \int_{\dE} \n \cdot \bm{\ma} \vh \ds - (\vh, \nabla \cdot \bm{\ma})_{\E}.
\end{align*}
Note that this time only the original degrees of freedom of Definition~\ref{def:originalDofs} are required.

However, in each case this enlarged space requires an extra 
${\rm card}(\MEstar{\k}) + {\rm card}(\MEstar{\k-1})$ degrees of freedom (namely the extra internal moments in \dref{def:extraDofs}) compared with the original spaces introduced in~\cite{BasicsPaper} and~\cite{NonconformingVEM}, which are described by the degrees of freedom of Definition~\ref{def:originalDofs} only.
To reduce the number of degrees of freedom, we adopt a generalisation of the procedure introduced in~\cite{EquivalentProjectors}, producing a family of different subspaces of $\biglocalspace$ spanned by the original sets of degrees of freedom of Definition~\ref{def:originalDofs}, yet in which we can still compute the required projections in~\eqref{eq:projectionsToCompute}. The procedure consists of the following three steps.

\begin{itemize}
\item[1)]
We introduce an equivalence relation $\sim$ on $\biglocalspace$, defining $\vh \sim \wh$ if all of the original degrees of freedom of $\vh$ and $\wh$ are equal, and consider the quotient space $\biglocalspace /_\sim$ which, by construction, is spanned by the original degrees of freedom of Definition~\ref{def:originalDofs}.

\item[2)] Since $\PE{\k} \subset \biglocalspace$ and any $\pk \in \PE{\k}$ is uniquely determined by the original degrees of freedom, 
we may conclude that any equivalence class $[\vh]$ contains at most one polynomial.
Then, 
we may unambiguously associate $\PE{\k}$ with the resulting `polynomial' subspace of $\biglocalspace /_{\sim}$.
Hence, we can introduce \emph{any} projection operator $\spaceProj{\k} : \biglocalspace /_{\sim} \rightarrow \PE{\k} \subset \biglocalspace /_{\sim}$ which associates a polynomial to each equivalence class $[\vh]$.\footnote{Alternatively, one may think of this as associating a polynomial to each combination of the original degrees of freedom in \dref{def:originalDofs}, independent of the extra degrees of freedom in \dref{def:extraDofs}.}
\item[3)]
The local virtual space $\VhE$ is then defined by selecting a specific representative from each equivalence class in $\biglocalspace /_\sim$. For each $[\vh]\in\biglocalspace /_{\sim}$
we take the function $\wh \in [\vh]$ such that the \emph{extra} degrees of freedom (in \dref{def:extraDofs}) of $\wh$ are equal to those of $\spaceProj{\k}[\vh]$.
Note in particular that $\PE{\k}\subset\VhE$.
\end{itemize}

\begin{remark}
	This is a generalisation of the idea introduced in~\cite{EquivalentProjectors}, where only the $H^1(\E)$-orthogonal projection of $\vh$ into $\PE{\k}$ was considered for $\spaceProj{}$.
	The space resulting from this choice is well defined because the $H^1(\E)$-orthogonal projection of $\vh$ is computable using just the original degrees of freedom.
	However, the freedom in choosing $\spaceProj{}$ is something we wish to exploit to produce a more computationally efficient method, particularly when $\spacedim = 3$.
	We explore more possible choices of $\spaceProj{}$ in Section~\ref{subsec:imple}, although for now we leave this choice open.
\end{remark}

In more concrete terms, given a projector $\spaceProj{\k}$, we define the local virtual element spaces to be
\begin{equation}
	\label{eq:localVEMSpace}
	\VhE := \Big\{ \vh \in \biglocalspace :  \left(\vh - \spaceProj{\k}\vh, p\right)_{\E} = 0 \,\,\, \forall p \in \MEstar{\k} \cup \MEstar{\k-1} \Big\},
\end{equation}
where $\biglocalspace$ denotes either the enlarged conforming or nonconforming space. Clearly, we can use the original degrees of freedom of~\dref{def:originalDofs} to describe $\VhE$.

Computing $\Po{\k} \vh$  for each $\vh\in\VhE$ is now possible, since the terms on the right-hand since of~\eqref{eq:computingPovh} are either degrees of freedom of $\vh$ or moments of $\spaceProj{\k}\vh$.

\subsection{The Global Spaces}
The global virtual element space in each case is constructed as a subspace of an infinite dimensional space $V$, defined differently for the conforming and nonconforming methods.
For the conforming method, we simply take $V := H^1_0(\D)$.
For the nonconforming method, we introduce the subspace $H^{1,\text{nc}}_{\k}(\Th)$ of the nonconforming broken Sobolev space $H^1(\Th)$ defined in~\eqref{eq:brokenH1}, by imposing certain weak inter-element continuity requirements such that
\begin{equation*}
	V:= H^{1,\text{nc}}_{\k}(\Th) = \left\{\v \in H^1(\Th) : \ints \jump{v} \cdot \n_\s \, q \, \ds = 0 \quad \forall q \in \P{\s}{\k-1},\, \forall \s \in \Edges \right\}.
\end{equation*}
The jump operator $\jump{\cdot}$ across a mesh interface $\s\in \Edges$ is
defined as follows for $v\in H^1(\Th)$.  If $\s \in \InternalEdges$, then there exist $\E^+$ and $\E^-$ such that
$\s\subset \partial\E^+\cap\partial\E^-$.  Denote by $v^{\pm}$ the
trace of $v_{|_{\E^{\pm}}}$ on $\s$ from within $\E^{\pm}$ and by
$\n_\s^{\pm}$ the unit outward normal on $\s$ from $\E^{\pm}$. Then, $\jump{v}:=v^+\n_{\s}^{+}+v^-\n_{\s}^{-}$.  If, on the other hand,
$\s \in \BoundaryEdges$, then $\jump{v}:=v\n_{\s}$, with $\v$
representing the trace of $v$ from within the element $\E$ having $\s$
as an interface and $\n_{\s}$ is the unit outward normal on $\s$ from $\E$.

It may be seen that the broken Sobolev norm $|\cdot|_{1,h}$ is a norm on $H^{1,\text{nc}}_{\k}(\Th)$ and hence the same will be true for any virtual element subspace $\Vh$, as required by Assumption~\ref{ass:vem}, cf.~\cite{Brenner-PF}.

Finally, the global space is constructed in either case from the local spaces presented above as
\begin{equation*}
	\Vh := \left\{ \vh \in V : \vh |_{\E} \in \VhE \quad \forall \E \in \Th \right\}.
\end{equation*}

\begin{figure}[!t]
  \centering
  \begin{tabular}{cccc}
    \includegraphics[width=0.2\textwidth]{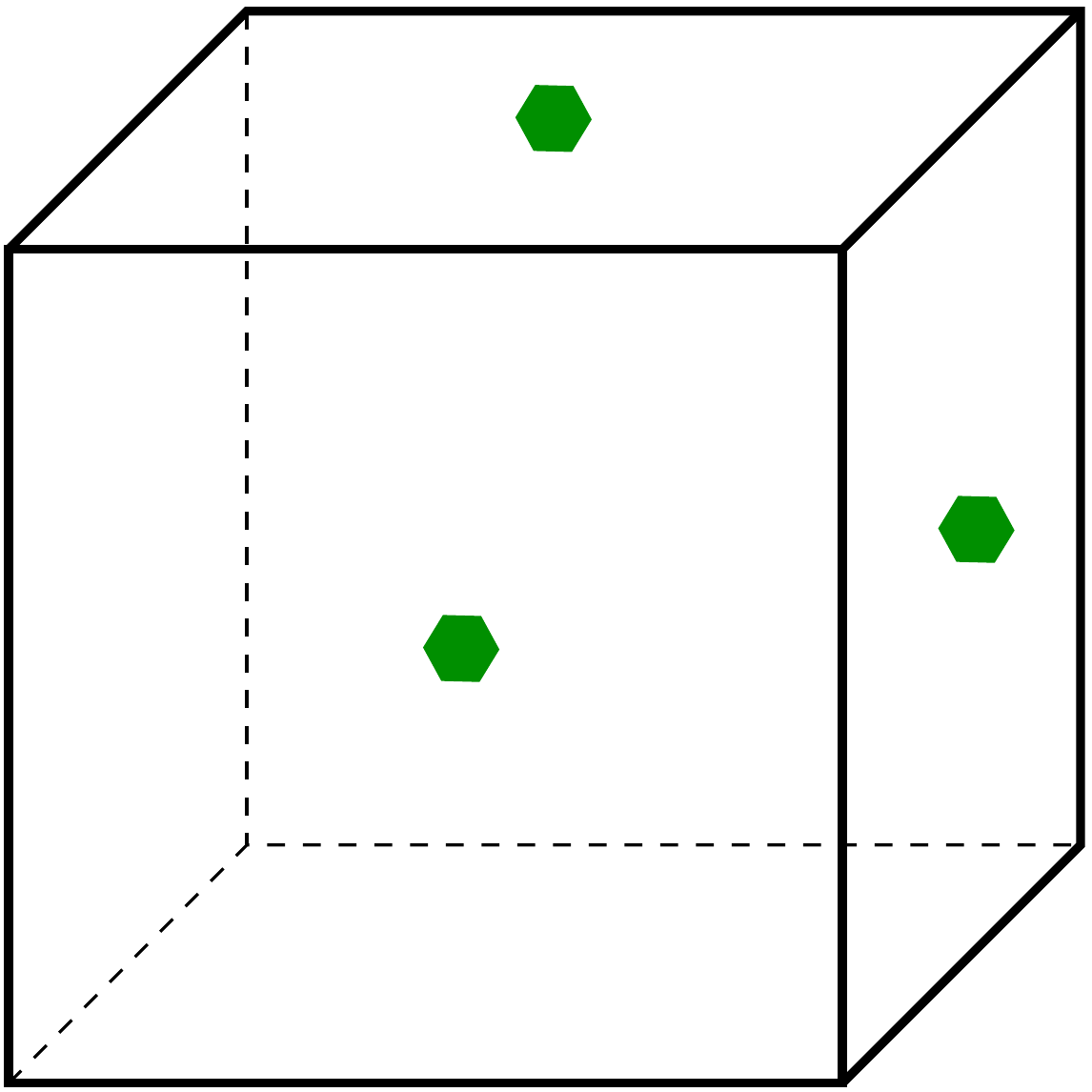} &\quad
    \includegraphics[width=0.2\textwidth]{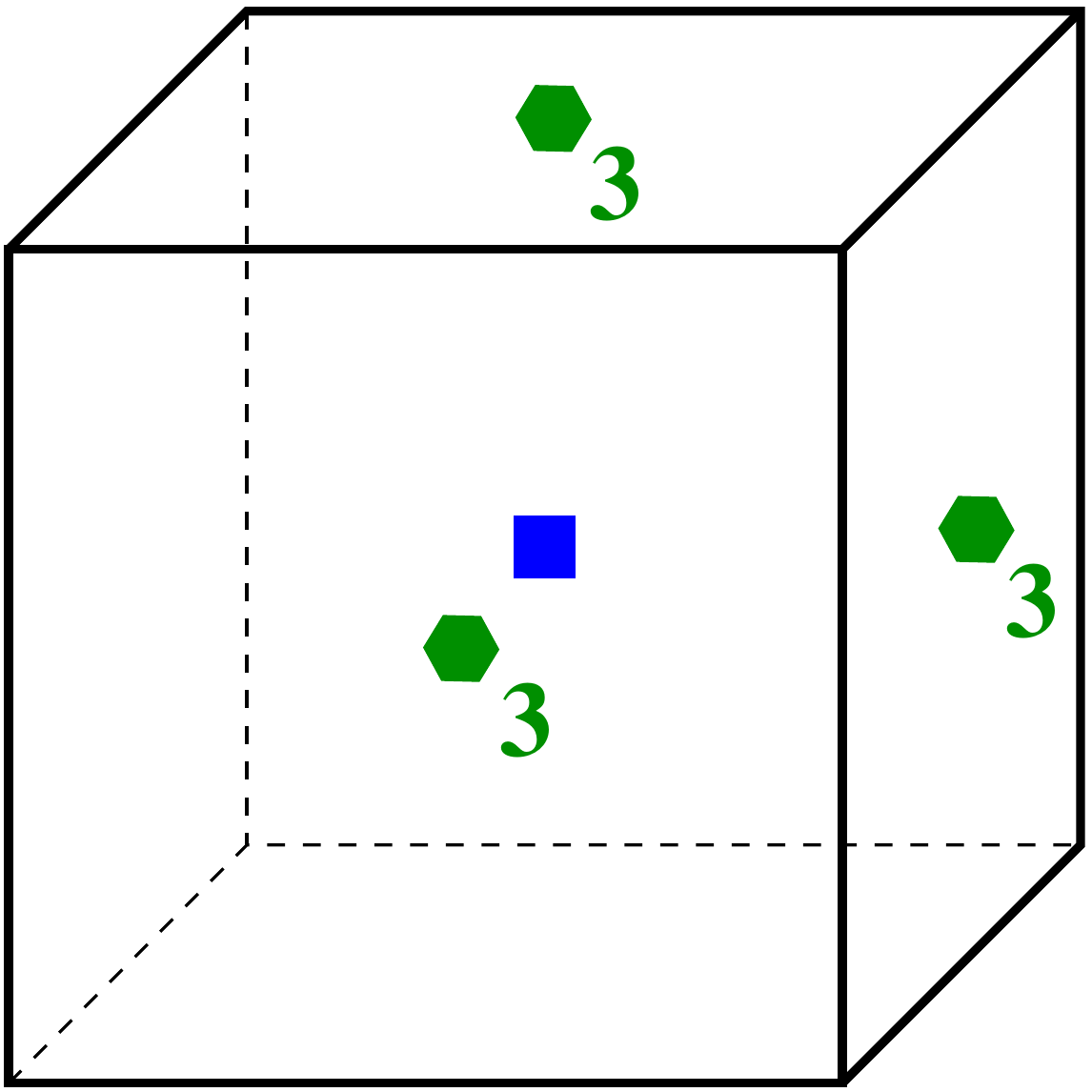} &\quad
    \includegraphics[width=0.2\textwidth]{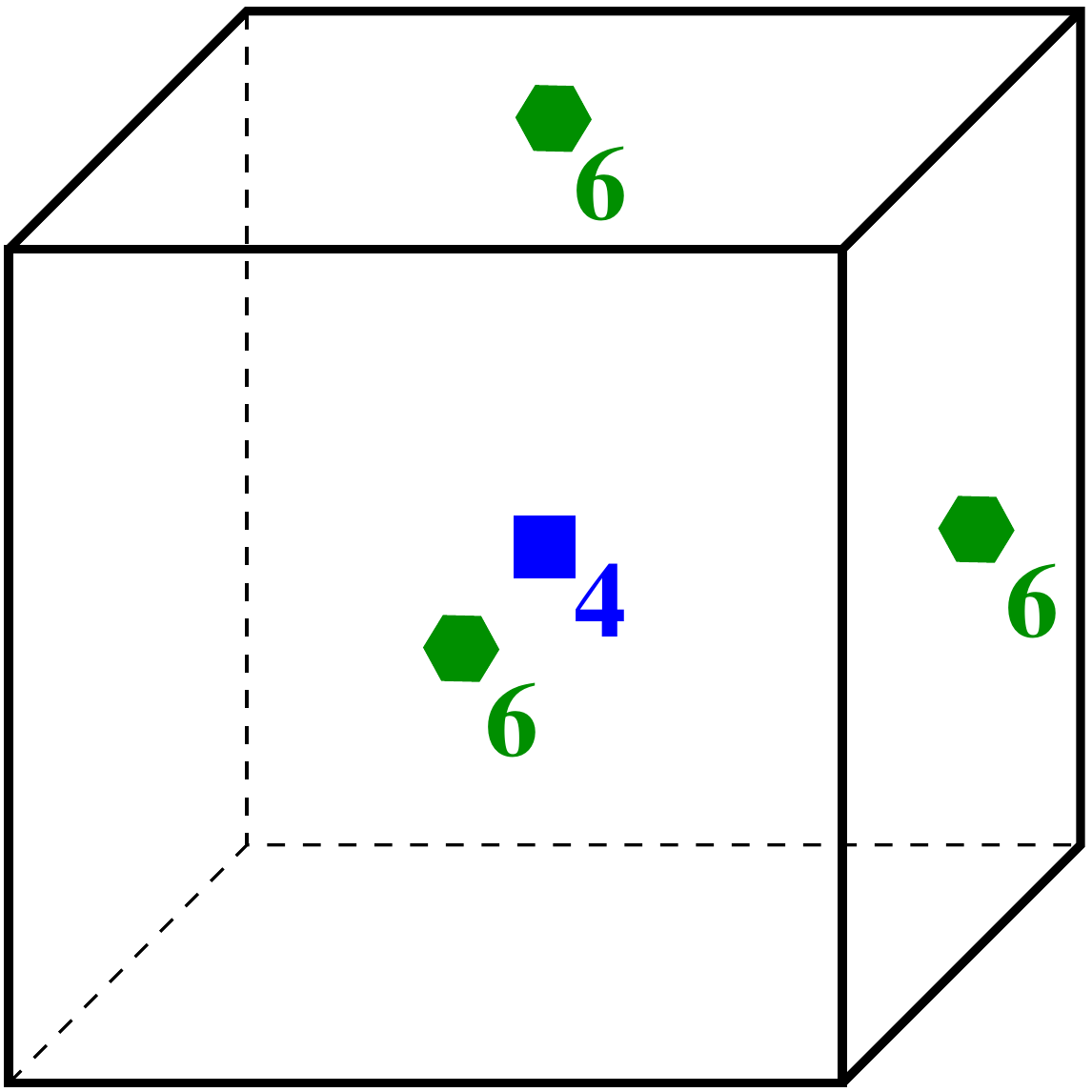} &\quad
    \includegraphics[width=0.2\textwidth]{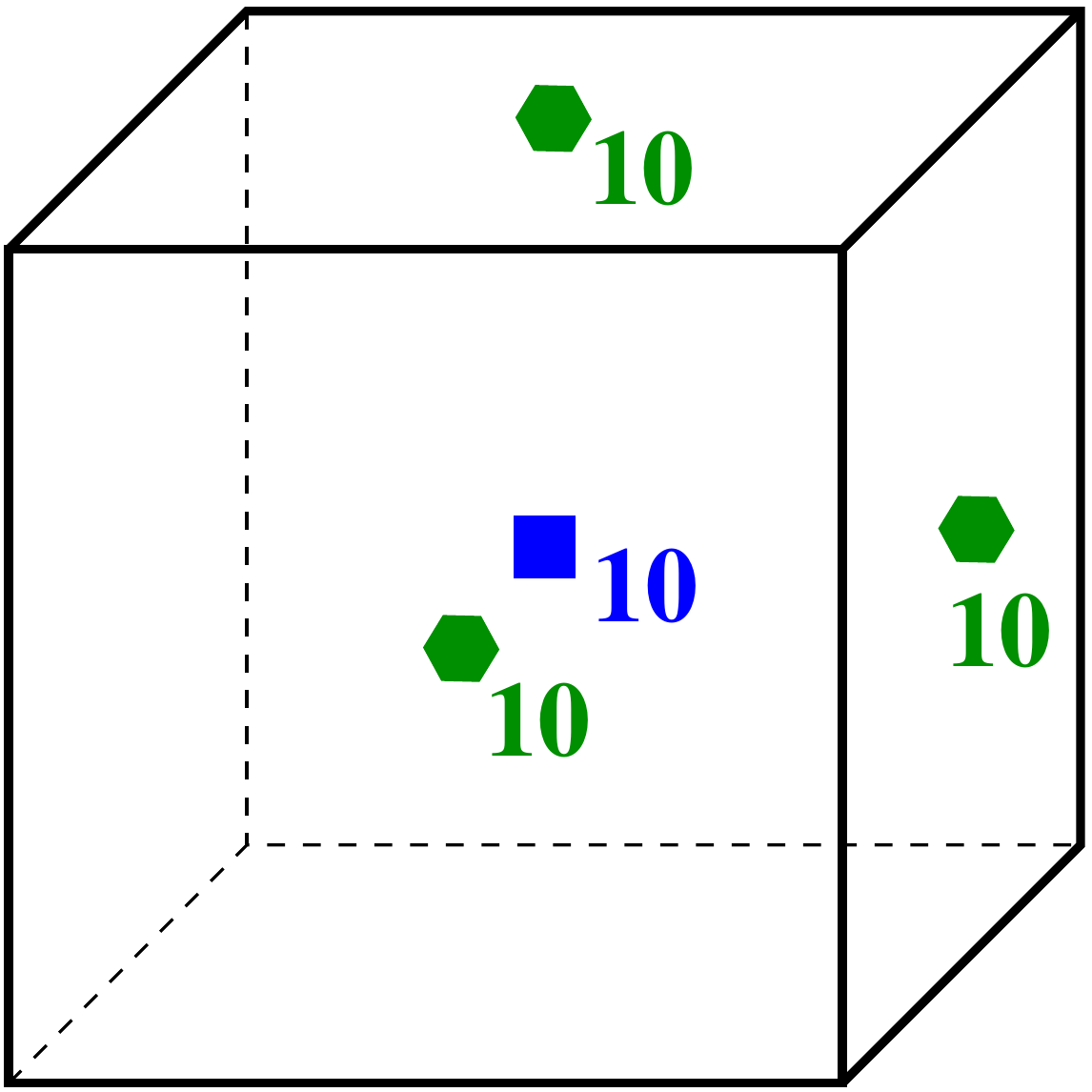} \\
    $\mathbf{k=1}$ & $\mathbf{k=2}$ & $\mathbf{k=3}$ & $\mathbf{k=4}$
  \end{tabular}
  \caption{Degrees of freedom of the non-conforming VEM for a
    cubic mesh element for $k=1,2,3,4$; face moments are marked by a
    hexagon; internal moments are marked by a square. Only the internal
    degrees of freedom and those of the visible faces are marked; the
    numeric labels indicate the number of degrees of freedom when they
    are more than $1$.}
  \label{fig:dofs:cube:nonconf-VEM}
\end{figure}

\begin{figure}[!t]
  \centering
  \begin{tabular}{cccc}
    \includegraphics[width=0.2\textwidth]{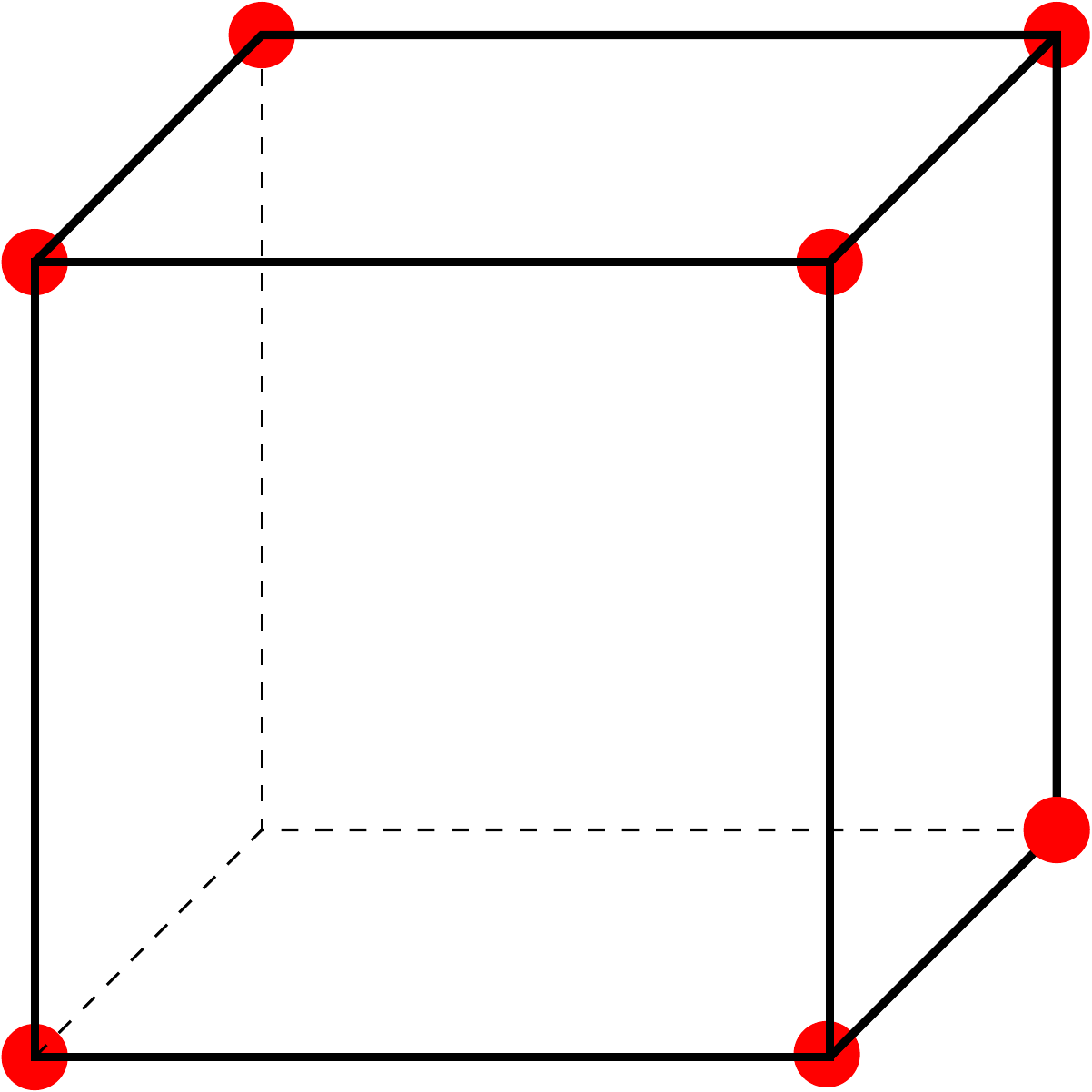} &\quad
    \includegraphics[width=0.2\textwidth]{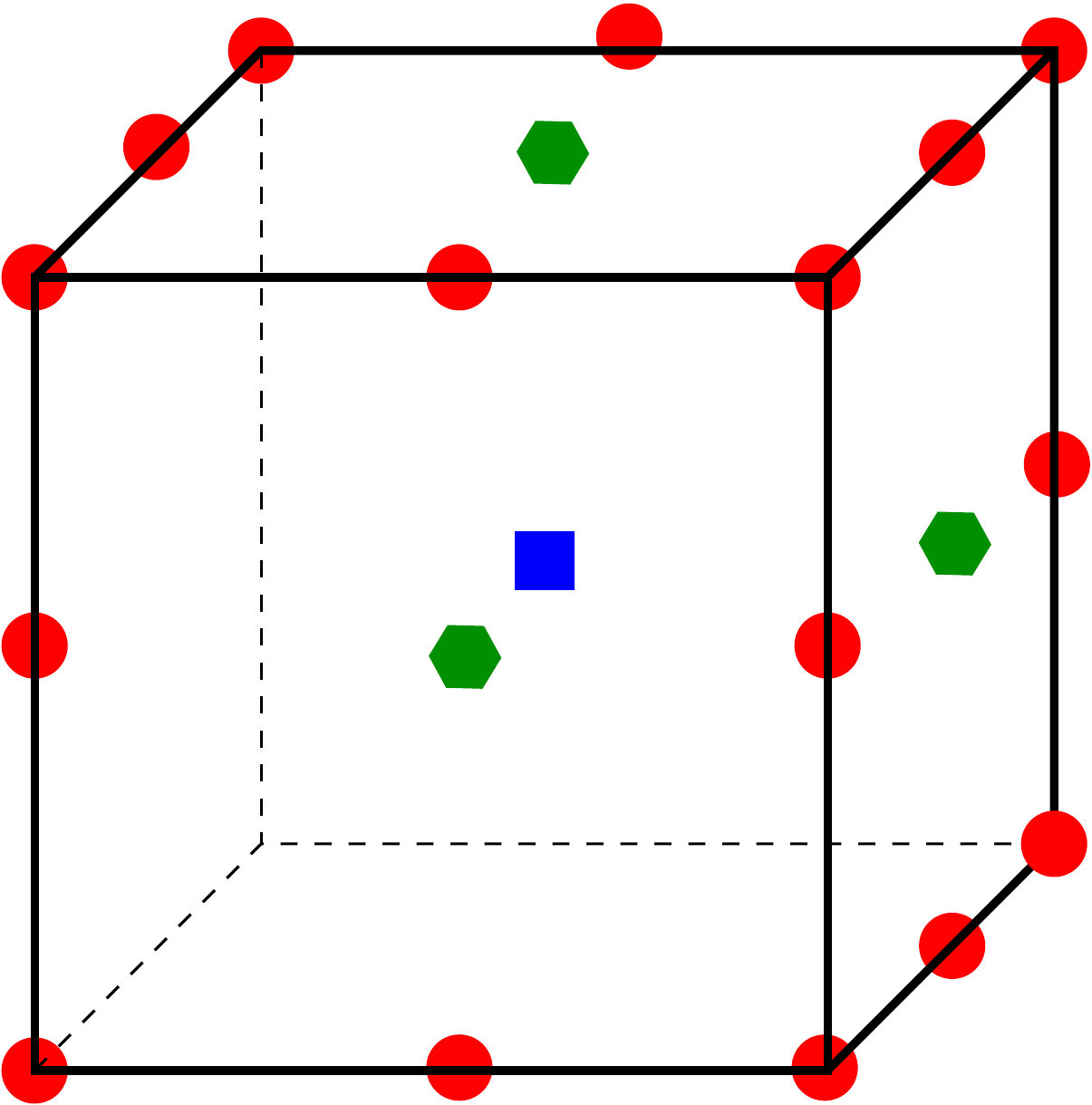} &\quad
    \includegraphics[width=0.2\textwidth]{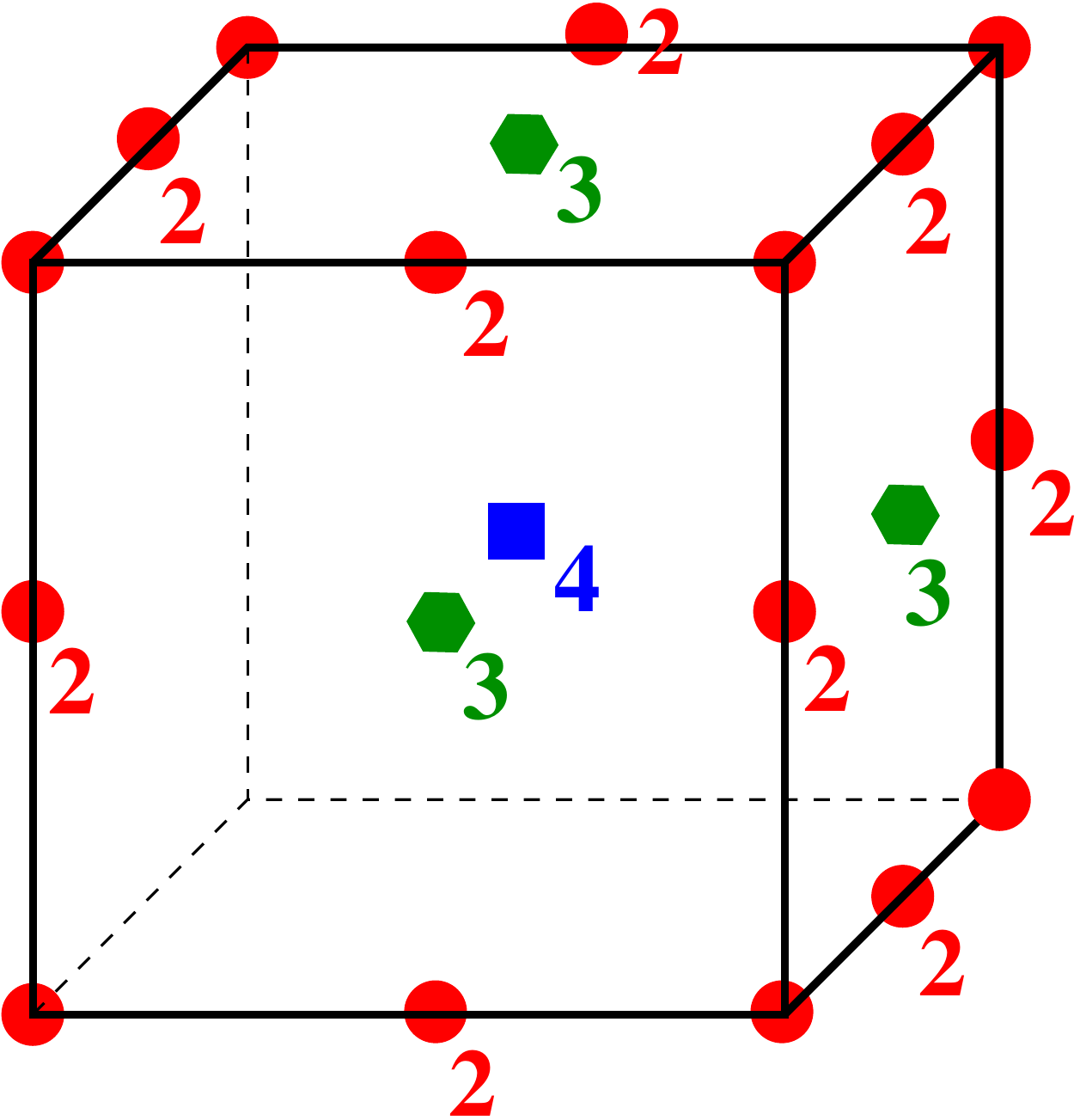} &\quad
    \includegraphics[width=0.2\textwidth]{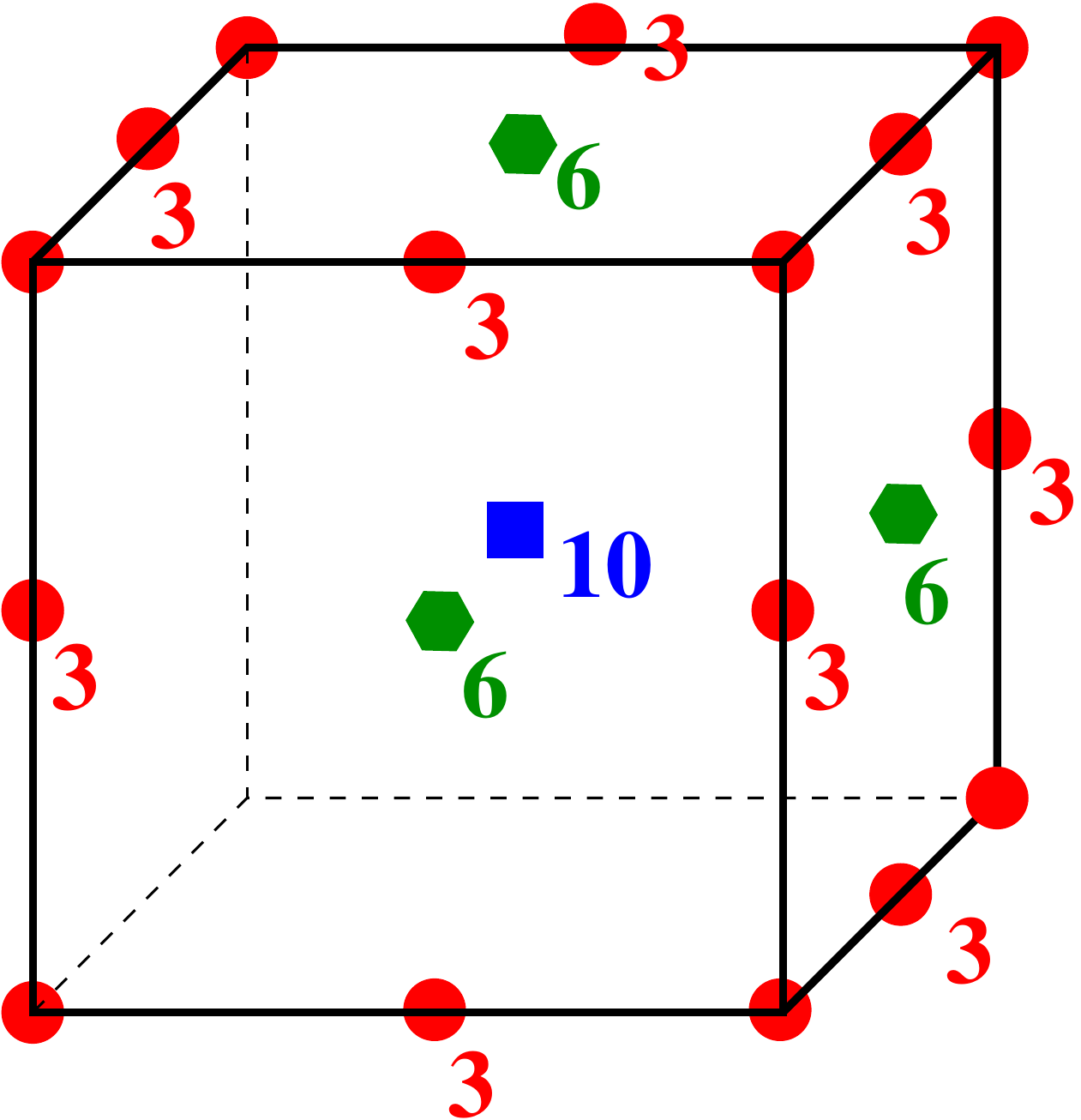} \\
    $\mathbf{k=1}$ & $\mathbf{k=2}$ & $\mathbf{k=3}$ & $\mathbf{k=4}$
  \end{tabular}
  \caption{Degrees of freedom of the conforming VEM for a cubic mesh 
	element for $k=1,2,3,4$; vertex values and edge moments are marked 
	by a circle; face moments are marked by an hexagon; internal moments are
    marked by a square. Only the internal degrees of freedom and those 
	of the visible faces and edges are marked; the numeric labels
    indicate the number of degrees of freedom when they are more than
    $1$.}
  \label{fig:dofs:cube:conf-VEM}
\end{figure}

As global degrees of freedom we take the equivalents of those in \dref{def:originalDofs}, namely, for each function $\vh \in \VhE$,
\begin{enumerate}[$(a)$]
	\item for the conforming space
		\begin{itemize}
			\item the value of $\vh$ at each internal vertex of $\Th$;
			\item for $\k > 1$, the moments of $\vh$ of up to order $\k-2$ on each mesh interface $\s \in \InternalEdges$ 
			\begin{equation*}
				\frac{1}{\abs{\s}} \ints \vh \ma \ds \quad \forall m_{\alpha} \in \Ms{\k-2};
			\end{equation*}
		\end{itemize}
		for the nonconforming space, the moments of $\vh$ of up to order $\k-1$ on each mesh interface $\s \in \InternalEdges$
                \begin{equation*}
                  \frac{1}{\abs{\s}} \ints \vh \ma \ds \quad \forall \ma \in \Ms{\k-1};
                \end{equation*}
		\item for $\k > 1$, the moments of $\vh$ of up to order $\k - 2$ inside each element $\E \in \Th$ 
		\begin{equation*}
			\frac{1}{\abs{\E}} \intE \vh \ma \dx \quad \forall m_{\alpha} \in \ME{\k-2}.
		\end{equation*}
\end{enumerate}
The local degrees of freedom corresponding to boundary vertices and edges $\s \in \BoundaryEdges$ are fixed as zero in accordance with the definition of the ambient spaces.
The unisolvency of these degrees of freedom follows from the definition of the relevant ambient space in each case and the unisolvency of the local degrees of freedom.

\subsection{The Conforming Space for $\spacedim = 3$}
The construction of the local conforming virtual element space for $\spacedim=3$ is based recursively on the space just detailed for $\spacedim=2$.
Define $\Vhb$ as the $2$-dimensional conforming virtual element space of order $k$ constructed over the polygonal interfaces making up $\dE$.
Then, we define the space $\biglocalspace$ in this case to be
\begin{equation*}
	\biglocalspace := \{ \vh \in H^1(\E) : \vh|_{\dE} \in \Vhb \text{ and } \Delta \vh \in \PE{\k} \}.
\end{equation*}
The degrees of freedom that we take for each function $\vh \in \VhE$ are then
\begin{enumerate}[$(a)$]
	\item the degrees of freedom of $\Vhb$;
	\item for $\k > 1$, the moments of $\vh$ of up to order $\k - 2$ inside the element $\E$ 
		\begin{equation*}
			\frac{1}{\abs{\E}} \intE \vh \ma \dx \quad \forall m_{\alpha} \in \ME{\k-2}.
		\end{equation*}
\end{enumerate}
plus the extra degrees of freedom of \dref{def:extraDofs}.

This allows us to construct the space $\VhE$ from $\biglocalspace$ in exactly the same manner detailed above. As before, the space $\VhE$ is , spanned by just the first sets of degrees of freedom given in (a) and (b) above. The proof that these degrees of freedom are unisolvent is again given in~\cite{EquivalentProjectors}. 
Also, it is clear that the dimension of the local space for $\spacedim = 3$ is $\NE = \nuE''+\nuE'N_{1,k-2}+\nuE N_{2,k-2}+N_{3,k-2}$ where $\nuE''$ and $\nuE'$  denote, respectively, the number of vertices and edges of $\E$. The degrees of freedom for a cubic element are shown in Figure~\ref{fig:dofs:cube:conf-VEM}.

Computing $\Po{\k} \vh$ is just the same as for $\spacedim = 2$, since the terms on the right hand since of~\eqref{eq:computingPovh} are either degrees of freedom of $\vh$ or moments of $\spaceProj{\k}\vh$.
To compute $\Po{\k-1}\nabla\vh$, we must compute the face terms in~\eqref{eq:computingProjGradvh}.
Using the $\LTWO(\s)$-orthogonal projection on the face $\s$, these may be rewritten as
\begin{equation*}
	\int_{\s} \n \cdot \bm{\ma} \Pos{\k}\vh \ds \qquad \forall \s \subset \dE.
\end{equation*}
The face projection $\Pos{\k}\vh$ is computable using the degrees of freedom of $\vh$ on the face $\s$ since $\vh|_{\s} \in V_h^{\s}$, and consequently this term is also computable.
This means that the degrees of freedom in this space allow us to compute both of the required terms in~\eqref{eq:projectionsToCompute}.

Finally, the global space and the set of global degrees of freedom for $\spacedim=3$ are constructed from the local ones in the obvious way, completely analogously to the case for $\spacedim=2$.

\subsection{Approximation Properties}
\label{subsec:approximation-properties}
Both the conforming and nonconforming spaces presented above satisfy optimal approximation results for the approximation of sufficiently smooth functions.
Since these results will be used throughout the remainder of the paper, we collect them together here.
They rely on the following assumption on the regularity of the mesh $\Th$:
\begin{assump}{A3}{(Mesh regularity).}
	\label{ass:meshReg}
	We assume the existence of a constant $\meshReg > 0$
 such that
	\begin{itemize}
		\item for every element $\E$ of $\Th$ and every interface $\s$ of $\E$, $h_{\s} \geq \meshReg h_{\E}$
		\item every element $\E$ of $\Th$ is star-shaped with respect to a ball of radius $\meshReg h_{\E}$
		\item for $\spacedim=3$, every interface $\s \in \Edges$ is star-shaped with respect to a ball of radius $\meshReg h_{\s}$.
	\end{itemize}
\end{assump}

\begin{theorem}[Approximation using polynomials]
	\label{thm:polynomialApproximation}
	Suppose that Assumption~\ref{ass:meshReg} is satisfied.
	Let $\E\in\Th$ and let $\Po{\ell} : \LTWO(\E) \rightarrow \PE{\ell}$, for $\ell\ge 0$, denote the $\LTWO(\E)$-orthogonal projection onto the polynomial space $\PE{\ell}$.
	Then, for any $\w \in H^{\reg}(\E)$, with  $1 \leq \reg \leq \ell+1$, it holds
	\begin{equation*}
		%\norm{\w - \w_{\pi}}_{0,\E} + h_{\E} \abs{\w - \w_{\pi}}_{1,\E} \leq Ch_{\E}^\reg \abs{\w}_{\reg,\E},
		\norm{\w - \Po{\ell}\w}_{0,\E} + h_{\E} \abs{\w - \Po{\ell}\w}_{1,\E} \leq Ch_{\E}^\reg \abs{\w}_{\reg,\E}.
	\end{equation*}
	Let $\s$ be an interface shared by $\E^{+}, \E^{-}\in\Th$ and let $\Pos{l}:\LTWO(\s)\rightarrow \mathcal{P}_l(\s)$, for $l\ge 0$, denote the $\LTWO(\s)$-orthogonal projector onto the polynomial space $\mathcal{P}_l(\s)$. Then,  for every $\w \in H^{\reg}(\E^{+}\cup\E^{-})$, with  $1 \leq \reg \leq \ell+1$, it holds
	\begin{equation*}
	\big|\w - \Pos{\ell}\w\big|_{0,\s} + h_\s \big|\w - \Pos{\ell}\w\big|_{1,\s} \leq Ch_{\s}^{\reg-1/2} \norm{\w}_{\reg,\E^{+}\cup\E^{-}}.
	\end{equation*}
	In both instances, the positive constant $C$ depends only on the polynomial degree $\ell$ and the mesh regularity.
	\end{theorem}
This theorem may be proven using the theory in~\cite{BrennerScott} for star-shaped domains and its extension to more general shaped elements presented in e.g.~\cite{ScottDupont}. We also have the following result regarding the approximation of sufficiently smooth functions by those of the virtual element space, which may be proven as in~\cite{SteklovVEM}.
\begin{theorem}[Approximation using virtual element functions]
	\label{thm:spaceApproximation}
	Suppose that Assumption~\ref{ass:meshReg} is satisfied and let $\Vh$ denote either the conforming or nonconforming virtual element space. Let $\reg$ be a positive integer such that $1 \leq \reg \leq \k+1$. Then, for any $\w \in H^\reg(\D)$, there exists an element $\wI \in \Vh$ such that
	\begin{equation*}
		\norm{\w - \wI}_{0} + h \abs{\w - \wI}_{1} \leq Ch^\reg \abs{\w}_{\reg}
	\end{equation*}
	where $C$ is a positive constant which depends only on the polynomial degree $\k$ and the mesh regularity.
\end{theorem}

\section{Error Analysis}
\label{sec:h1ErrorBound}

We are now in a position to prove an optimal order error bound in the
$H^1$- and $\LTWO$-norm for any of the VEM introduced above,
starting with an estimate of the nonconformity error introduced by
using the nonconforming virtual element space.
\begin{lemma}
  \label{lem:strangNonconformityTerm}
  Suppose that Assumptions~\ref{ass:vem}-\ref{ass:meshReg} are
  satisfied  and let $\u \in
  H^{\reg+1}(\D)$ for some positive integer $\reg \geq 1$ be the
  solution to~\eqref{eq:origVariationalForm}. Define $r =
  \min(k,\reg)$ and suppose that the coefficients $\diff, \conv,
  \reacSym \in W^{r+1,\infty}(\Omega)$. Then, there exists a
  positive constant $C$ independent of $h$ and $\u$ such that
  \begin{equation*}
    \sup_{\wh \in \Vh} \frac{\abs{\A(\u, \wh) - (\force, \wh)}}{\norm{\wh}_{1,h}} \leq Ch^{r} \norm{u}_{r+1},
  \end{equation*}
  where $\Vh$ is the nonconforming virtual element space described in Section~\ref{sec:vemSpaces}.
\end{lemma}
\begin{proof}
  We apply the definition of the jump operator $\jump{\cdot}$ and the
  Green's identity under the assumption that $\reg\geq 1$, which
  implies that $\u \in H^2(\D)$, and use the fact that $\Vh\subset
  H^{1,\text{nc}}_{\k}(\Th)$ to obtain
  \begin{align}
    &\abs{\A(\u, \wh) - (\force, \wh)} = \abs{\sum_{\s \in \Edges} \ints (\diff \nabla \u - \frac{1}{2}\u\conv ) \cdot \jump{\wh} \ds} \notag \\
    &\qquad\quad= \abs{\sum_{\s \in \Edges} \ints ( (\diff \nabla \u - \frac{1}{2}\u\conv )-\Pos{\k-1}(\diff \nabla \u - \frac{1}{2}\u\conv)) \cdot (\jump{\wh}-\Pos{0}\jump{\wh}) \ds} \notag.
    % &\qquad\quad\leq \sum_{\s \in \Edges} \abs{\ints (\diff \nabla \u) \cdot \jump{\wh} \ds} + \frac{1}{2}\abs{\ints \u \conv \cdot \jump{\wh} \ds}. \notag\\
  \end{align}
  Using the Cauchy-Schwartz inequality  and then applying the approximation estimates of Theorem~\ref{thm:polynomialApproximation} to bound each of the resulting terms, we obtain, 
  cf.~\cite{NonconformingVEM} or~\cite{CrouzeixRaviart}, 
  \begin{align*}
    \abs{\A(\u, \wh) - (\force, \wh)} \leq Ch^{r} \norm{\u}_{r+1,\E^{+} \cup \E^{-}} \abs{\wh}_{1,\E^{+} \cup \E^{-}},
  \end{align*}
  where for each side $\s$ the symbols $E^{+}$ and $E^{-}$ denote the
  two elements sharing that side, and consequently the lemma holds.
\end{proof}

\begin{theorem}[$H^1$ error bound]
	\label{thm:h1ErrorBound}
	Suppose that Assumptions~\ref{ass:vem}-\ref{ass:meshReg} are satisfied. 
	Let $\k \geq 1$ be a positive integer and let $\u \in H^{\reg+1}(\D)$ be the true solution to problem~\eqref{eq:origVariationalForm} for some positive integer $\reg$.
	Define $r = \min(\k, \reg)$ and suppose that the coefficients $\diff, \conv, \reacSym \in W^{r+1,\infty}(\Omega)$ satisfy \eqref{eq:diffEllipticity} and \eqref{eq:reacBound}. Let $\langle \forceh, \vh\rangle:=\sum_{\E\in\Th}(\forceh,\vh)_\E$, with $\forceh|_\E := \Po{\max(\k-2, 0)} \force|_\E$. 	
	Denote by $\uh \in \Vh$ the corresponding virtual element solution to problem~\eqref{eq:VEMproblem} where $\Vh$ is either the conforming or the nonconforming virtual element space presented in Section~\ref{sec:vemSpaces}.
	Then, there exists a constant $C$ independent of $h$ and $\u$ such that
	\begin{equation*}
		\|\u - \uh\|_1 \leq C h^r%(1 + h) 
		(\norm{u}_{r+1} + %C h^r 
		\norm{\force}_{r-1})
	\end{equation*}
\end{theorem}
\begin{proof}
	We prove the theorem by separately bounding the terms of the Strang-type abstract convergence result of Theorem~\ref{thm:vemStrang}.
    The first term on the right-hand side of~\eqref{eq:vemStrang}, i.e. $\inf_{\vh \in \Vh} \norm{\u - \vh}_{1,h}$, is easily bounded by introducing any interpolant $\uI\in\Vh$ of $u$ as in Theorem~\ref{thm:spaceApproximation}.
        For the second term, first 
        we suppose that $\k \geq 2$. In this case we may apply the definition of the $\LTWO$-projection of $\force$ to find that
	\begin{align*}
		\sup_{\wh \in \Vh} \frac{\abs{\langle\forceh, \wh\rangle - (\force, \wh)}}{\norm{\wh}_{1,h}} &= \sup_{\wh \in \Vh} \frac{\abs{\sum_{\E \in \Th} \displaystyle\intE (\Po{\k-2} \force - \force) \wh \dx}}{\norm{\wh}_{1,h}} \\
			&= \sup_{\wh \in \Vh} \frac{\abs{\sum_{\E \in \Th} \displaystyle\intE (\Po{\k-2} \force - \force) (\wh - \Po{0}\wh) \dx}}{\norm{\wh}_{1,h}} \\
		&\leq \sup_{\wh \in \Vh} \frac{\sum_{\E \in \Th} \norm{\Po{\k-2} \force - \force}_{0,\E} \norm{\wh - \Po{0}\wh}_{0,\E}}{\norm{\wh}_{1,h}} \\
			&\leq C h^{r} \norm{\force}_{r-1}
	\end{align*}
    	having used the bounds of Theorem~\ref{thm:polynomialApproximation}.
	A similar argument applies with $\forceh = \Po{0}\force$ when $\k=1$.
	
	Turning now to the infimum over the polynomial subspace, we first observe that
	\begin{align*}
		\inf_{p \in \P{\Th}{\k}} &\left[\norm{\u - p}_{1, h} + \sum_{\E \in \Th} \sup_{\wh \in \VhE} \frac{ \abs{\AE(p, \wh) - \AhE(p, \wh)}}{\norm{\wh}_{1,\E}}\right] \\
			&\leq \norm{\u - \Po{\k}\u}_{1, h} + \sum_{\E \in \Th} \sup_{\wh \in \VhE} \frac{ \abs{\AE(\Po{\k}\u, \wh) - \AhE(\Po{\k}\u, \wh)}}{\norm{\wh}_{1,\E}}.
	\end{align*}

	To  treat the second term on the right-hand side, we recall the splitting of the bilinear forms into their symmetric and skew-symmetric parts, and bound each part separately.  
	 Here we only detail the bounding of the difference between the symmetric parts  since the skew-symmetric parts can be treated analogously.
	From the definition of $\aE$ and the polynomial consistency property of $\ahE$ (cf., Assumption~(A2)) it follows that 
	\begin{align*}
		\abs{\aE(\Po{\k}\u, \wh) - \ahE(\Po{\k}\u, \wh)} 
            &\leq \abs{\intE \diff \nabla \Po{\k}\u \cdot 
            (\Id-\Po{\k-1})\nabla \wh \dx}
            + \abs{\intE \reac \Po{\k}\u (\Id-\Po{\k})\wh \dx}\\
&= \abs{\intE (\Id - \Po{\k-1})(\diff \nabla \Po{\k}\u) \cdot \nabla \wh \dx} + \abs{\intE (\Id-\Po{\k})(\reac \Po{\k}\u) \wh \dx}\\
			&\leq \norm{\wh}_{1,\E} \left(\norm{(\Id - \Po{\k-1})(\diff \nabla \Po{\k}\u)}_{0,\E}+\norm{(I-\Po{\k})(\reac \Po{\k}\u) }_{0,\E}\right),
	\end{align*}
    and now the results of Theorem~\ref{thm:polynomialApproximation}  and the regularity assumption on $\diff$ and $\reacSym$ implies
	\begin{align*}
		\sup_{\wh \in \VhE} \frac{ \abs{\aE(\Po{\k}\u, \wh) - \ahE(\Po{\k}\u, \wh)}}{\norm{\wh}_{1,\E}} 
        %&\leq \norm{(\Id - \Po{\k-1})(\diff \nabla \Po{\k}\u)}_{0,\E} +\norm{(I-\Po{\k})(\reac \Po{\k}\u) }_{0,\E}
		&\leq C h_{\E}^{r} \norm{\u}_{r+1,\E}.
	\end{align*}
     A similar bound holds for $\abs{\bE(\Po{\k}\u,\wh) - \bhE(\Po{\k}\u,\wh)}$ due to the regularity assumption on $\conv$. Combining these bounds we therefore obtain
	\begin{align*}
		\sup_{\wh \in \VhE} \frac{ \abs{\AE(\Po{\k}\u, \wh) - \AhE(\Po{\k}\u, \wh)}}{\norm{\wh}_{1,\E}} 
        &\leq C h_{\E}^{r} \norm{\u}_{r+1,\E}.
	\end{align*}
	The result then follows by observing that Lemma~\ref{lem:strangNonconformityTerm} provides an optimal order bound for the remaining term relevant to the nonconforming case only.
\end{proof}

\begin{theorem}[$\LTWO$ error bound]
	Suppose that Assumptions~\ref{ass:vem}-\ref{ass:meshReg} are satisfied, and further assume that the domain $\D$ is convex. 
	Let $\k \geq 1$ be a positive integer and let $\u \in H^{\reg+1}(\D)$ be the solution to the problem~\eqref{eq:origVariationalForm} for some positive integer $\reg$.
	Define $r = \min(\k, \reg)$ and suppose that the coefficients $\diff, \conv, \reacSym \in W^{r+1,\infty}(\Omega)$ satisfy \eqref{eq:diffEllipticity} and \eqref{eq:reacBound}. Let $\langle \forceh, \vh\rangle:=\sum_{\E\in\Th}(\forceh,\vh)_\E$, with $\forceh|_\E := \Po{\k-1} \force|_\E$.
	Denote by $\uh \in \Vh$ the corresponding virtual element solution to problem~\eqref{eq:VEMproblem} where $\Vh$ is either the conforming or the nonconforming virtual element space presented in Section~\ref{sec:vemSpaces}.
	Then, there exists a constant $C$ independent of $h$ and $\u$ such that
	\begin{equation*}
		\norm{\u - \uh}_{0} \leq Ch^{r+1} \norm{\u}_{r+1}.
	\end{equation*}
\end{theorem}
\begin{proof}
	Let $\psi \in H^2(\D) \cap H^1_0(\D)$ be the solution to the dual problem
	\begin{equation}
		\label{eq:dualProblem}
		-\nabla \cdot (\diff \nabla \psi) - \conv \cdot \nabla \psi + (\reac - \nabla \cdot \conv) \psi = \u - \uh.
	\end{equation}
	Then, due to the convexity of $\D$, $\psi$ satisfies the regularity bound
	\begin{equation*}
		\norm{\psi}_{2} \leq C \norm{\u - \uh}_{0},
	\end{equation*}
	and consequently for any interpolant $\psiI$ as in Theorem~\ref{thm:spaceApproximation}, we have
	\begin{equation*}
		\norm{\psi - \psiI}_{1,h} \leq Ch\norm{\psi}_{2} \leq C h \norm{\u-\uh}_{0}.
	\end{equation*}
	Multiplying~\eqref{eq:dualProblem} by $\u=\uh$ and integrating, we find that
	\begin{align*}
		\norm{\u - \uh}_{0}^2 &= (\u - \uh, -\nabla \cdot (\diff \nabla \psi) - \conv \cdot \nabla \psi + (\reac - \nabla \cdot \conv) \psi) \\
			&= %\marco{\sum_{\E\in\Th}\AE(\u-\uh, \psi)}
			\A(\u-\uh, \psi)
			 + \sum_{\s \in \Edges} \int_{\s} \left(\diff \nabla \psi - \frac{1}{2} \conv \psi\right) \cdot \jump{\u - \uh} \ds,
	\end{align*}
	and the edge-wise term may be bounded by arguing as in Lemma~\ref{lem:strangNonconformityTerm}, to find that
	\begin{align*}
		\sum_{\e \in \Edges} \int_{\e} \left(\diff \nabla \psi - \frac{1}{2} \conv \psi\right) \cdot \jump{\u - \uh} \ds &\leq Ch \norm{\psi}_{2} \norm{\u - \uh}_{1,h}  \\
			&\leq Ch^{r+1} \norm{\u}_{r+1} \norm{\u - \uh}_{0}.
	\end{align*}

 	To bound the other term, we add and subtract appropriately,
 	\begin{align*}
 		\A(\u-\uh, \psi) &= \A(\u-\uh, \psi - \psiI) + \A(\u-\uh, \psiI) \\
 			&= \A(\u-\uh, \psi - \psiI) + \left(\A(\u, \psiI) - (\force, \psiI)\right) +\\&\quad+ \left(\Ah(\uh, \psiI) - \A(\uh, \psiI)\right) + \left((\force, \psiI) - \langle\forceh, \psiI\rangle\right),
 	\end{align*}
 	obtaining terms which relate to those of the original abstract error bound.	We label these as $T_1$, $T_2$, $T_3$ and $T_f$ respectively, and bound them separately. 
	
	We first observe that $T_1$ may be bounded using the continuity of the variational form and the $H^1$-norm  error bound of Theorem~\ref{thm:h1ErrorBound} as
	\begin{align*}
		T_1:= \A(\u - \uh, \psi - \psiI) &\leq C\norm{\u - \uh}_{1,h} \norm{\psi - \psiI}_{1,h} \\
			&\leq Ch^{r+1} \norm{\u}_{r+1} \norm{\u - \uh}_{0}.
	\end{align*}
	The term $T_2$ measures the nonconformity of the method and may be bounded using Lemma~\ref{lem:strangNonconformityTerm} as
	\begin{align*}
		\abs{T_2} := \abs{\A(\u, \psiI) - (\force, \psiI)} &\leq Ch^{r}\norm{\u}_{r+1}\norm{\psi - \psiI}_{1,h} \\
			&\leq Ch^{r+1} \norm{\u}_{r+1} \norm{\u - \uh}_{0}.
	\end{align*}
	Using the definition of the $\LTWO$-projection, we can rewrite $T_f$ as
	\begin{align*}
		T_f :&= (\force, \psiI) - \langle\forceh, \psiI\rangle = \sum_{\E\in\Th}(\force - \Po{\k-1}\force , \psiI)_E = \sum_{\E\in\Th}(\force - \Po{\k-1}\force, \psiI - \Po{0}\psiI)_E \\
			&\leq \sum_{\E\in\Th}\norm{\force - \Po{\k-1}\force}_{0,E} \norm{\psiI - \Po{0}\psiI}_{0,E} \\
			&\leq C\sum_{\E\in\Th} h^{r} \norm{\force}_{r,E} h \norm{\psiI}_{1,E} \leq Ch^{r+1} \sum_{\E\in\Th}\norm{\force}_{r,E} \norm{\psi}_{2,E} \\
			&\leq Ch^{r+1} \norm{\force}_{r} \norm{\u - \uh}_{0}.
	\end{align*}
	Finally we turn to the inconsistency term $T_3$, namely
	\begin{align*}
		T_3 :&= \Ah(\uh, \psiI) - \A(\uh, \psiI) = \sum_{\E \in \Th} \AhE(\uh, \psiI) - \AE(\uh, \psiI) \\
			&= \sum_{\E \in \Th} \left(\AhE(\uh - \Po{\k} \u, \psiI - \Po{1}\psi) - \AE(\uh - \Po{\k} \u, \psiI - \Po{1}\psi)\right) + \\
				&\quad\qquad+ \left(\AhE(\Po{\k}\u, \psiI) - \AE(\Po{\k}\u, \psiI)\right) + \left(\AhE(\uh, \Po{1}\psi) - \AE(\uh, \Po{1}\psi)\right)
	\end{align*}
	The first difference can then easily be bounded using the fact that both the variational form and the VEM bilinear form are continuous in the $H^1$ norm, so
	\begin{align*}
		\AhE(\uh - \Po{\k} \u, \psiI - \Po{1}\psi) - \AE(\uh - \Po{\k} \u, \psi - \Po{1}\psi) &\leq C \norm{\uh - \Po{\k} \u}_{1,\E} \norm{\psiI - \Po{1}\psi}_{0, \E} \\
			&\leq Ch^{r+1} \norm{\u}_{r+1,\E} \norm{\u - \uh}_{0,\E}.
	\end{align*}
	The bound for the other two differences is obtained by splitting each bilinear form up into its constituent terms and applying the definition of polynomial consistency.
	For the diffusion terms, the polynomial consistency property means we consider
	\begin{align*}
		\ahE(\Po{\k}\u, \psiI) - \aE(\Po{\k}\u, \psiI) &= \intE \diff \nabla \Po{\k}\u \cdot  \left(\Po{\k-1} - \Id\right)\nabla \psiI \dx \\
			&= \intE \left(\Po{\k-1} - \Id\right) \left(\diff \nabla \Po{\k}\u \right) \cdot \nabla \left(\psiI - \psi\right) \dx +\\ 
				&\qquad+ \intE \left(\Po{\k-1} - \Id\right) \left(\diff \nabla \Po{\k}\u\right) \cdot \nabla \left(\psi - \Po{1} \psi\right) \dx \\
			&\leq Ch^{r+1} \norm{\u}_{r+1,\E} \norm{\u - \uh}_{0,\E},
	\end{align*}
	having applied the Cauchy-Schwarz inequality and the polynomial approximation bounds in the final step.
	The second difference is similarly treated by adding and subtracting terms and applying the Cauchy-Schwarz inequality and the polynomial approximation bounds along with the regularity of the dual solution $\psi$ and the $H^1$-norm error bound
	\begin{align*}
		\ahE(\uh, \Po{1}\psi) - \aE(\uh, \Po{1}\psi) &= \intE \diff \left(\Po{\k-1} - \Id\right) \nabla  \uh \cdot \nabla \Po{1} \psi \dx \\
			&= \intE \diff \left( \Po{\k-1} - \Id \right) \nabla (\uh - \u) \cdot \nabla \left( \Po{1} - \Id \right) \psi \dx +\\
				&\qquad+ \intE \diff \left( \Po{\k-1} - \Id \right) \nabla \u \cdot \nabla \left( \Po{1} - \Id \right) \psi \dx +\\
				&\qquad+ \intE \nabla \left(\uh - \Po{\k}\u \right) \cdot \left(\Po{\k-1} - \Id \right) \left(\diff \nabla \psi \right) \dx \\
			&\leq Ch^{r+1} \norm{\u}_{r+1,\E} \norm{\u - \uh}_{0,\E}.
	\end{align*}
	The bounds for the other components of the bilinear form in these differences are treated completely analogously. Consequently, we may combine these individual bounds to determine the optimal order bound in the statement of the theorem.
\end{proof}

\section{The Bilinear Forms}
\label{sec:bilinearForms}
In this section, we introduce a choice of the virtual element bilinear forms $\ahE$ and $\bhE$ that satisfy the abstract properties presented in Section~\ref{sec:vemFramework}. 
We remark that the bilinear forms that we pick are exactly the same regardless of whether we are considering the conforming or the nonconforming method.
Moreover, as described in Section~\ref{subsec:imple}, the implementation of the two methods differs only in the practical construction of the $\LTWO$-projection operators due to the 
different choice of the degrees of freedom.

\begin{definition}
	\label{def:admissibleStabilisingTerm}
	Let $\E\in\Th$. A computable (see \dref{def:computable}) bilinear form $\StaE:\VhE / \PE{\k}\times\VhE / \PE{\k}\rightarrow \Re$ is said to be a local \emph{admissible} stabilising bilinear form if it is symmetric, positive definite and it satisfies
	\begin{equation*}
		c_0\aE(\vh, \vh)\le\StaE(\vh, \vh)\le c_1\aE(\vh, \vh)\qquad\forall\vh\in \VhE / \PE{\k},
	\end{equation*}
	for some constants $c_0$ and $c_1$ independent of $\E$ and $h$.
\end{definition}

Given an admissible stabilising bilinear form $\StaE(\cdot, \cdot)$ and a computable projection $\stabProj{\k} : \VhE \rightarrow \PE{\k}$,  we simply define
\begin{equation}
	\ahE(\uh, \vh) := 
    (\diff \Po{\k-1} \nabla \uh, \Po{\k-1} \nabla \vh)_{\E} +
    (\reacSym \Po{\k} \uh, \Po{\k} \vh) 
    + \StaE((\Id - \stabProj{\k})\uh, (\Id - \stabProj{\k})\vh),
    \label{eq:ahE:def}
\end{equation}
and
\begin{equation}
	\bhE(\uh, \vh) := \frac{1}{2} \left[ (\conv \cdot \Po{\k-1}\nabla \uh, \Po{\k}\vh) - (\Po{\k} \uh, \conv \cdot \Po{\k-1} \nabla \vh)\right].
    \label{eq:bhE:def}
\end{equation}
It is clear that any bilinear form $\AhE$ resulting from~\eqref{eq:ahE:def} and~\eqref{eq:bhE:def} with $\StaE$ admissible satisfies Assumption~\ref{ass:bilinearForms}, cf.~\cite{BasicsPaper}.
Moreover, all of the terms in this bilinear form are computable since the construction and degrees of freedom of the space allow us to compute $\Po{\k}\vh$ and $\Po{\k-1}\nabla \vh$ for any $\vh \in \VhE$, while $\StaE$ and $\stabProj{\k}$ are computable by assumption.
The projection $\stabProj{\k}$ could be chosen in many ways.
A common choice would be to use an operator which is already computed, such as the $\LTWO(\E)$-orthogonal projection $\Po{\k}$ or the projection $\spaceProj{\k}$ used in the definition of the space.
Another option could be to choose $\stabProj{\k}$ in conjunction with the stabilising term in such a way as to incorporate important new features into the virtual element method, such as monotonicity, positivity and maximum/minimum principles for the numerical solutions, like the mimetic finite difference stabilising terms described in~\cite{MFD-Monotonicity}. This topic, however, is beyond the scope of the present paper and will be considered for future works.

We now introduce a choice of admissible stabilising bilinear form. This is essentially the one already used in~\cite{BasicsPaper}.
\begin{proposition}
	\label{prop:stabilisingTerms}
	Let $\overline{\diff}_\E$, $\overline{\nabla \cdot \conv}_{\E}$ and $\overline{\reac}_\E$ be some constant approximations 
        %%\olivercomment{or not approximations, since we only need the local scaling} 
        of $\diff$, $\nabla \cdot \conv$ and $\reacSym$ over $\E$ respectively.
	Then, the bilinear form
	\begin{align*}
		\StaE(\vh, \wh) &:= (\overline{\diff}_\E h_\E^{\spacedim-2} - \frac{1}{2} \overline{\nabla \cdot \conv}_{\E} h^{\spacedim - 1} + \overline{\reac}_\E h_\E^{\spacedim} ) \sum_{r=1}^{\NE} \dof_r(\vh) \dof_r(\wh), 
	\end{align*}
	for $\vh, \wh \in \VhE / \PE{\k}$, is admissible.
\end{proposition}

\begin{proof}
	The stabilising term $\StaE$ is an inner product over the finite dimensional space $\Re^{\NE - \dim(\PE{\k})}$ of vectors of degrees of freedom, which is isomorphic to $\VhE / \PE{\k}$.
	Since $\aE$ is an inner product on $\VhE$ and thus also on $\VhE / \PE{\k}$, the existence of the constants $c_0$ and $c_1$ of Definition~\ref{def:admissibleStabilisingTerm} follows from the equivalence of the norms induced by these inner products.
    
    The fact that these constants are independent of $h$ is due to the fact that $\StaE$ scales the same as $\aE$.
	It is clear that the $H^1$ part of $\aE$ scales like $\overline{\diff}_\E h^{\spacedim-2}$ while the $\LTWO$ term $((\reac - \frac{1}{2} \nabla \cdot \conv) \vh, \wh)_{\E}$ scales like $\overline{\reacSym}_\E h^{\spacedim} - \frac{1}{2} \overline{\nabla \cdot \conv}_{\E} h^{\spacedim - 1}$.
    Then, since the degrees of freedom are specifically chosen to scale like 1 (cf.~\cite{Hitchhikers}), the coefficient at the front of $\StaE$ ensures that the term has the correct scaling even when one of the coefficients $\diff$, $\nabla \cdot \conv$ or $\reacSym$ locally degenerates.
\end{proof}

\begin{remark}
	The stabilising term in Proposition~\ref{prop:stabilisingTerms} is just one of a family of admissible stabilising terms, defined as  appropriately scaled inner products on the subspace of the degrees of freedom relating to functions in $\VhE / \PE{\k}$.
	Here we have chosen the Euclidean inner product for simplicity.
\end{remark}

\subsection{The Effects of Numerical Integration}
\label{subset:effect:numerical:integration}
In any practical implementation of the method, the
coefficients $\diff, \conv$, and $\reacSym$, have to be approximated, meaning that the polynomial consistency properties of Assumption~\ref{ass:bilinearForms} will hold in an approximate way in general. 
One possibility, which we assess here, is to utilise numerical quadratures. Crucially, such \emph{numerical quadratures will only affect the consistency terms}.
These have precisely the same structure of standard finite elements terms (integral products of polynomial trial and test functions weighted by the coefficients), and as such the variational crime introduced by their approximation can be assessed using the classical finite element analysis.
Indeed, within this section we show that the actual implemented methods retain the stability and optimal accuracy properties of the theoretical virtual element methods proposed above, provided that the quadrature scheme used is of at least polynomial order $2\k-2$.
We emphasise that this is \emph{exactly the same requirement as for the classical finite element methods} used to solve the same problem (cf.~\cite{Ciarlet}).

In more concrete terms, suppose that we are  approximating integrals over the element $\E$ using a quadrature rule $Q^{\E}_{\quadDegree}$ of degree $\quadDegree$, so
\begin{equation*}
	\intE g \dx \approx Q^{\E}_{\quadDegree}(g) := \sum_{\ell = 1}^{L} \omega_{\ell} g(q_{\ell}),
\end{equation*}
for a finite set of quadrature points $\{q_{\ell}\}_{\ell=1}^{L}$ and associated weights $\{\omega\}_{\ell=1}^{L}$. Thus in practice, the implementation of the method will be based on the perturbed bilinear form
\begin{align}
\label{eqn:quadratureBF}
\begin{split}
	\ahEquad(\uh, \vh) 
    &:= \sum_{\ell = 1}^{L} \omega_{\ell} \,\diff(\ql) (\Po{\k-1} \nabla \uh)(\ql) \cdot (\Po{\k-1} \nabla \vh)(\ql) \\
    &\quad + \sum_{\ell = 1}^{L} \omega_{\ell} \,\reacSym(\ql) (\Po{\k}\uh)(\ql) \, (\Po{\k}\vh)(\ql) \\[0.5em] 
    &\quad + (\overline{\diff}_\E h_\E^{\spacedim-2} - \frac{1}{2} \overline{\nabla \cdot \conv}_{\E} h^{\spacedim - 1} +\overline{\reacSym}_\E h_\E^{\spacedim} ) 
    \sum_{r=1}^{\NE}\dof_r((\Id - \stabProj{\k})\uh) \dof_r((\Id - \stabProj{\k})\vh),
\end{split}
\end{align}
and similarly for the skew-symmetric part $\bhEquad$. Once more, we note that the use of quadrature only affects the consistency term. The following two theorems show that the use of an appropriate quadrature rule does not affect either the stability or the accuracy of the method.

\begin{theorem}
	\label{thm:quadratureStability}
	Suppose the quadrature scheme $Q^{\E}_{\quadDegree}$ with $\quadDegree \geq 2\k - 2$ has strictly positive weights and is exact for the space $\PE{2\k-2}$ and/or the set $\{q_{\ell}\}_{\ell=1}^{L}$ of quadrature points contains a $\PE{\k-1}$ unisolvent subset.
	Let $\Ahquad$ denote the bilinear form $\Ah$ with the polynomial consistency integrals approximated using the quadrature scheme $Q^{\E}_{\quadDegree}$.
	Then, $\Ahquad$ satisfies the stability property of Assumption~\ref{ass:bilinearForms}.
\end{theorem}

\begin{proof}
	We first wish to show that the bilinear form $\ahEquad$ defines a norm on $\VhE$, or on $\VhE / \PE{0}$ when $\reacSym \equiv 0$, equivalent to the norm imposed by $\aE$ in either case.
	Suppose $\reacSym \not\equiv 0$.
	Then $\ahEquad$ is clearly already a semi-norm and all that remains to be shown is that $\ahEquad(\vh, \vh) = 0 \Rightarrow \vh = 0$.
	
	Let $\ahEquad(\vh, \vh) = 0$.
	Then, we must have 
	\begin{equation*}
		Q^{\E}_{\quadDegree} \big((\Po{\k-1} \nabla \vh) \cdot (\diff\Po{\k-1} \nabla \vh)\big) = 0.
	\end{equation*}
	By the assumptions on $Q^{\E}_{\quadDegree}$ and the strong ellipticity of $\diff$ (see~\eqref{eq:diffEllipticity}), this implies that $\Po{\k-1} \nabla \vh$ = 0, and consequently we may deduce that either (a), $\vh \in \PE{0}$ or (b), $\vh \in \VhE / \PE{\k}$ where, with a slight abuse of notation, we associate $\VhE / \PE{\k}$ with the non-polynomial subspace of $\VhE$.
	
	Suppose case (a) holds, so $\vh \in \PE{0}$.
	Then we also have 
	\begin{equation*}
		0 = Q^{\E}_{\quadDegree}(\reacSym (\Po{\k}\vh)^2) = Q^{\E}_{\quadDegree}(\reacSym \vh^2),
	\end{equation*}
	since $\Po{\k}$ is the identity on $\PE{0}$, and thus we deduce that $\vh \equiv 0$.
	
	Alternatively, suppose that case (b) holds, so $\vh \in \VhE / \PE{\k}$.
	Then, since $\ahEquad(\vh, \vh) = 0$, it follows that 
	\begin{equation*}
		\StaE((\Id - \stabProj{\k})\vh, (\Id - \stabProj{\k})\vh) = 0,
	\end{equation*}
	and we may deduce that $\vh \equiv 0$.
	
	From this, we may conclude that $(\ahEquad(\cdot, \cdot))^{\frac{1}{2}}$ is a norm on $\VhE$.
	Moreover, $(\aE(\cdot, \cdot))^{\frac{1}{2}}$ is also a norm on $\VhE$ and since this is a finite dimensional subspace of $H^1(\E)$, the resulting norms are equivalent.
	As with the bilinear form $\ahE$, the constants in the equivalence are independent of $h$ due to the correct scaling of $\ahEquad$.
	
	On the other hand, when $\reacSym \equiv 0$ we find that $(\ahEquad(\cdot, \cdot))^{\frac{1}{2}}$ and $(\aE(\cdot, \cdot))^{\frac{1}{2}}$ are both norms on $\VhE / \PE{0}$ and both zero on $\PE{0}$.
	Again, the fact that $\VhE / \PE{0}$ is finite dimensional allows us to deduce that the two norms are equivalent.
	
	The stability property for $\bhEquad$ also holds because $\bhEquad(\vh, \vh) = \bhE(\vh, \vh) = 0$ and it is straightforward to check that
    $\bhEquad(\uh, \vh) \leq C\norm{\conv}_{\infty} \norm{\uh}_{1,\E} \norm{\vh}_{1,\E}$.
\end{proof}

The next result addresses the questions about the accuracy of the method when the quadrature is employed to evaluate the integrals of the virtual bilinear forms, and should be compared with Theorem~\ref{thm:h1ErrorBound}.
\begin{theorem}
	\label{thm:quadratureAccuracy}
	Suppose that Assumptions~\ref{ass:vem}-\ref{ass:meshReg} are
        satisfied. Let $\k \geq 1$ be a positive integer and let $\u
        \in H^{s+1}(\D)$ be the true solution to
        problem~\eqref{eq:origVariationalForm} for some positive
        integer $s$. Define $r = \min(\k, s)$ and suppose that the
        coefficients $\diff, \conv, \reacSym \in
        W^{r+1,\infty}(\Omega)$, satisfying \eqref{eq:diffEllipticity}
        and \eqref{eq:reacBound}. Suppose that the right-hand side
        function $\force \in H^{r-1}(\D)$ is approximated by 
        $\forceh := \Po{\max(\k-2, 0)} \force$.
	
	Suppose that the quadrature scheme $Q^{\E}_{\quadDegree}$ with $m \geq 2\k-2$ is exact for the space $\PE{2\k-2}$.
	Let $\Ahquad$ denote the virtual element bilinear form $\Ah$ obtained by approximating the integrals using the quadrature scheme $Q^{\E}_{\quadDegree}$ 
    and $\uhquad \in \Vh$ be the solution obtained from this scheme.
	Then, there exists a positive constant $C$, independent of $h$ and $\u$ such that
	\begin{equation*}
		\norm{\u - \uhquad}_{1,h} \leq C h^r(1+h)\norm{\u}_{r+1}+C h^r \norm{\force}_{r-1}.
	\end{equation*}
\end{theorem}
\begin{proof}
	Expanding as in the proof of Theorem~\ref{thm:vemStrang}, it may be shown that the only extra term depending on the quadrature scheme which arises in the abstract error bound is
	\begin{equation*}
		\sum_{\E \in \Th}\sup_{\wh \in \VhE} \frac{\abs{\AhE(\uProj, \wh) - \AhEquad(\uProj, \wh)}}{\norm{\wh}_{1,\E}},
	\end{equation*}
	where $\uProj := \Po{\k}\u$.
	As usual, we split this term into the different components of the bilinear form and bound them separately.
	We give here the bound for $\ahE(\uProj, \wh) - \ahEquad(\uProj, \wh)$; the bound for the skew-symmetric term follows analogously.
	Since the stabilising term is unaffected by the quadrature, we only need to bound
	\begin{align*}
		&\sup_{\wh \in \VhE} \frac{\abs{Q^{\E}_{\quadDegree} ((\diff \nabla \uProj) \cdot \Po{\k-1} \nabla \wh) 
          - \displaystyle\intE (\diff \nabla \uProj) \cdot \Po{\k-1} \nabla \wh \dx}}{\norm{\wh}_{1,\E}}
        \\
        &\qquad+
        \sup_{\wh \in \VhE} \frac{\abs{Q^{\E}_{\quadDegree} ((\reacSym \uProj) \cdot \Po{\k}\wh) 
        - \displaystyle\intE (\reacSym \uProj) \cdot \Po{\k} \wh \dx}}{\norm{\wh}_{1,\E}}
	\end{align*}
	Arguing as in Theorem~4.1.4 in~\cite{Ciarlet} and using the stability of the $\LTWO(\E)$ projector, we find that
	\begin{align*}
		\abs{Q^{\E}_{\quadDegree} ((\diff \nabla \uProj) \cdot \Po{\k-1} \nabla \wh) - \intE (\diff \nabla \uProj) \cdot \Po{\k-1} \nabla \wh \dx}
			&\leq Ch^r \abs{\u}_{r+1,\E} \norm{\wh}_{1,\E},
	\end{align*}
    and 
    \begin{align*}
		\abs{Q^{\E}_{\quadDegree} ((\reacSym \uProj) \cdot \Po{\k} \wh) - \intE (\reacSym \uProj) \Po{\k} \wh \dx}
			&\leq Ch^r \abs{\u}_{r,\E} \norm{\wh}_{1,\E}.
	\end{align*}
	The theorem then follows by treating the skew-symmetric term similarly and combining the result with the original $H^1$-norm bound of Theorem~\ref{thm:h1ErrorBound}.
\end{proof}

A similar analysis can be carried out to control the error in the
numerical approximation $\forceh$ of the forcing function $\force$,
and to recover the optimal order of convergence in the $\LTWO$-norm, see always~\cite{Ciarlet}.

\section{Implementation}
\label{subsec:imple}

To fix the definition of the space $\VhE$, we must first define the computable projection $\spaceProj{\k} : \biglocalspace /_{\sim} \rightarrow \PE{\k}$. The original approach in~\cite{EquivalentProjectors}  is to contruct a particular projector, in that case an elliptic projection therein denoted by $\Pnk$. 
Here instead we first construct the general family of projections and subsequently fix a simple choice.

A projection $\spaceProj{\k}:\biglocalspace /_{\sim}\rightarrow \PE{\k}$ may be defined using any computable inner product 
$\bilBP{\cdot}{\cdot}$ on $\biglocalspace /_\sim \PE{\k}$.  For any $\vh \in \biglocalspace /_{\sim}$, we define the polinomial $\spaceProj{\k} \vh \in \PE{\k}$ as the solution of
\begin{equation}
	\bilBP{\spaceProj{\k} \vh}{\ma} = \bilBP{\vh}{\ma} \qquad \forall \ma \in \PE{\k}.
	\label{eq:abstractProjectionProblem}
\end{equation}

Let $\{\bigbasis_i\}_{i=1}^{\NE}$ be the Lagrangian basis functions of $\biglocalspace /_{\sim}$ with respect to the original degrees of freedom in \dref{def:originalDofs}, and define the matrix $\matD$ such that $\matD_{i \alpha} = \dof_{i}(\ma)$.
Since $\biglocalspace /_{\sim}$ is finite dimensional, $\bilBP{\cdot}{\cdot}$ can be written as the symmetric positive definite matrix $\matB = (\bilBP{\bigbasis_i}{\bigbasis_j})$.
From this, it can also be seen that $(\bilBP{\ma}{\bigbasis_i}) = \matD^T \matB$ and $(\bilBP{\ma}{\m_{\beta}}) = \matD^T \matB \matD$.

Define the action of $\spaceProj{\k}$ on the shape functions $\{\psi_i\}_{i = 1}^{\NE}$ through the matrix $\matPiBkP$, where
\begin{align*}
	\spaceProj{\k} \bigbasis_i=\sum_{\alpha=1}^{N_{d,k}}\ms_{\alpha}(\matPiBkP)_{\alpha i},
\end{align*}
so that the $j$-th column of $\matPiBkP$ contains the coefficients of the expansion of the polynomial $\spaceProj{\k} \psi_j$ in the monomial basis $\{\ms_{\alpha}\}_{\alpha=1}^{N_{d,k}}$.
Consequently, the projection problem~\eqref{eq:abstractProjectionProblem} can be written in the matrix form
\begin{align*}
  \matPiBkP = \big(\matD^T\matB\matD\big)^{-1} \matD^T\matB
\end{align*}
The matrix $\matD^T\matB\matD$ is invertible due to the assumption that $\bilBP{\cdot}{\cdot}$ is an inner product on $\biglocalspace /_{\sim}$.

Finally, we have to choose the bilinear form $\mathcal{B}$ to fix the projection.
This is equivalent to picking any symmetric positive definite matrix $\matB$. Here we choose $\matB = \Id$, yielding the simple choice
\begin{equation}\label{eq:Dproj}
 \matPiBkP = \big(\matD^T\matD\big)^{-1} \matD^T.
\end{equation}

Having thus chosen the projection $\spaceProj{\k}$ and therefore the space $\VhE$, we can introduce the Lagrangian basis $\{\uphi_i\}_{i=1}^{\NE}$ of $\VhE$ associated with the degrees of freedom in \dref{def:originalDofs}.
These shape functions are necessarily different for the conforming and the nonconforming methods, although the implementation of the two methods is formally the same since we are only concerned with the degrees of freedom.

We shall determine the local matrix $\AhE(\uphi_i,\uphi_j)$ and right-hand side vector $(\forceh, \uphi_j)_\E$ associated to the Lagrangian basis introduced above. To this end, we first need to evaluate the projections $\Po{\k} \uphi_i$ and $\Po{\k-1} \nabla \uphi_i$, for all $i=1,\dots,\NE$.

The polynomial $\Pzk \phi_i$ is the solution of the projection problem:
\begin{align}
  \big(\upol_{\alpha},\Pzk \uphi_i\big)_{\E} &= (\upol_{\alpha}\,\uphi_{i})_{\E}
  \qquad\forall\alpha=1,\ldots,N_{d,k}.
  \label{eq:L2:proj:problem}
\end{align}
Since $\Pzk \uphi_{i} $ is both a polynomial of degree $\k$ and a function in the
virtual element space $\VhE$, it can be expanded on the monomials
generating $\PE{\k}$
and the shape functions generating $\VhE$ as
\begin{align*}
  \Pzk \uphi_{i} 
  =\sum_{\alpha=1}^{N_{d,k}} \upol_{\alpha}\big(\matPzk\big)_{\alpha i}
  =\sum_{j=1}^{\NE} \uphi_{j}\big(\matPz\big)_{ji},
\end{align*}
and the coefficients of these expansions are collected in the matrices
$\matPzk$ and $\matPz$, respectively.
The matrix $\matPz$ will be used at the end of the subsection to compute
the stabilising term.
By comparison, it follows that
$\matPz=\matD\matPzk$.

%% Matrix $\matC$
We reformulate the projection problem~\eqref{eq:L2:proj:problem} in
matrix form as $\matH\matPzk=\matC$, where the coefficients of
matrices $\matH$ and $\matC$ are given by
\begin{align*}
  \matC_{\alpha i} = (\upol_{\alpha},\uphi_{i})_{\E}
  \qquad\textrm{and}\qquad
  \matH_{\alpha\beta} = (\upol_{\alpha}\,\upol_{\beta})_{\E}
\end{align*}
for $\alpha=1,\ldots,N_{d,k}$ and $i=1,\ldots,\NE$.
%%
%% Computability issue, enhancement 
Since the space $\VhE$ does not use the extra degrees of freedom in \dref{def:extraDofs}, the matrix $\matC$ must be constructed in two parts according to the definition of the space.
Then, we have 
\begin{align}
  \matC_{\alpha i} =
  \begin{cases}
    (\upol_{\alpha},\uphi_{i})_{\E}
    & \textrm{if~}\ms_\alpha\in\ME{\k-2},\\[1em]
    ( \upol_{\alpha}, \spaceProj{\k} \uphi_{i})_{\E} & 
    \textrm{if~}\ms_\alpha\in\MEstar{\k-1}\cup\MEstar{\k}.
  \end{cases} 
  \label{eq:matC:def-general}
\end{align}
With the choice of $\spaceProj{\k}$ presented above, $\matC$ becomes
\begin{align}
  \matC_{\alpha i} =
  \begin{cases}
    (\upol_{\alpha},\uphi_{i})_{\E}
    & \textrm{if~}\ms_\alpha\in\ME{\k-2},\\[1em]
    \big( \matH(\matD^T\matD)^{-1}\matD^T \big)_{\alpha i} & 
    \textrm{if~}\ms_\alpha\in\MEstar{\k-1}\cup\MEstar{\k},
  \end{cases} 
  \label{eq:matC:def}
\end{align}
where the moments of $\phi_i$ are simply degrees of freedom, so $\matC$ is fully computable.

%% Formal definition
The other crucial term which must be computed is $\Po{\k-1}\nabla\uphi_{i}$, where the projection is defined componentwise such that
\begin{align*}
  (\Pzkk\nabla\uphi_{i}, \bm{\ma})_{\E} &= (\nabla\uphi_{i}, \bm{\ma})_{\E}
  \qquad\forall \bm{\ma} \in (\ME{\k-1})^2 \notag \\
  	&= \intdE \bm{\ma} \cdot \n \uphi_{i} \ds - (\uphi_{i}, \nabla \cdot \bm{\ma})_{\E}.
  	%\label{eq:computingPoGrad}
\end{align*}
%%
%% Computability issue
The second term on the right-hand side of this expression is simply a combination of the internal degrees of freedom of $\uphi_{i}$.
For the nonconforming method, or for the conforming method when $\spacedim = 2$, the first term is also just a combination of edge degrees of freedom of $\uphi_{i}$.
However, for the conforming method when $\spacedim = 3$, we must compute this term using the $\LTWO$-orthogonal projection of $\uphi_i$ on each face $s \subset \dE$ as
\begin{equation*}
	\intdE \bm{\ma} \cdot \n \uphi_{i} \ds = \sum_{\s \subset \dE} \ints \bm{\ma} \cdot \n \Pos{\k}\uphi_{i} \ds.
\end{equation*}
The projection $\Pos{\k}$ can be computed on each face of $\E$ exactly as when $\spacedim = 2$.

Thus, we end up with $\spacedim$ linear systems to compute the projection of the $\spacedim$ components of $\Po{\k-1}\nabla\uphi_{i}$, namely $G\matPox{\k-1}{\ell} = R^{x_{\ell}}$ where $G_{\alpha \beta} = (\ma, \m_{\beta})_{\E}$ with $\ma, \m_{\beta} \in \ME{\k-1}$ 
and
\begin{equation*}
	(R^{x_{\ell}})_{\alpha i} = \sum_{\s \subset \dE} \ints \ma n_{\ell} \uphi_{i} \ds - (\uphi_{i}, \frac{\partial\ma}{\partial{x_{\ell}}} )_{\E},
\end{equation*}
for $\ell = 1,\dots,\spacedim$, where $n_{\ell}$ is the $\ell$-th component of $\n$.
Then, we have that
\begin{align*}
  \Pzkk\frac{\partial\uphi_i}{\partial x_{l}}
  = \sum_{\alpha=1}^{N_{d,k-1}}\upoh_{\alpha}\big(\matPox{\k-1}{\ell} \big)_{\alpha i}.
\end{align*}

Having the matrices $\matPzk$ and $\matPox{\k-1}{\ell}$, $l=1,\dots,d$, we are able to easily implement all the terms in the local bilinear form $\AhE$.
Indeed, for the term of $\ahE$ that contains the diffusion
coefficient $\diff_{|_{\ell n}}$ we have:
\begin{align*}
  \Big( \diff_{|_{\ell n}}\Pzkk\frac{\partial\uphi_i}{\partial x_\ell},
  \Pzkk\frac{\partial\uphi_j}{\partial x_n} \Big)_{\E}
  =
  \sum_{\alpha,\beta=1}^{N_{d,k-1}}
  (\diff_{|_{\ell n}}\,\upoh_{\alpha},\,\upoh_{\beta})_{\E}
  \big(\matPox{\k-1}{\ell} \big)_{\alpha i}
  \big(\matPox{\k-1}{n} \big)_{\beta j}
  \qquad \ell,n=1,\dots,d.
\end{align*}
For the reaction term in $\ahE$, we easily find that:
\begin{align*}
  (\reacSym\Pzk\uphi_i,\Pzk\uphi_j)_{\E}
  =\sum_{\alpha,\beta=1}^{N_{d,k}}
  (\reacSym\,\upol_{\alpha},\upol_{\beta})_{\E}
  \big(\matPzk\big)_{\alpha i}
  \big(\matPzk\big)_{\beta j}.
\end{align*}
For the skew-symmetric bilinear form $\bhE$, first notice that  
$  \conv\cdot\Pzkk\nabla\uphi_j =
  \sum_{l=1}^{d}
  \beta_l\Pzkk\frac{\partial\uphi_j}{\partial x_l},
$
where $\beta_l$ is the $l$-th component of $\conv$.
Therefore, the first term of $\bhE$ is given by:
\begin{align*}
  \int_{\E}\conv\Pzk\uphi_i\cdot\Pzkk\nabla\uphi_j\dV 
=\sum_{l=1}^{d}\sum_{\alpha=1}^{N_{d,k}}\sum_{\beta=1}^{N_{d,k-1}} (\beta_l\,\upoh_{\alpha},\,\upoh_{\beta})_{\E}
  \big(\matPzk\big)_{\alpha i}
  \big(\matPox{\k-1}{l} \big)_{\beta j},
\end{align*}
and a similar expression is found for the second term by
exchanging $i$ and $j$. 

Finally, according to the expressions given in
Proposition~\ref{prop:stabilisingTerms}, the stabilising term is given
by
\begin{align*}
  \StaE((\Id - \Po{\k})\uphi_i, (\Id - \Po{\k})\uphi_j) =
  \big(\overline{\diff}_\E h_\E^{\spacedim-2} - \frac{1}{2} \overline{\nabla \cdot \conv}_{\E} h^{\spacedim - 1} + \overline{\reac}_\E h_\E^{\spacedim} \big) 
  \left( \big(\matI-\matPz\big)^T\big(\matI-\matPz\big) \right)_{ij}
\end{align*}
since 
$\dof_r((\Id - \Po{\k})\uphi_i)=\big(\matI-\matPz\big)_{ir}$.

Similarly, since we have the projector $\Po{\k}$ at our disposal, we might as well use  $\forceh:=\Po{\k}\force|_\E$ to approximate $\force|_\E$. In this case the right-hand side vector is given by:
\begin{align}\label{eq:VEMrhs}
(\forceh, \uphi_j)_\E=\intE\Po{\k}\force\uphi_j\dV=\intE\force\Po{\k}\uphi_j\dV
=\sum_{\alpha=1}^{N_{d,k}} (\force,\upol_{\alpha})_\E\big(\matPzk\big)_{\alpha i}.
\end{align}
In practice, the $\LTWO$-products on the right-hand side of the above formulas have to be somehow approximated in accordance with the theory presented in Section~\ref{subset:effect:numerical:integration}.

\section{Numerical Results}
\label{sec:numerics}

\newcommand{\MeshONE}  {$\mathcal{M}_1$}
\newcommand{\MeshTWO}  {$\mathcal{M}_2$}
\newcommand{\MeshTHREE}{$\mathcal{M}_3$}
\newcommand{\HAT}[1]{\widehat{#1}}

All the numerical experiments presented in this section are obtained 
using~\eqref{eq:VEMrhs} for the approximation of the right-hand side, the choice given by~\eqref{eq:Dproj} for $\spaceProj{\k}$,  setting $\stabProj{\k}=\Po{\k}$, and using~\eqref{eq:matC:def} for the definition of matrix $\matC$.
However, a comparison with the implementation using
$\stabProj{\k}=\Pnk$ and $\spaceProj{\k}=\Pnk$ (cf.~\cite{EquivalentProjectors})
in~\eqref{eq:matC:def-general} did not reveal any significant 
difference in the behaviour of the method.

The numerical experiments  are aimed to
confirm the \emph{a priori} analysis developed in the previous
sections.
In a preliminary stage, the consistency of both the conforming and
nonconforming VEM, i.e. the exactness of these methods for
polynomial solutions, has been tested numerically by solving the
elliptic equation with boundary and source data determined by
$\u(x,y)=x^m+y^m$ on different set of polygonal meshes and for $m=1$
to $4$.
In all the cases, we measure an error whose magnitude is of the order
of the arithmetic precision, thus confirming this property.

To study the accuracy of the method we solve the
convection-reaction-diffusion equation on the domain
$\Omega=]0,1[\times]0,1[$.
The variable coefficients of the equation are given by
\begin{align*}
  &\diff(x,y) = \left(
    \begin{array}{cc}
      1+y^2 & -xy\sin(2\pi x)\sin(2\pi y) \\
      -xy\sin(2\pi x)\sin(2\pi y) & 1+x^2
    \end{array}
  \right),\\[1em]
  &\conv(x,y) = \left(
    \begin{array}{c}
      -2 \, ( x    + 2y^2 - 1 ) \\
      3  \, ( 3x^2 - 2y   + 3 )  
    \end{array}
  \right),
  \qquad 
  \reac(x,y) = x^2 + y^3 + 1.
\end{align*}
The forcing term and the Dirichlet boundary condition are set in
accordance with the exact solution
\begin{align*}
  \u(x,y) = \sin(2\pi x)\sin(2\pi y) + x^5 + y^5.
\end{align*}

The performance of the methods presented above are investigated by evaluating the rate of
convergence on three different sequences of five meshes, labeled
by~\MeshONE{}, \MeshTWO{} and \MeshTHREE{}, respectively.
The top panels of Fig.~\ref{fig:meshes} show the first mesh of each
sequence and the bottom panels show the mesh of the first refinement.
\begin{figure}[t]
  \centering
  \begin{tabular}{ccc}
    \includegraphics[width=0.32\textwidth]{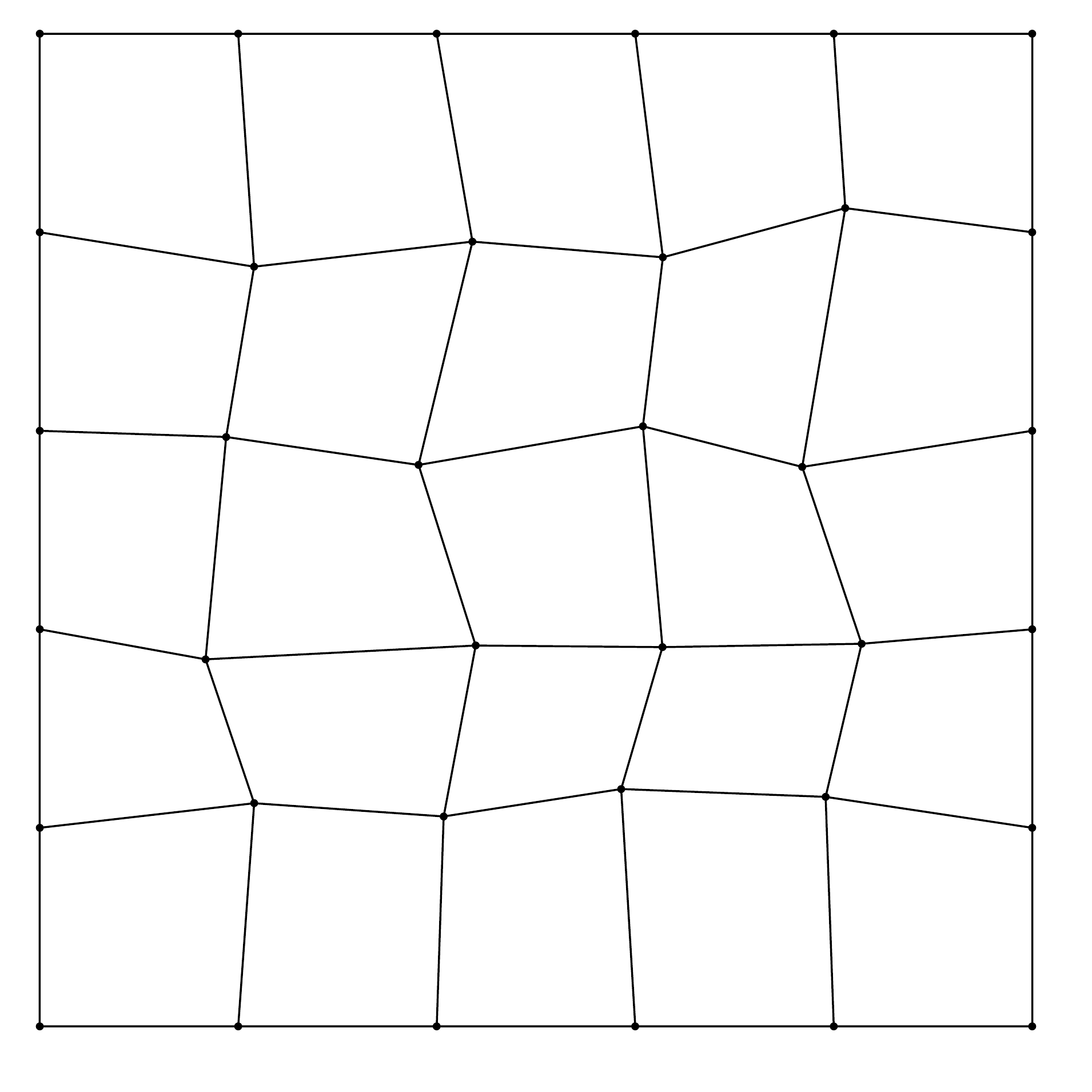}&\hspace{-0.25cm}%\quad
    \includegraphics[width=0.32\textwidth]{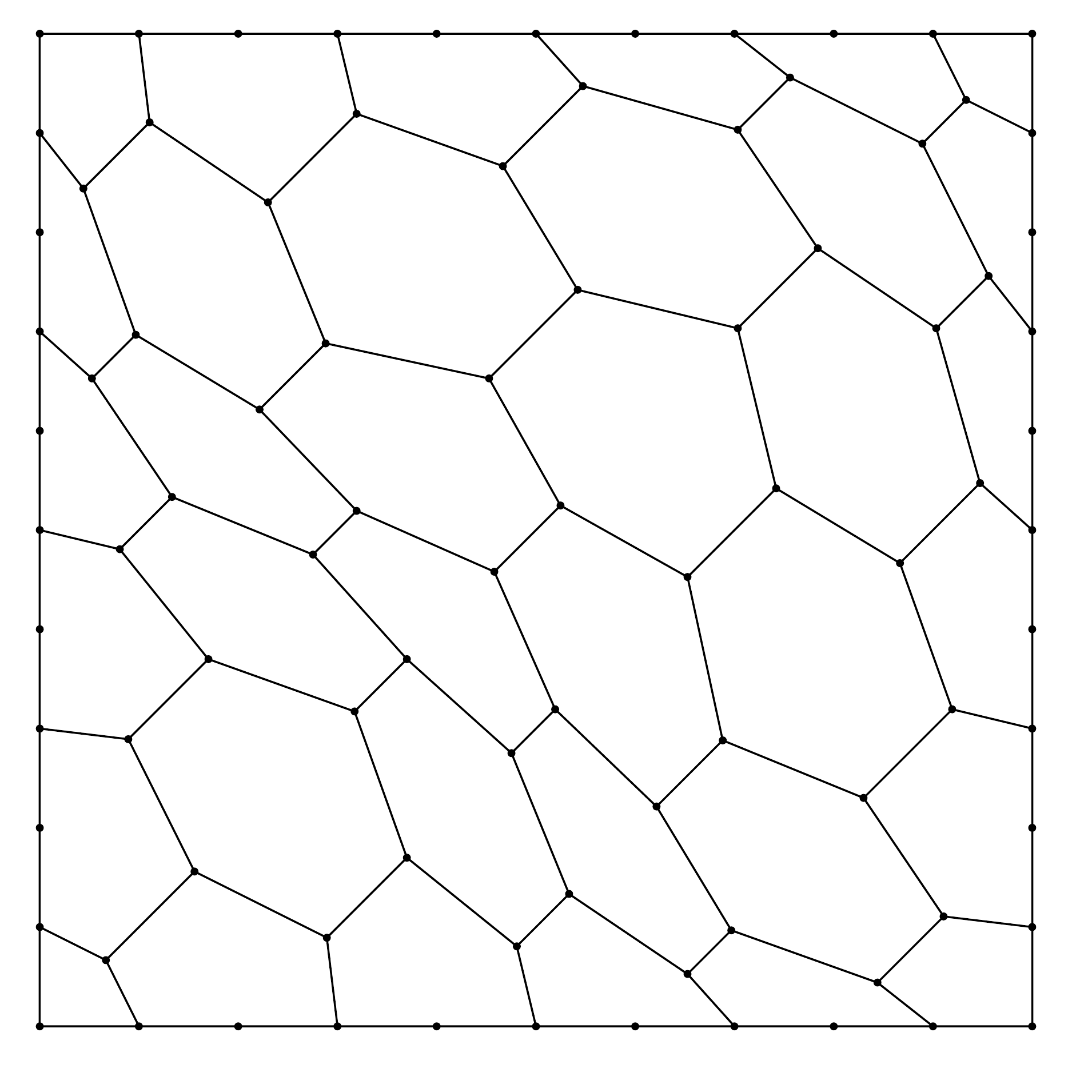}&\hspace{-0.25cm}%\quad
    \includegraphics[width=0.32\textwidth]{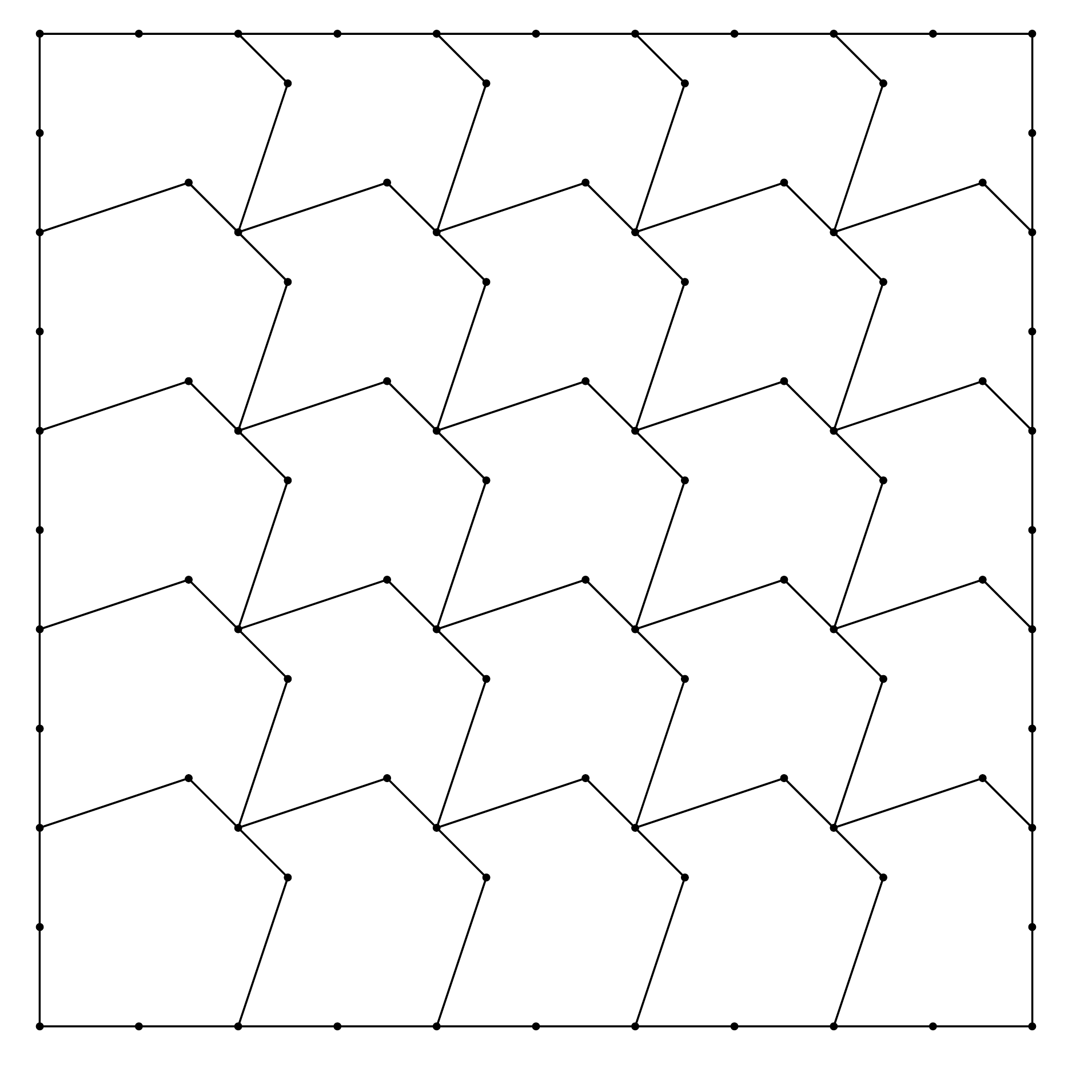}\\
    %% ---------
    \includegraphics[width=0.32\textwidth]{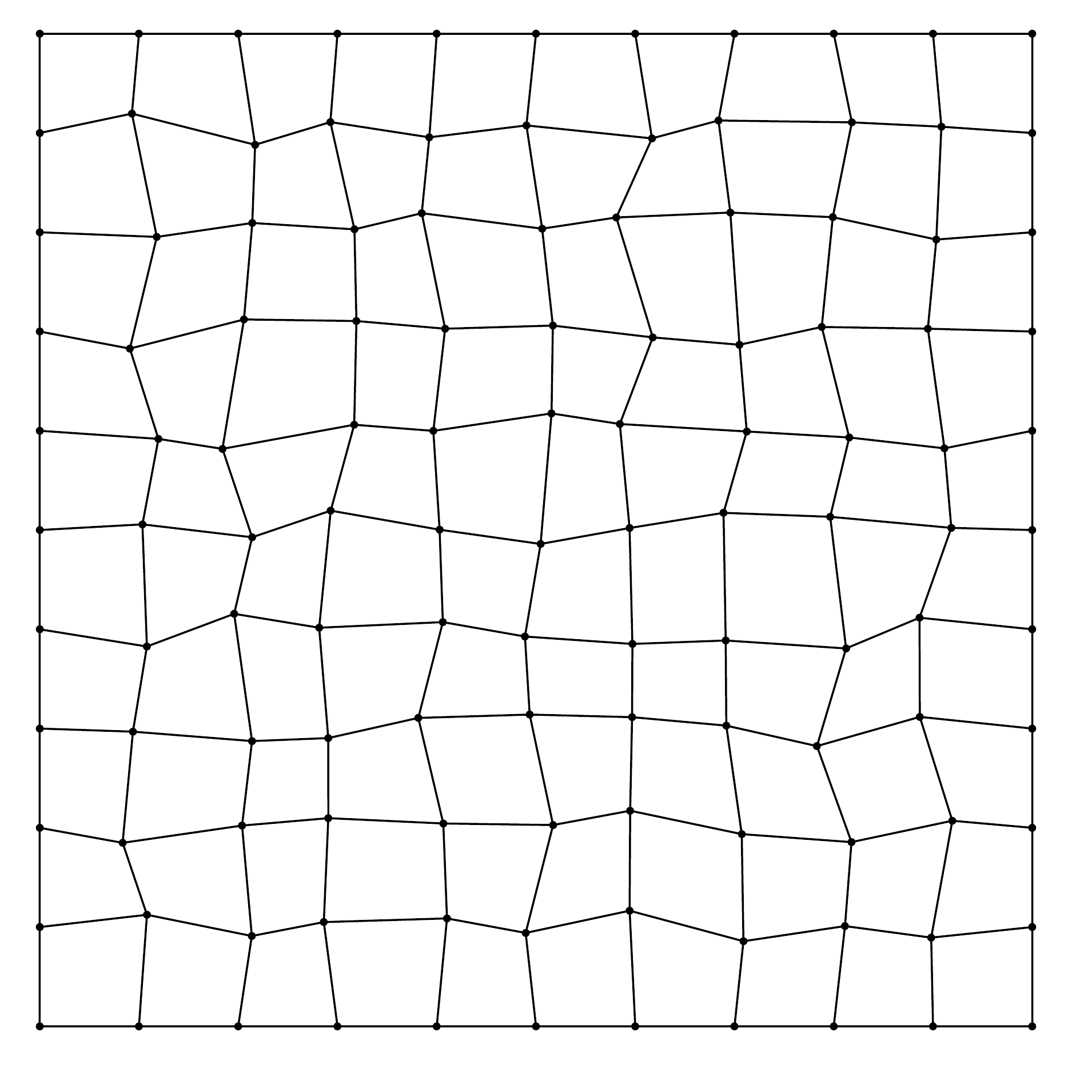}&\hspace{-0.25cm}%\quad
    \includegraphics[width=0.32\textwidth]{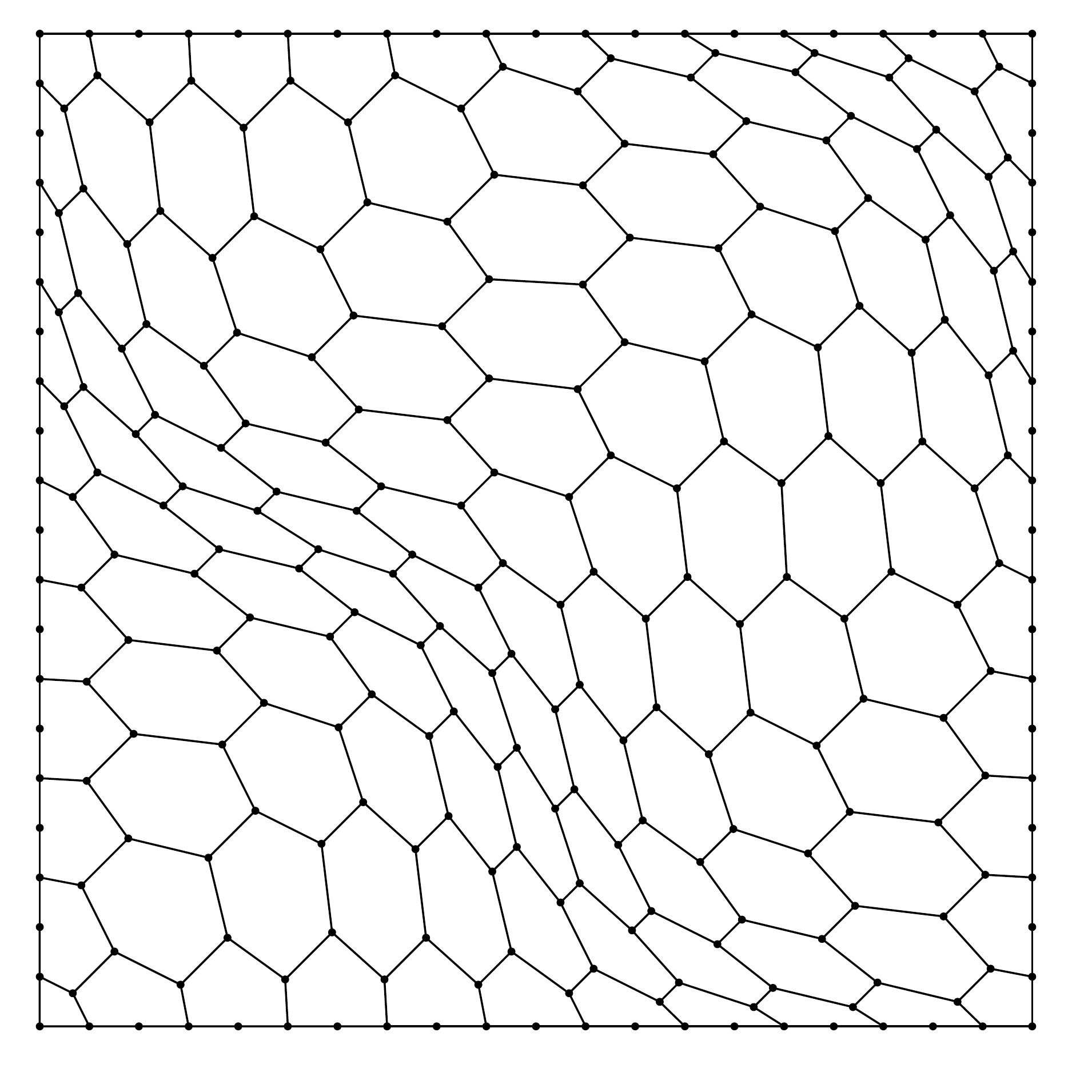}&\hspace{-0.25cm}%\quad
    \includegraphics[width=0.32\textwidth]{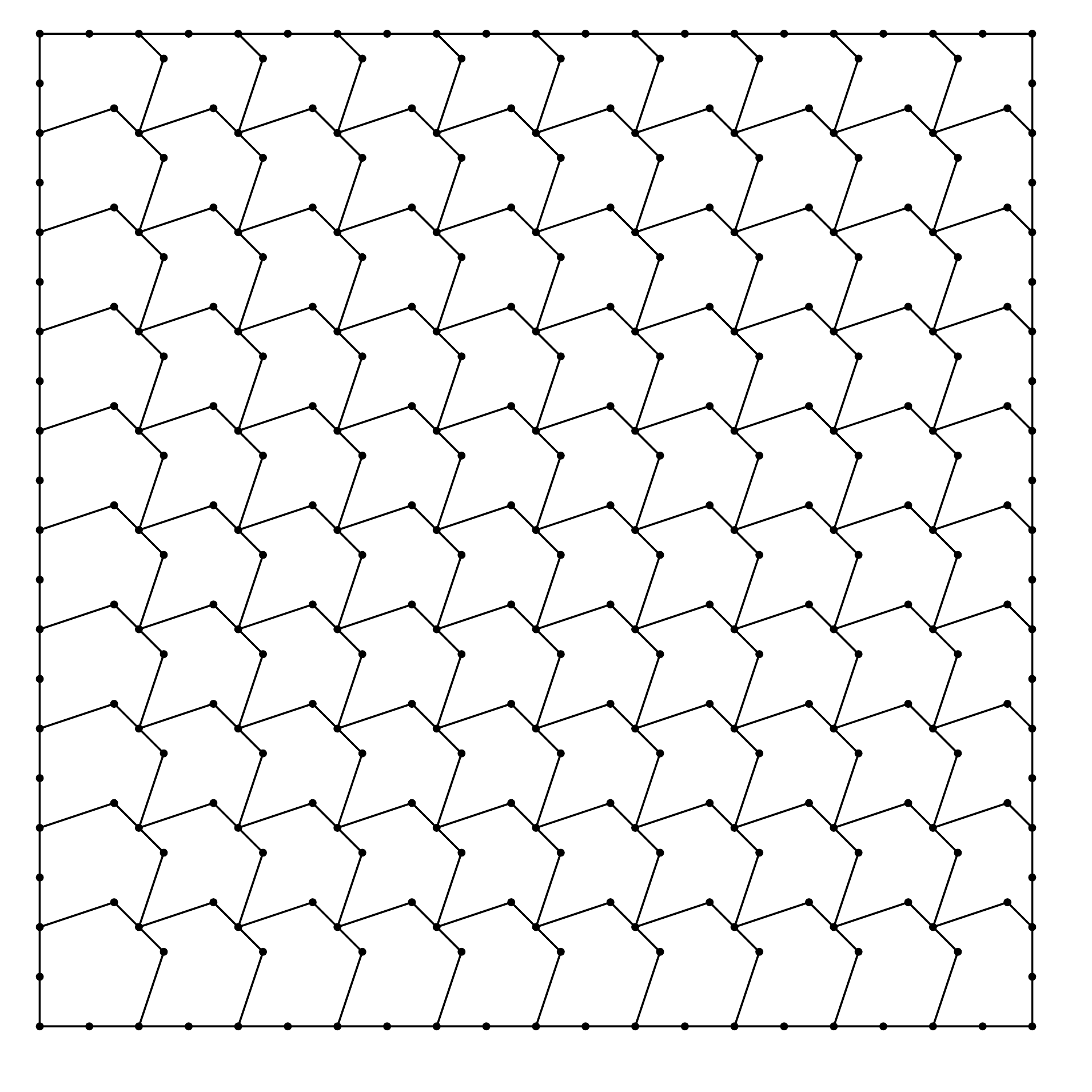}\\
  \end{tabular}
  \caption{First (top) and second (bottom) mesh of the three mesh families.}
  \label{fig:meshes}
\end{figure}
The meshes in \MeshONE{} are built by partitioning the domain $\Omega$
into square cells and relocating each interior node to a random
position inside a square box centered at that node.
The sides of this square box are aligned with the coordinate axis and
their length is equal to $0.8$ times the minimum distance between two
adjacent nodes of the initial square mesh.
The meshes in \MeshTWO{} are built as follows.
%%%
First, we determine a primal mesh by remapping the position
$(\HAT{x},\HAT{y})$ of the nodes of an uniform square partition of
$\Omega$ by the smooth coordinate transformation~\cite{Kuznetsov-Lipnikov-Shashkov:2004}:
\begin{align*}
  x &= \HAT{x} + (1\slash{10}) \sin(2\pi\HAT{x})\sin(2\pi\HAT{y}),\\
  y &= \HAT{y} + (1\slash{10}) \sin(2\pi\HAT{x})\sin(2\pi\HAT{y}).
\end{align*}
%%%
The corresponding mesh of \MeshTWO{} is built from the primal mesh by
splitting each quadrilateral cell into two triangles and connecting
the barycenters of adjacent triangular cells by a straight segment.
%%%
The mesh construction is completed at the boundary by connecting the
barycenters of the triangular cells close to the boundary to the
midpoints of the boundary edges and these latters to the boundary
vertices of the primal mesh.
The meshes in \MeshTHREE{} are obtained by filling the unit square
with a suitably scaled non-convex octagonal reference cell.

All the meshes are parametrised by the number of partitions in each
direction.
The starting mesh of every sequence is built from a $5\times 5$
regular grid, and the refined meshes are obtained by doubling this
resolution.
Mesh data for each refinement level, i.e., numbers of mesh elements,
number of edges, number of vertices, are reported in
Table~\ref{tab:mesh:data}.
\renewcommand{\TABROW}[5]{$#1$ & $#2$ & $#3$ & $#4$ & $#5$\\}
\newcommand{\ilev}{n}
\newcommand{\nE}{\mathcal{N}_{\E}}
\newcommand{\nF}{\mathcal{N}_{\s}}
\newcommand{\nv}{\mathcal{N}_{v}}

\begin{table}
  \begin{center}
    \begin{tabular}{|c|ccc|c|ccc|c|ccc|c|}
      \hline %% quads
      &\multicolumn{4}{|c|}{Randomised quadrilaterals}
&\multicolumn{4}{|c|}{Remapped hexagons}
&\multicolumn{4}{|c|}{Non-convex octagons}\\
      \hline
      {$\ilev$}&{$\nE$}&{$\nF$}&{$\nv$}&{$h$}&{$\nE$}&{$\nF$}&{$\nv$}&{$h$}&{$\nE$}&{$\nF$}&{$\nv$}&{$h$}\\
      \hline
      {$1$}&{$  25$}&{$   60$}&{$  36$}&{$0.331$}
      &{$  36$}&{$  125$}&{$   90$}&{$0.328$}
	  &{$  25$}&{$  120$}&{$   96$}&{$0.291$}\\
      {$2$}&{$ 100$}&{$  220$}&{$ 121$}&{$0.186$}
	  &{$ 121$}&{$  400$}&{$  280$}&{$0.185$}
	  &{$ 100$}&{$  440$}&{$  341$}&{$0.146$}\\
      {$3$}&{$ 400$}&{$  840$}&{$ 441$}&{$0.094$}
	  &{$ 441$}&{$ 1400$}&{$  960$}&{$0.097$}
	  &{$ 400$}&{$ 1680$}&{$ 1281$}&{$0.073$}\\
      {$4$}&{$1600$}&{$ 3280$}&{$1681$}&{$0.047$}
	  &{$1681$}&{$ 5200$}&{$ 3520$}&{$0.049$}
	  &{$1600$}&{$ 6560$}&{$ 4961$}&{$0.036$}\\
      {$5$}&{$6400$}&{$12960$}&{$6561$}&{$0.024$}
      &{$6561$}&{$20000$}&{$13440$}&{$0.025$}
      &{$6400$}&{$25920$}&{$19521$}&{$0.018$}\\
      \hline
    \end{tabular}
  \end{center}
  \caption{Mesh data for the meshes in \MeshONE{}, \MeshTWO{}, and \MeshTHREE{};
    $\nE$, $\nF$ and $\nv$ are the numbers of mesh elements, interfaces
    and vertices, respectively, and $h$ is the mesh size parameter.}
  \label{tab:mesh:data}
\end{table}
Approximation errors are measured by comparing the polynomial
quantities $\Po{k}\uh$ and $\Po{k-1}\nabla\uh$, which are obtained
by a post-processing of the numerical solution, with the exact
solution $\u$ and solution's gradient $\nabla\u$.
The relative errors for the approximation of solution $\u$ and its
gradient in function of the mesh size $h$ are shown in the log-log plots of Fig.~\ref{fig:errors:mesh1}
for the mesh sequence \MeshONE{}, Fig.~\ref{fig:errors:mesh2} for the
mesh sequence \MeshTWO{}, and Fig.~\ref{fig:errors:mesh3} for the mesh
sequence \MeshTHREE{}. 
%% Mesh M3-Rnd-Hex, comparison for low order methods, L2 norm
\begin{figure}[t]
  \centering
  \begin{tabular}{cc}
      %% randomized quads
      \begin{overpic}[scale=0.375]{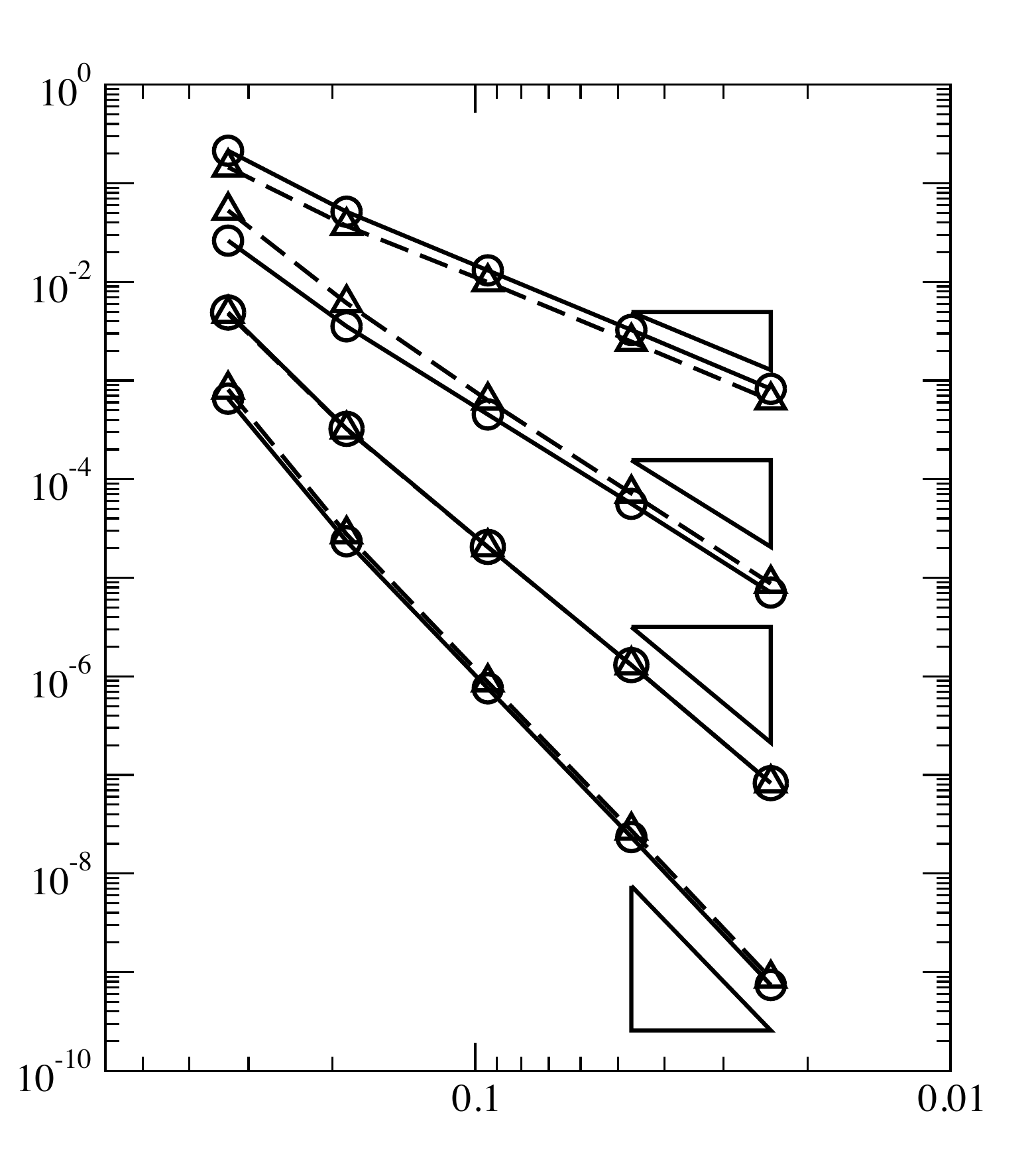} 
        \put(67,69) {$\mathbf{2}$}
        \put(67,55) {$\mathbf{3}$}
        \put(67,40) {$\mathbf{4}$}
        \put(49,15) {$\mathbf{5}$}
        \put(28,0){\textbf{Mesh size $h$}}
        \put(-4,18){\begin{sideways}\textbf{$\LTWO$ Approximation errors}\end{sideways}}
      \end{overpic}
      & %%\qquad
      \begin{overpic}[scale=0.375]{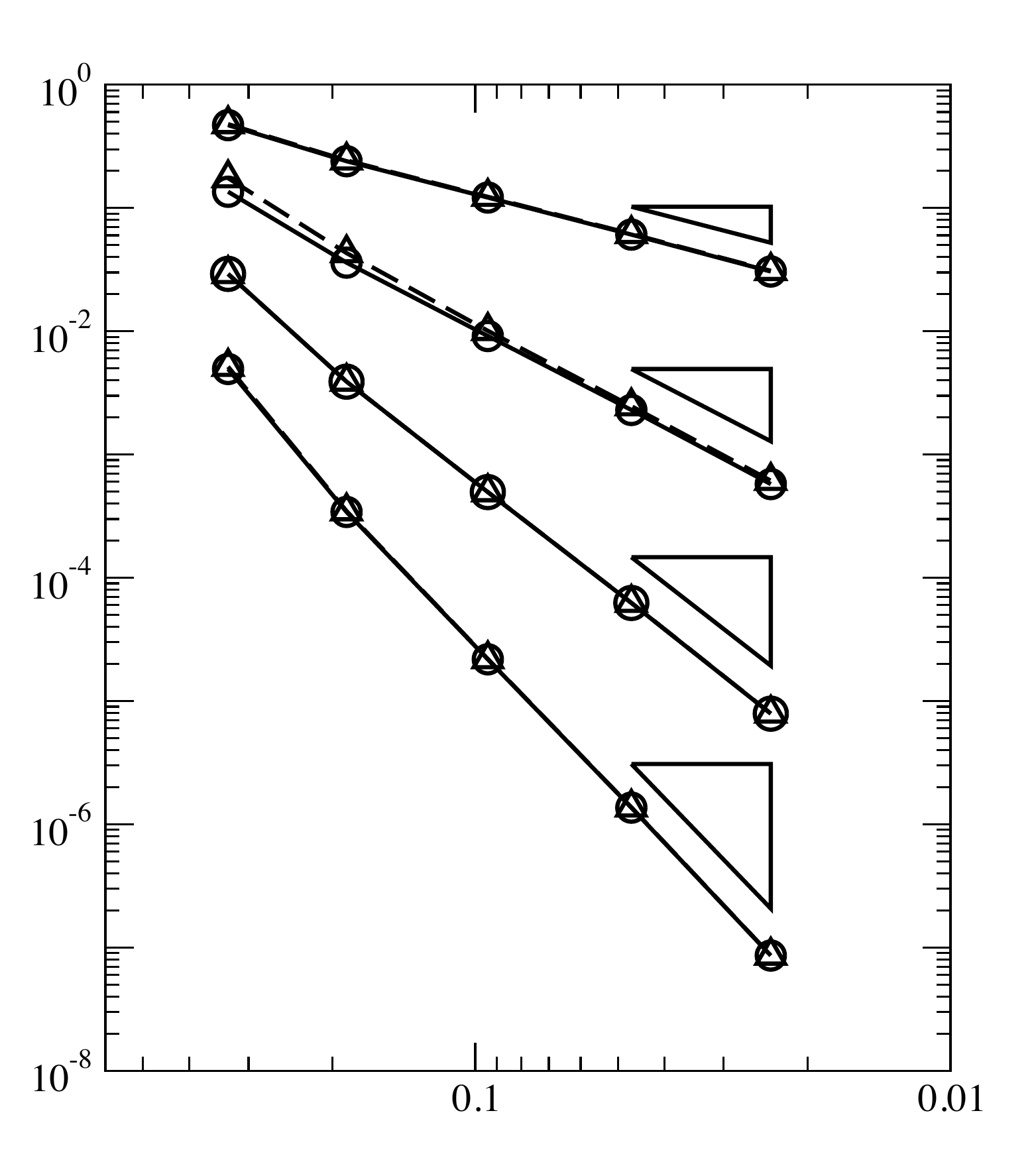} 
        \put(67,80) {$\mathbf{1}$}
        \put(67,64) {$\mathbf{2}$}
        \put(67,47) {$\mathbf{3}$}
        \put(67,28) {$\mathbf{4}$}
        \put(28,0){\textbf{Mesh size $h$}}
        \put(-4,18){\begin{sideways}\textbf{$\HONE$ Approximation errors}\end{sideways}}
      \end{overpic}\\
  \end{tabular}
  \caption{
    Error curves for the conforming VEM (circles) and the non-conforming (triangles) 
    applied to the mesh family of 
    randomised quadrilaterals with $\k=1,2,3$, and $4$.
    The left panels show the relative $\LTWO$ error; 
    the right panels show the relative $\HONE$ errors.
    The expected slopes are indicated by triangles.}
  \label{fig:errors:mesh1}
\end{figure}
\begin{figure}[t]
  \centering
  \begin{tabular}{cc}
      %% hexagons
      \begin{overpic}[scale=0.375]{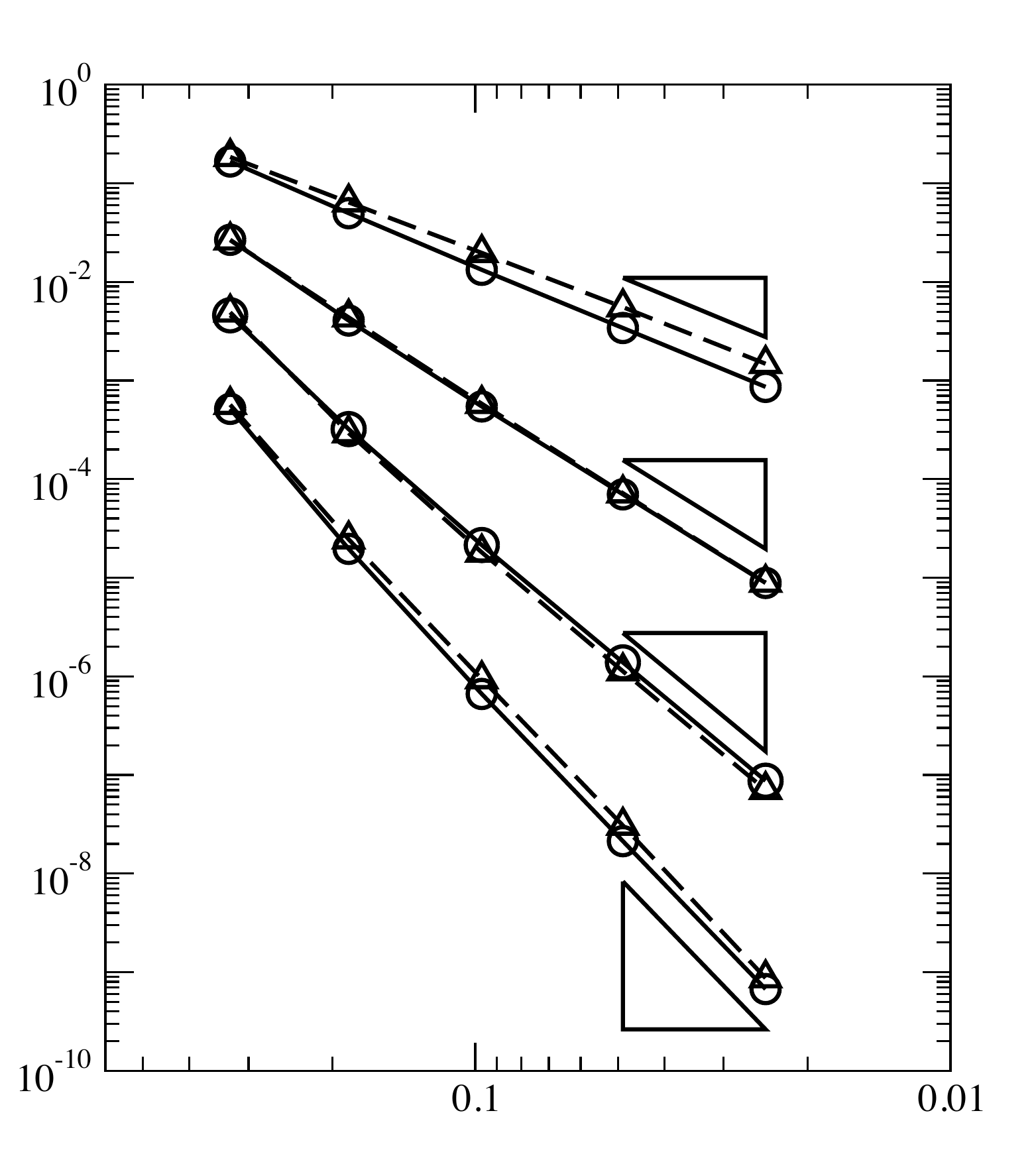} 
        \put(67,72) {$\mathbf{2}$}
        \put(67,55) {$\mathbf{3}$}
        \put(67,40) {$\mathbf{4}$}
        \put(48,15) {$\mathbf{5}$}
        \put(28,0){\textbf{Mesh size $h$}}
        \put(-4,18){\begin{sideways}\textbf{$\LTWO$ Approximation errors}\end{sideways}}
      \end{overpic}
      & %%\qquad
      \begin{overpic}[scale=0.375]{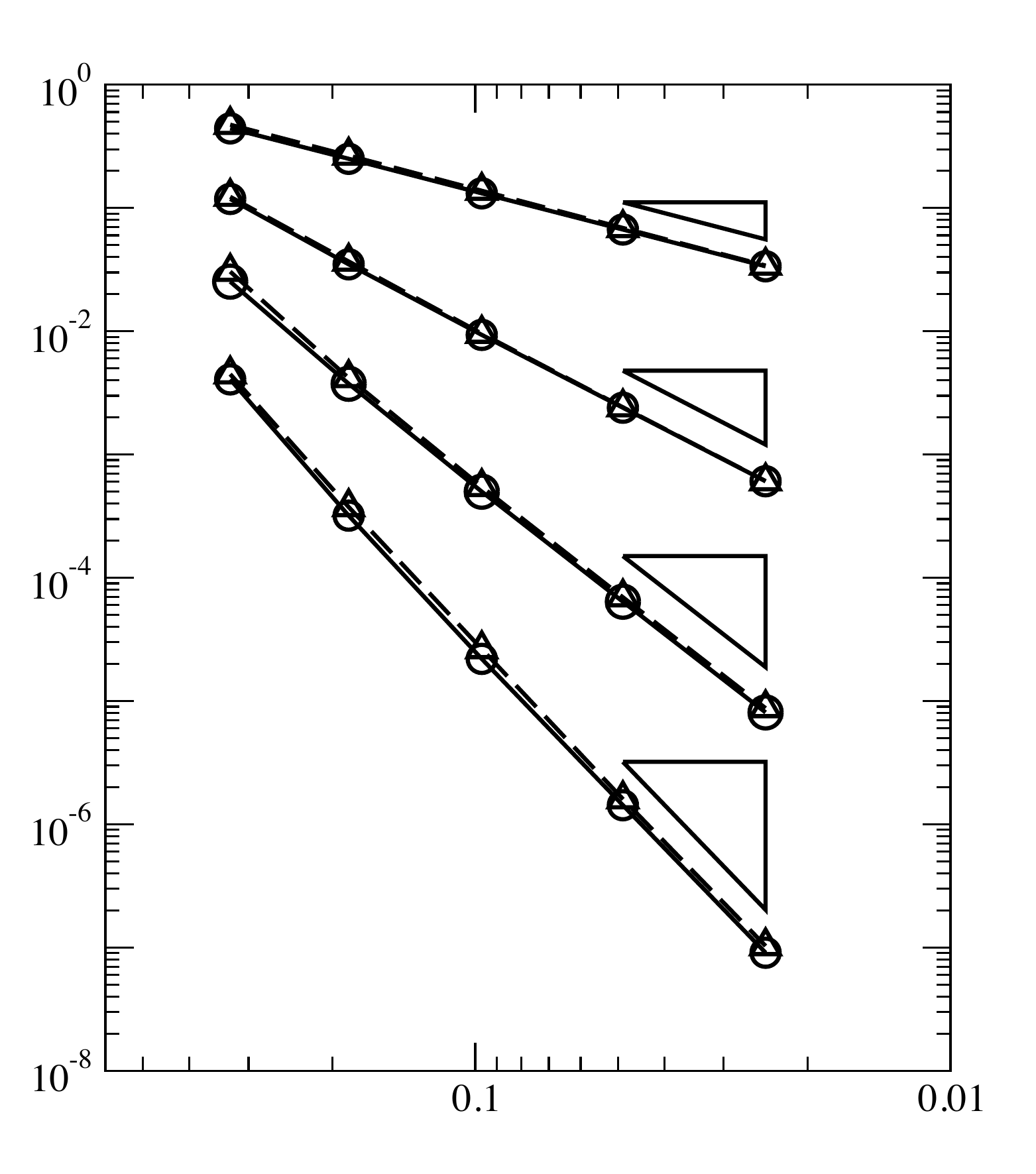} 
        \put(67,80) {$\mathbf{1}$}
        \put(67,63) {$\mathbf{2}$}
        \put(67,46) {$\mathbf{3}$}
        \put(67,27) {$\mathbf{4}$}
        \put(28,0){\textbf{Mesh size $h$}}
        \put(-4,18){\begin{sideways}\textbf{$\HONE$ Approximation errors}\end{sideways}}
      \end{overpic}\\
  \end{tabular}
  \caption{
    Error curves for the conforming VEM (circles) and the non-conforming (triangles)
    applied to the mesh family of mainly hexagonal cells with $\k=1,2,3$, and $4$.
    The left panels show the relative $\LTWO$ error; 
    the right panels show the relative $\HONE$ errors.
    The expected slopes are indicated by triangles.}
  \label{fig:errors:mesh2}
\end{figure}
\begin{figure}[t]
  \centering
  \begin{tabular}{cc}
      %% octagons
      \begin{overpic}[scale=0.375]{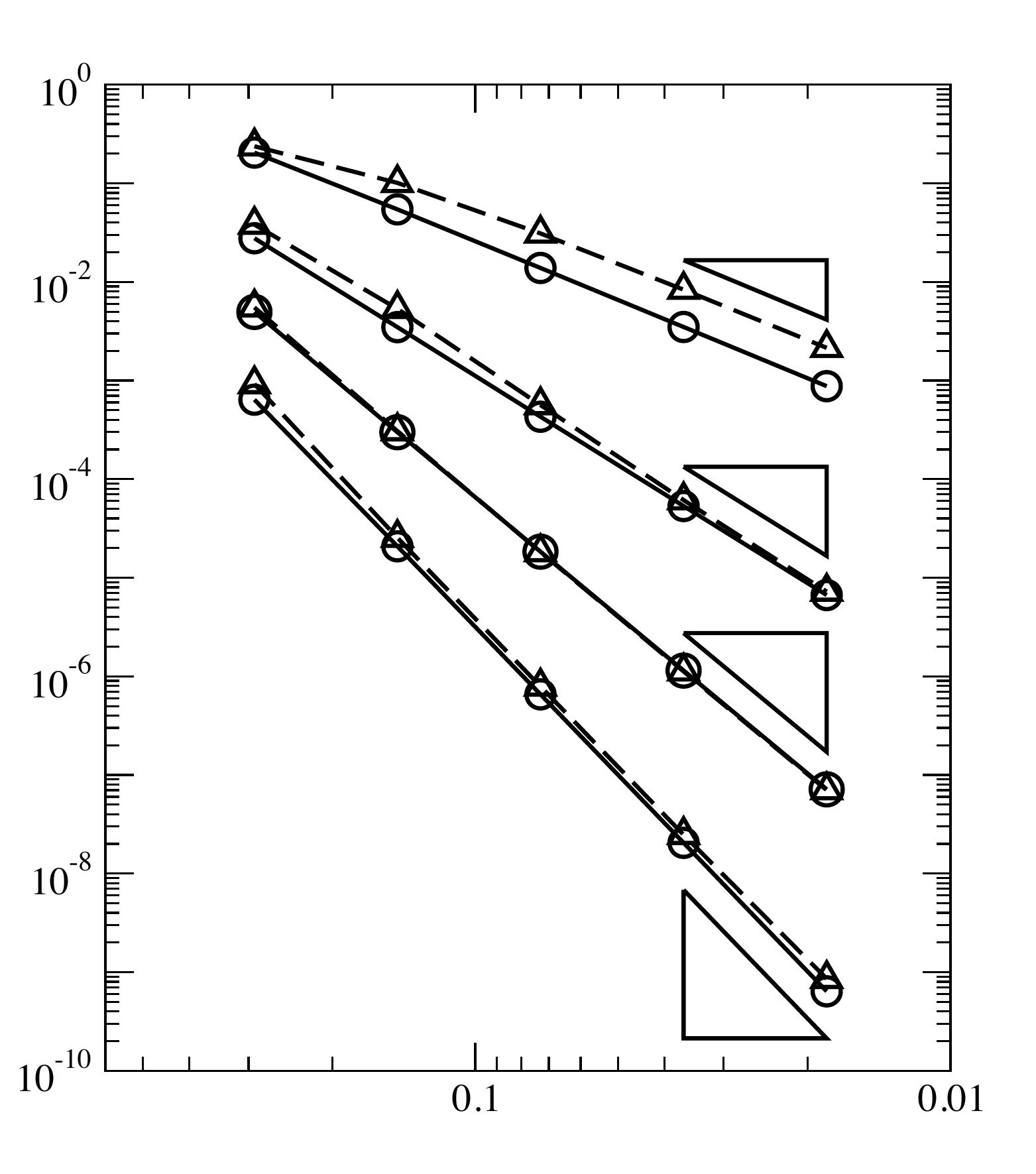}
        \put(72,74) {$\mathbf{2}$}
        \put(72,55) {$\mathbf{3}$}
        \put(72,40) {$\mathbf{4}$}
        \put(54,16) {$\mathbf{5}$}
        \put(28,0){\textbf{Mesh size $h$}}
        \put(-4,18){\begin{sideways}\textbf{$\LTWO$ Approximation errors}\end{sideways}}
      \end{overpic}
      & %%\qquad
      \begin{overpic}[scale=0.375]{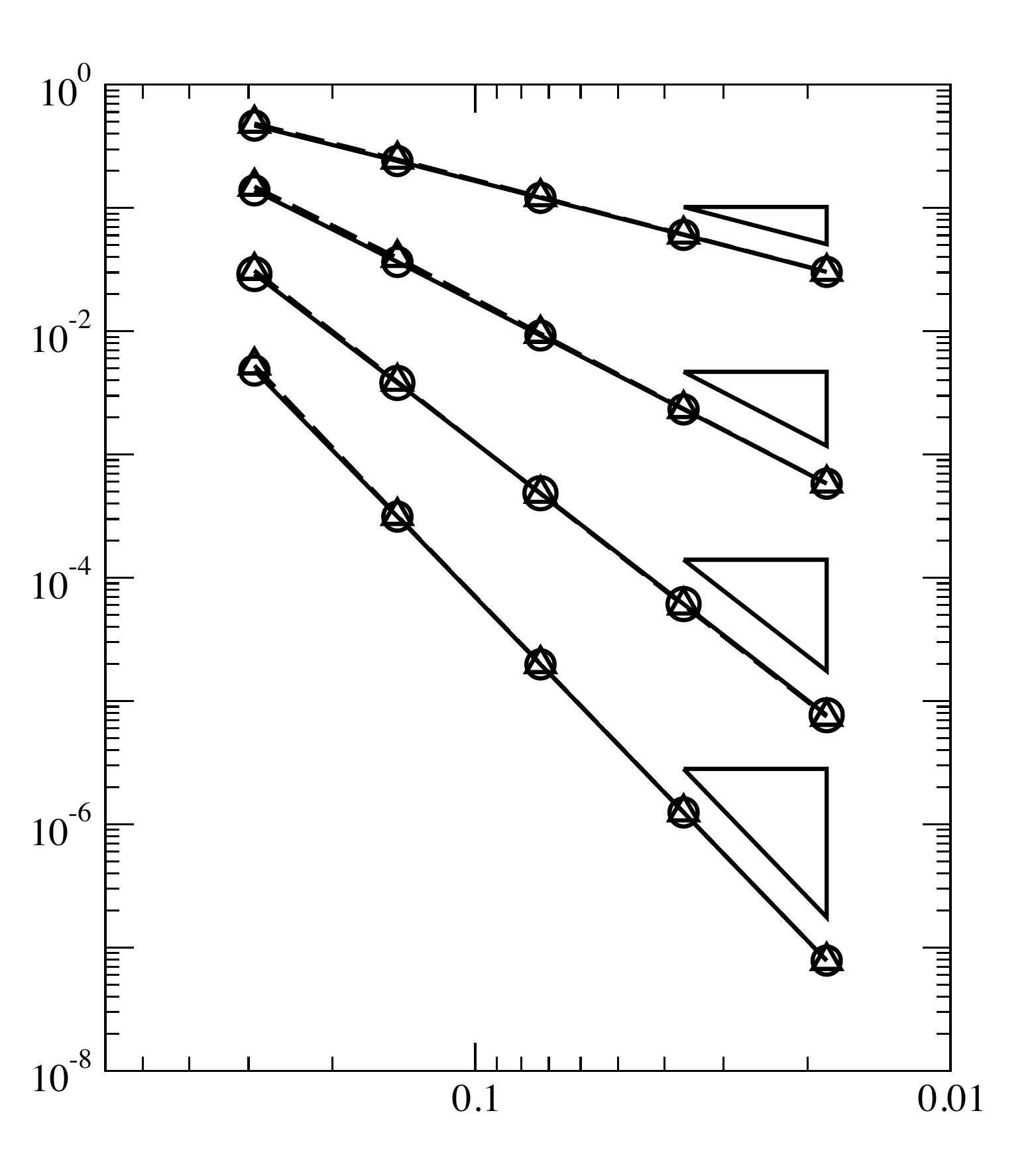} 
        \put(72,79) {$\mathbf{1}$}
        \put(72,64) {$\mathbf{2}$}
        \put(72,46) {$\mathbf{3}$}
        \put(72,27) {$\mathbf{4}$}
        \put(28,0){\textbf{Mesh size $h$}}
        \put(-4,18){\begin{sideways}\textbf{$\HONE$ Approximation errors}\end{sideways}}
      \end{overpic}      
  \end{tabular}
  \caption{
    Error curves for the conforming VEM (circles) and the non-conforming (triangles)
    applied to the mesh family of regular non-convex octagons with $\k=1,2,3$, and $4$.
    The left panels show the relative $\LTWO$ error; 
    the right panels show the relative $\HONE$ errors.
    The expected slopes are indicated by triangles.}
  \label{fig:errors:mesh3}
\end{figure}
The values of the measured error are labeled by a circle for the
conforming VEM and by a square for the nonconforming VEM.
The plots on the left show the relative errors for the approximation
of the solution, while the plots on the right show the relative errors
for the approximation of the solution's gradient.
The expected slopes are shown for each error curve directly on the
plots.
The numerical results confirm the theoretical
rate of convergence. The conforming and nonconforming VEMs provide very close results on any fixed mesh, with the conforming method  slightly over performing the nonconforming VEM in few cases.
Similar results (not shown) are observed when comparing the two methods with respect to the respective number of degrees of freedom. Indeed, for each mesh shown, the difference on the number of degrees of freedom does \emph{not} depend on the polynomial degree $k$ and is about equal to the number of elements, in favour of the conforming VEM.

\section{Conclusion}
\label{sec:conclusion}

We have introduced a unified abstract framework for the Virtual
Element Method, through which conforming and nonconforming VEMs for
solving general second order elliptic convection-reaction-diffusion
problems with non-constant coefficients in two and three dimensions
are defined, analysed, and implemented in a largely identical manner.
We have shown that both methods produce solutions which converge to
the true solution at the optimal rate in the $H^1$- and $\LTWO$-norms,
supported by numerical experiments on a variety of different mesh
topologies including non-convex polygonal elements.

The framework is based on assuming that the $L^2$-projector onto the
polynomial subspace of the virtual element space is computable, in
this respect following the approach
of~\cite{EquivalentProjectors,General}.  By generalising the process
considered in~\cite{EquivalentProjectors}, we have
introduced families of new possible conforming
and nonconforming virtual element spaces in which the
$\LTWO$-projection 
is indeed exactly computable directly from the degrees of
freedom used to describe the space.  From this family we have detailed
a particular space for which the implementation of the $\LTWO$-projection takes a simple form independent of the method, the polynomial degree, and the space dimension.
It also becomes apparent that since the accurate approximation of the problem's data
is only needed to evaluate the polynomial consistency part of the 
bilinear form, the variational crime theory classical of finite element 
methods applies to the virtual element setting. 
Extensions of the present framework to include stabilisation techniques 
for convection-dominated diffusion problems and the design of virtual
element methods for Stokes problems will be considered in future
works.

\vskip3mm
\section*{Acknowledgements}

AC was partially supported by the EPSRC (Grant EP/L022745/1).
GM was partially supported by the Laboratory Directed Research and 
Development program (LDRD), U.S. Department of Energy Office of Science, 
Office of Fusion Energy Sciences, under the auspices of the National 
Nuclear Security Administration of the U.S. Department of Energy by 
Los Alamos National Laboratory, operated by Los Alamos National Security 
LLC under contract DE-AC52-06NA25396.
OS was supported by a Ph.D. Studentship from the College of Science and Engineering at the University of Leicester and an EPSRC Doctoral Training Grant.
All this support is gratefully acknowledged.

\bibliographystyle{acm}
\bibliography{references}

\begin{thebibliography}{10}

\bibitem{EquivalentProjectors}
{\sc Ahmad, B., Alsaedi, A., Brezzi, F., Marini, L.~D., and Russo, A.}
\newblock {Equivalent projectors for virtual element methods}.
\newblock {\em Computers {\&} Mathematics with Applications 66}, 3 (Sept.
  2013), 376--391.

\bibitem{StokesVEM}
{\sc Antonietti, P.~F., Beir{\~a}o~da Veiga, L., Mora, D., and Verani, M.}
\newblock {A stream virtual element formulation of the Stokes problem on
  polygonal meshes}.
\newblock {\em SIAM J. Numer. Anal. 52}, 1 (2014), 386--404.

\bibitem{NonconformingVEM}
{\sc {Ayuso de Dios}, B., {Lipnikov}, K., and {Manzini}, G.}
\newblock {The nonconforming virtual element method}.
\newblock {\em ArXiv:1405.3741v2 e-prints\/} (May 2014).

\bibitem{GFEM}
{\sc Babu{\v{s}}ka, I., and Osborn, J.~E.}
\newblock Generalized finite element methods: their performance and their
  relation to mixed methods.
\newblock {\em SIAM J. Numer. Anal. 20}, 3 (1983), 510--536.

\bibitem{BasicsPaper}
{\sc Beir{\~a}o~da Veiga, L., Brezzi, F., Cangiani, A., Manzini, G., Marini,
  L.~D., and Russo, A.}
\newblock {Basic principles of virtual element methods}.
\newblock {\em Math. Models Methods Appl. Sci.\/} (2013).

\bibitem{LinearElasticity2D}
{\sc Beir{\~a}o~da Veiga, L., Brezzi, F., and Marini, L.~D.}
\newblock {Virtual Elements for Linear Elasticity Problems}.
\newblock {\em SIAM J. Numer. Anal. 51}, 2 (2013), 794--812.

\bibitem{Hitchhikers}
{\sc Beir{\~a}o~da Veiga, L., Brezzi, F., Marini, L.~D., and Russo, A.}
\newblock {The hitchhiker's guide to the virtual element method}.
\newblock {\em Math. Models Methods Appl. Sci. 24}, 8 (2014), 1541--1573.

\bibitem{General}
{\sc Beir{\~a}o~da Veiga, L., Brezzi, F., Marini, L.~D., and Russo, A.}
\newblock {Virtual Element Methods for general second order elliptic problems
  on polygonal meshes}.
\newblock {\em arXiv:1412:2646\/} (Dec 2014).

\bibitem{MFDBook}
{\sc Beir{\~a}o~da Veiga, L., Lipnikov, K., and Manzini, G.}
\newblock {\em The mimetic finite difference method for elliptic problems},
  vol.~11 of {\em Modeling, Simulation and Applications}.
\newblock Springer, Cham, 2014.

\bibitem{ArbitraryRegularityVEM}
{\sc Beir{\~a}o~da Veiga, L., and Manzini, G.}
\newblock {A virtual element method with arbitrary regularity}.
\newblock {\em IMA J Numer Anal (published online)\/} (July 2013).

\bibitem{DiscreteFractureVEM}
{\sc Benedetto, M.~F., Berrone, S., Pieraccini, S., and Scial{\`o}, S.}
\newblock {The virtual element method for discrete fracture network
  simulations}.
\newblock {\em Comput. Methods Appl. Mech. Engrg. 280\/} (2014), 135--156.

\bibitem{Brenner-PF}
{\sc Brenner, S.~C.}
\newblock Poincar\'e-{F}riedrichs inequalities for piecewise {$H^1$} functions.
\newblock {\em SIAM J. Numer. Anal. 41}, 1 (2003), 306--324.

\bibitem{BrennerScott}
{\sc Brenner, S.~C., and Scott, L.~R.}
\newblock {\em The Mathematical Theory of Finite Element Methods}.
\newblock Springer, 2008.

\bibitem{PlateBendingVEM}
{\sc Brezzi, F., and Marini, L.~D.}
\newblock {Virtual element methods for plate bending problems}.
\newblock {\em Comput. Methods Appl. Mech. Engrg. 253\/} (2013), 455--462.

\bibitem{Cangiani-Georgoulis-Houston:2014}
{\sc Cangiani, A., Georgoulis, E.~H., and Houston, P.}
\newblock {$hp$}-version discontinuous {G}alerkin methods on polygonal and
  polyhedral meshes.
\newblock {\em Math. Models Methods Appl. Sci. 24}, 10 (2014), 2009--2041.

\bibitem{Ciarlet}
{\sc Ciarlet, P.~G.}
\newblock {\em {The Finite Element Method for Elliptic Problems}}.
\newblock Elsevier, Burlington, MA, 1978.

\bibitem{Cockburn-Qiu-Solano:2014}
{\sc Cockburn, B., Qiu, W., and Solano, M.}
\newblock A priori error analysis for {HDG} methods using extensions from
  subdomains to achieve boundary conformity.
\newblock {\em Math. Comp. 83}, 286 (2014), 665--699.

\bibitem{CrouzeixRaviart}
{\sc Crouzeix, M., and Raviart, P.~A.}
\newblock {Conforming and nonconforming finite element methods for solving the
  stationary Stokes equations. I}.
\newblock {\em Rev. Francaise Automat. Informat. Recherche Op\'erationnelle
  S\'er. Rouge 7}, R-3 (1973), 33--75.

\bibitem{ScottDupont}
{\sc Dupont, T., and Scott, L.~R.}
\newblock {Polynomial approximation of functions in Sobolev spaces}.
\newblock {\em Math. Comp. 34}, 150 (1980), 441--463.

\bibitem{XFEM}
{\sc Fries, T.-P., and Belytschko, T.}
\newblock The extended/generalized finite element method: an overview of the
  method and its applications.
\newblock {\em Internat. J. Numer. Methods Engrg. 84}, 3 (2010), 253--304.

\bibitem{LinearElasticity3D}
{\sc Gain, A.~L., Talischi, C., and Paulino, G.~H.}
\newblock {On the virtual element method for three-dimensional elasticity
  problems on arbitrary polyhedral meshes}.
\newblock {\em arXiv:1311.0932\/} (Nov 2013).

\bibitem{hackbusch_sauter_cfe_nm}
{\sc Hackbusch, W., and Sauter, S.}
\newblock Composite finite elements for the approximation of {PDE}s on domains
  with complicated micro-structures.
\newblock {\em Numer. Math. 75\/} (1997), 447â--472.

\bibitem{Hughes-Brooks:1979}
{\sc Hughes, T. J.~R., and Brooks, A.}
\newblock A multidimensional upwind scheme with no crosswind diffusion.
\newblock In {\em Finite element methods for convection dominated flows
  ({P}apers, {W}inter {A}nn. {M}eeting {A}mer. {S}oc. {M}ech. {E}ngrs., {N}ew
  {Y}ork, 1979)}, vol.~34 of {\em AMD}. Amer. Soc. Mech. Engrs. (ASME), New
  York, 1979, pp.~19--35.

\bibitem{Kuznetsov-Lipnikov-Shashkov:2004}
{\sc Kuznetsov, Y., Lipnikov, K., and Shashkov, M.}
\newblock The mimetic finite difference method on polygonal meshes for
  diffusion-type problems.
\newblock {\em Computational Geosciences 8}, 4 (2004), 301--324.

\bibitem{MFDReviewJCP}
{\sc Lipnikov, K., Manzini, G., and Shashkov, M.}
\newblock Mimetic finite difference method.
\newblock {\em Journal of Computational Physics 257 -- Part B\/} (2014),
  1163--1227.

\bibitem{MFD-Monotonicity}
{\sc Lipnikov, K., Manzini, G., and Svyatskiy, D.}
\newblock Analysis of the monotonicity conditions in the mimetic finite
  difference method for elliptic problems.
\newblock {\em Journal of Computational Physics 230}, 7 (2011), 2620 -- 2642.

\bibitem{SteklovVEM}
{\sc Mora, D., Rivera, G., and Rodríguez, R.}
\newblock A virtual element method for the steklov eigenvalue problem.
\newblock {\em Mathematical Models and Methods in Applied Sciences online
  ready\/} (2015).

\bibitem{Mu-Lin-Wang-YE:2015}
{\sc Mu, L., Wang, J., and Ye, X.}
\newblock Weak {G}alerkin finite element methods on polytopal meshes.
\newblock {\em Int. J. Numer. Anal. Model. 12}, 1 (2015), 31--53.

\bibitem{PFEM}
{\sc Sukumar, N., and Tabarraei, A.}
\newblock {Conforming polygonal finite elements}.
\newblock {\em Int. J. Numer. Meth. Engng. 61}, 12 (2004), 2045--2066.

\end{thebibliography}

\end{document}